\documentclass{amsart}

\usepackage{amsmath,amsthm,amssymb,amscd}
\usepackage[mathscr]{euscript}
\usepackage{graphicx}
\usepackage{multicol} 
\usepackage{wrapfig} 
\usepackage{tikz}
\usepackage{tikz-cd}
\usepackage{listings}
\usepackage{color}
\usepackage{mathtools}

\usepackage{fancyhdr}
 
\definecolor{codegreen}{rgb}{0,0.6,0}
\definecolor{codegray}{rgb}{0.5,0.5,0.5}
\definecolor{codepurple}{rgb}{0.58,0,0.82}
\definecolor{backcolour}{rgb}{0.95,0.95,0.92}
 
\lstdefinestyle{mystyle}{
    backgroundcolor=\color{backcolour},   
    commentstyle=\color{codegreen},
    keywordstyle=\color{magenta},
    numberstyle=\tiny\color{codegray},
    stringstyle=\color{codepurple},
    basicstyle=\footnotesize,
    breakatwhitespace=false,         
    breaklines=true,                 
    captionpos=b,                    
    keepspaces=true,                 
    numbers=left,                    
    numbersep=5pt,                  
    showspaces=false,                
    showstringspaces=false,
    showtabs=false,                  
    tabsize=2
}
 
\newtheorem{theorem}{Theorem}[section]
\newtheorem{lemma}[theorem]{Lemma}
\newtheorem{proposition}[theorem]{Proposition}
\newtheorem{corollary}[theorem]{Corollary}
\theoremstyle{definition}
\newtheorem{definition}[theorem]{Definition}
\newtheorem{example}{Example}[theorem]
\theoremstyle{remark}

\numberwithin{equation}{section}
\numberwithin{figure}{subsection}

\newcommand{\mon}{\text{Mon}}

\newcommand{\Z}{\mathbb{Z}}
\newcommand{\R}{\mathbb{R}}
\newcommand{\HH}{\mathbb{H}}

\newcommand{\Q}{\mathbb{Q}}

\newcommand{\C}{\mathbb{C}}

\newcommand{\absgal}{\Gal(\overline{\Q}/\Q)}
\newcommand{\Gal}{\text{Gal}}

\newcommand{\vertiii}[1]{{\left\vert\kern-0.25ex\left\vert\kern-0.25ex\left\vert #1 
    \right\vert\kern-0.25ex\right\vert\kern-0.25ex\right\vert}}

\DeclareMathOperator*{\modd}{mod}

\DeclareMathOperator*{\phii}{\varphi}

\DeclareMathOperator*{\aut}{\text{Aut}}
\DeclareMathOperator*{\homeo}{\text{Homeo}}

\begin{document}
\title{Counting the Number of \\ Quasiplatonic Topological Actions\\ of the Cyclic Group on Surfaces}
\author{Charles Camacho}
\address{Kidder Hall 314 \\ Oregon State University \\ Corvallis, OR 97333}
\email{camachoc@math.oregonstate.edu}

\begin{abstract}
Define $QC(n)$ to be the number of quasiplatonic topological actions of the cyclic group $C_n$ on surfaces of genus at least two. We use formulas of Benim and Wootton \cite{ben} to give an explicit formula for $QC(n)$. In addition, we relate the number of quasiplatonic topological actions of $C_n$ to the number of regular dessins d'enfants having $C_n$ as a group of automorphisms.
\end{abstract}

\maketitle

\section{Introduction}

The enumeration of combinatorial maps, topological group actions, and Riemann surfaces are intimately related. The aim of this paper is to derive the total number of distinct quasiplatonic topological group actions of the cyclic group $C_n\cong\langle \rho:\rho^n=1\rangle$ on compact, connected, orientable surfaces of genus at least two. We find that the number of distinct quasiplatonic $C_n$-actions includes a known count on the number of bipartite maps as two-cell embeddings (i.e., hypermaps) with $C_n$ as its group of symmetries. 

Our explicit formula on enumerating quasiplatonic actions on surfaces makes a connection to the number of \emph{(regular) dessins d'enfants} on these surfaces. The modern theory of dessins was originally introduced by A. Grothendieck in \cite{groth}. Dessins are interesting to Galois theory and algebraic geometry thanks to the following incredible fact due to Bely{\u\i}'s Theorem \cite{belyi}: compact, connected surfaces admitting an embedded dessin are in one-to-one correspondence with complex projective curves defined over a number field. A recent result of G. Gonz{\'a}lez-Diez and A. Jaikin-Zapirain proves that the absolute Galois group $\absgal$ acts faithfully on the set of regular dessins and their underlying (quasiplatonic) surfaces \cite{gonz}. Moreover, Grothendieck observed that the existence of a dessin endows the surface with a unique complex structure \cite{groth}. Thus, studying the absolute Galois group amounts to understanding its action on regular dessins and quasiplatonic surfaces.

On the other hand, there is a rich history of group actions on surfaces. Certain group actions, namely quasiplatonic topological group actions, turn out to be well-suited for the language of dessins.

There are many known results about cyclic group actions on surfaces. W. Harvey in \cite{harvey} determined necessary and sufficient conditions for $C_n$ to act topologically on a surface of genus $\sigma\geq 2$, which then provides admissible \emph{signatures} of a surface guaranteeing such a $C_n$-action. K. Lloyd in \cite{lloyd1972} gives the number of topological actions of a given finite group $G$ on a surface $X$ with a fixed Fuchsian group $\Delta$ containing a torsion-free surface group $\Gamma$ uniformizing $X$. In particular, Lloyd obtains a generating function for the number of actions when $G=C_p$ for an odd prime $p$ (see Section \ref{subsec:lloyd} in this paper). Lloyd also outlines when topological actions are equivalent in terms of surface kernel epimorphisms (Theorem 4, \cite{lloyd1972}) based on results of H. Zieschang \cite{zieschang}. A. Broughton in \cite{broughton} describes these topological group action equivalences in terms of \emph{generating vectors} for $G$ (Proposition 2.2, \cite{broughton}). We can then connect topological actions to \emph{conformal actions} thanks to a result by A. Wootton (Theorem 3.6, \cite{woottextend}): a topological action of a finite group $G$ can uniquely extend to a conformal action on a surface if and only if $G$ is a quotient of a (hyperbolic) triangle group. In our case, the automorphism group of a dessin is indeed given by a quotient of a triangle group. Thus, quasiplatonic topological group actions can be described visually by a regular dessin d'enfant.

Our computation makes extensive use of the explicit formulas of R. Benim and A. Wootton \cite{ben}, which enumerate the distinct quasiplatonic topological actions of $C_n$ on surfaces of a given admissible signature. Those enumeration results build upon the counting formulas of A. Wootton (Proposition 4.4, \cite{woottextend}), which are derived by analyzing the equivalence classes of generating vectors of $C_n$. Moreover, the counting functions of \cite{ben} (restated in Theorem \ref{thm:ben} below) are dependent on how $C_n$ acts topologically on a surface, which is related to the genus of the surface by the Riemann-Hurwitz formula. The novelty of our enumeration formula is that it will be independent of the genus of the surface, and in particular, it will give a measure of the prevalence of possible quasiplatonic topological $C_n$-actions.

We count the distinct quasiplatonic topological actions of $C_n$ on surfaces by the following procedure:
\begin{enumerate}
\item find all admissible signatures for a given $n$ satisfying Harvey's Theorem (Theorem 2, \cite{harvey}; or see Theorem \ref{thm:harvey} below);
\item use one of three different formulas from \cite{ben} giving the number of distinct topological actions of $C_n$ on quasiplatonic surfaces, based on three different signature types;
\item compute the sum of the values given by the formulas in \cite{ben} over all possible signatures for $n$. This number will be denoted as $QC(n)$.
\end{enumerate}


There are several motivating reasons for our work.
\begin{itemize}
\item \emph{Mapping Class Group} It is of great interest to understand the structure of the group $\Gamma_\sigma$ of orientation-preserving self-homeomorphisms of a surface $X$ of genus $\sigma$ up to isotopy, called the \emph{mapping class group}. In particular, since we consider surfaces having a complex structure, $\Gamma_\sigma$ will be the group of outer automorphisms of the fundamental group of a Riemann surface of genus $\sigma$ which preserve orientation \cite{broughton}. We assume $\sigma\geq 2$, so that $\aut(X)$ is a finite group of size at most $84(\sigma-1)$ by Hurwitz' Theorem (e.g., \cite{farkas}). Then equivalence classes of finite group conformal actions on surfaces are in one-to-one correspondence with conjugacy classes of finite subgroups of $\Gamma_\sigma$ (Section 2.1, \cite{broughton}). Thus, enumeration formulas for the number of group actions provide, for instance, lower bounds on the number of conjugacy classes of maximal finite subgroups of $\Gamma_\sigma$. Our counting formula in this paper can give an overall size of the number of conjugacy classes of cyclic subgroups of $\Gamma_\sigma$ as $\sigma$ ranges over all genera of surfaces admitting a $C_n$ action.

\item\emph{Moduli Space of Riemann Surfaces} Suppose the genus $\sigma$ of a surface $X$ is at least two, so that $\aut(X)$ is a finite group. Then $\aut(X)$ determines a conjugacy class of subgroups of the mapping class group, $\Gamma_\sigma$. Conversely, the surfaces whose full automorphism group corresponds to a fixed conjugacy class of subgroups of $\Gamma_\sigma$ stratifies the \emph{moduli space} $M_\sigma$ of Riemann surfaces of genus $\sigma$ into a finite, disjoint union of smooth subvarieties (Theorem 2.1, \cite{broughton}). In this way, studying group actions on Riemann surfaces may help better understand the singularities of $M_\sigma$. 


\item \emph{Quasiplatonic Surfaces} A compact Riemann surface with a regular ramified covering of the Riemann sphere over three branch values is called \emph{quasiplatonic}. A compact Riemann surface is quasiplatonic if and only if it admits a regular dessin d'enfant. The quasiplatonic surfaces are the most symmetric of the compact Riemann surfaces, containing the Hurwitz curves whose full automorphism group reaches the upper bound of $84(\sigma-1)$ (Section 5.1.2, \cite{jones}). It turns out that there are only finitely many quasiplatonic surfaces in each genus $\sigma\geq 2$, which begs the question posed by J. Wolfart in \cite{wolf}: how many quasiplatonic surfaces are there? J. Schalge-Puchta and Wolfart derive asymptotic bounds for the number of quasiplatonic surfaces of genus $\sigma$, on the order of $\sigma^{\log\sigma}$ (\cite{puchta}). We hope that our enumeration formula on quasiplatonic cyclic actions, along with the extension results of \cite{woottextend}, can determine a lower bound on the number of quasiplatonic surfaces. Indeed, for cyclic prime quasiplatonic actions, Riera and Rodriguez (Theorem 1, \cite{riera}) show that distinct actions correspond to the distinct Lefschetz curves. Moreover, quasiplatonic surfaces, and their higher dimensional analogues called Beauville surfaces, play a key role in studying the absolute Galois group \cite{gonz}.
\item \emph{Dessins d'Enfants} A combinatorial map can be used to describe a quasiplatonic action of a finite group $G$ on a surface $X$. Suppose $G$ acts conformally on $X$, i.e., $G$ injects into $\aut(X)$, and suppose that the orbit space $X/G$ is conformally equivalent to the Riemann sphere with ramified covering $\beta:X\rightarrow X/G$ branching over three values. Then there is an associated regular dessin $D:=\beta^{-1}([0,1])$ whose automorphism group is isomorphic to $G$. We are interested in enumerative properties of quasiplatonic surfaces as they relate to the combinatorics of a dessin.
\end{itemize}

\subsection{Structure of the paper}

The paper will be structured as follows. Section \ref{sec:results} will state the main results of the paper. Section \ref{sec:prelim} will provide preliminary background on Riemann surfaces, dessins d'enfants, quasiplatonic surfaces, and quasiplatonic topological group actions. Section \ref{sec:lem} contains various lemmas and definitions. Sections \ref{sec:qceven} and \ref{sec:qcodd} will be devoted to proving the main results. Section \ref{sec:conseq} will discuss the consequences of deriving the explicit form of $QC(n)$, as well as some key examples. Section \ref{sec:conc} concludes the paper with future directions for research.

\section{Results}\label{sec:results}

Throughout, we write the positive integer $n$ as $n=\prod_{i=1}^rp_i^{a_i}$ for distinct primes $p_i$ and positive integers $a_i$.

\begin{theorem}\label{thm:qceven}
Suppose $n$ is even and $n\geq 8$, so that $p_1=2$. Then the number of distinct topological actions $QC(n)$ of $C_n$ on quasiplatonic surfaces is
\begin{equation*}
QC(2^{a_1}p_2^{a_2}\cdots p_r^{a_r})=2^{a_1-2}\left(\prod_{i=2}^rp_i^{a_i-1}(p_i+1)\right)-1+\left\lbrace\begin{array}{lr}
 2^{r-2} & a_1=1\\
 2^{r-1} & a_1=2\\
 2^{r} & a_1\geq 3
\end{array}\right..
\end{equation*}
\end{theorem}

\begin{theorem}\label{thm:qcodd}
Suppose $n$ is odd, with either $n\geq 5$. Then the number of distinct topological actions $QC(n)$ of $C_n$ on quasiplatonic surfaces is

\begin{align*}
QC(p_1^{a_1}\cdots p_r^{a_r})&=\frac{1}{6}\left(\prod_{i=1}^r p_i^{a_i-1}(p_i+1)\right)-1\\
&\qquad +\left\lbrace\begin{array}{lr}
 2^{r-1} & \quad p_1=3, a_1\geq 2, p_i\equiv 1\modd 6 \text{ for } 2\leq i\leq r;\\
  &\quad\text{ or } p_i\equiv 5\modd 6\text{ for some } i\\
   \\
\frac{4}{3}\cdot 2^{r-1} & \quad p_1=3, a_1=1, p_i\equiv 1\modd 6\text{ for } 2\leq i\leq r\\
   \\
\frac{5}{3}\cdot 2^{r-1} & p_i\equiv 1\modd 6\text{ for all } i
\end{array}\right..
\end{align*}
\end{theorem}

The formulas for $QC(n)$ above are related to the number of regular dessins $D$. Let $r(C_n)$ be the number of regular dessins, up to isomorphism, with $\aut(D)\cong C_n$. G. Jones (Example 3.1, \cite{jones1}) determined an explicit formula for $r(C_n)$:
\begin{equation*}
r(C_n)=n\prod_{p\mid n}\left(1+\frac{1}{p}\right).
\end{equation*}
This allows us to rewrite the formulas for $QC(n)$ as follows.

\begin{corollary}
For even $n\geq 8$,
\begin{equation*}
QC(n)-\frac{1}{6}r(C_n)=-1+\left\lbrace\begin{array}{lr}
 2^{r-2} & a_1=1\\
 2^{r-1} & a_1=2\\
 2^{r} & a_1\geq 3.
 \end{array}\right..
\end{equation*}
For odd $n\geq 5$,
\begin{equation*}
QC(n)-\frac{1}{6}r(C_n)=-1+\left\lbrace\begin{array}{lr}
 2^{r-1} & \quad p_1=3, a_1\geq 2, p_i\equiv 1\modd 6 \text{ for } 2\leq i\leq r;\\
  &\quad\text{ or } p_i\equiv 5\modd 6\text{ for some } i\\
   \\
\frac{4}{3}\cdot 2^{r-1} & \quad p_1=3, a_1=1, p_i\equiv 1\modd 6\text{ for } 2\leq i\leq r\\
   \\
\frac{5}{3}\cdot 2^{r-1} & p_i\equiv 1\modd 6\text{ for all } i
\end{array}\right..
\end{equation*}
\end{corollary}



We include here $QC(p^a)$ for later use.

\begin{corollary}\label{cor:primepower}
For any positive integer $a$ and any prime number $p$,
\begin{equation*}
QC(p^a)=\left\lbrace\begin{array}{lr}
 2^{a-2}+1 & p=2, a\geq 3\\
 2\cdot 3^{a-2} & p=3, a\geq 2\\
 \frac{1}{6}p^{a-1}(p+1) & p\equiv 5\modd 6\\
 \frac{1}{6}p^{a-1}(p+1)+\frac{2}{3} & p\equiv 1\modd 6.
\end{array}\right..
\end{equation*}
\end{corollary}

The above formulas can be framed in the following way. For $n=\prod_{i=1}^r p_i^{a_i}$ with $p_i$ distinct primes and $a_i$ positive integers,
\begin{equation*}
QC(n)=\frac{1}{6}n\prod_{p\mid n}\left(1+\frac{1}{p}\right)-1+a\cdot 2^r
\end{equation*}
where
\begin{equation*}
a=\left\lbrace\begin{array}{lr}
 \frac{1}{4} & p_1=2, a_1=1\\
 & \\
 \frac{1}{2} & p_1=2, a_1=2\\
 & \\
 1 & p_1=2, a_1\geq 3\\
  & \\
 \frac{5}{6} & p_i\equiv 1\modd 6 \text{ for } 1\leq i\leq r\\
  & \\
 \frac{2}{3} & p_1=3, a_1=1, p_i\equiv 1\modd 6 \text{ for } 2\leq i\leq r\\
 & \\
 \frac{1}{2} & p_1=3, a_1\geq 2, p_i\equiv 1\modd 6 \text{ for } 2\leq i\leq r\\
 & \\
 \frac{1}{2} & p_i\equiv 5\modd 6 \text{ for some } 1\leq i\leq r
\end{array}\right..
\end{equation*}

In other words, the number of distinct quasiplatonic topological actions of $C_n$ on surfaces (without necessarily fixing the genus nor the signature of the action) is, up to a constant, one-sixth the number of regular dessins with $C_n$ as their automorphism group.

\section{Preliminaries}\label{sec:prelim}

This section summarizes important facts about Riemann surfaces, quasiplatonic surfaces, topological group actions, and dessins d'enfants. We will also describe formulas from Benim and Wootton found in \cite{ben}. 

By a \emph{surface} we mean a real, connected 2-manifold. That is, a surface is a Hausdorff topological space with a countable basis, locally homeomorphic to $\R^2$. We will assume throughout the paper that our surfaces are connected, compact, and orientable. Moreover, a \emph{Riemann surface} is a surface with a complex structure, i.e., transition functions on the surface are holomorphic functions on the complex plane, $\C$. We will also use the classification of compact Riemann surfaces, called the Uniformization Theorem \cite{farkas}: all Riemann surfaces have as universal covering space either the sphere, the complex plane, or the open unit disc in the complex plane, depending on whether the genus of the surface is zero, one, or at least two. Since the open unit disc and the upper-half plane are conformally equivalent, we will use either model for the hyperbolic plane when discussing higher genus surfaces.

\subsection{Topological Group Actions}

Let $G$ be a finite group. A \emph{topological group action} of $G$ on a surface $X$ is an injective group homomorphism
\begin{equation*}
\epsilon:G\xhookrightarrow{}\text{Homeo}^+(X)
\end{equation*}
into the group of orientation-preserving self-homeomorphisms $\text{Homeo}^+(X)$ of $X$. We may also refer to $G$ acting on $X$ as a \emph{topological $G$-action}. There is the associated group action
\begin{align*}
\epsilon_G:G\times X&\longrightarrow X\\
(g,x)&\longmapsto\epsilon(g)(x)
\end{align*}
Another action $\epsilon':G'\xhookrightarrow{}\homeo^+(Y)$ will be equivalent to $\epsilon$ if there exists a group isomorphism $\omega:G\rightarrow G'$ and a homeomorphism $h:X\rightarrow Y$ such that the following diagram commutes:
\begin{equation*}
\begin{tikzcd}
 G\times X \arrow[r, "\epsilon_G"] \arrow[d, "\omega", shift right=2.5ex] \arrow[d, "h", shift left=2.5ex] & X \arrow[d, "h"]\\
 G'\times Y \arrow[r, "\epsilon_{G'}"] &Y
\end{tikzcd}
\end{equation*}
That is, for any $g\in G, x\in X$,
\begin{equation*}
h\circ\epsilon(g)\circ h^{-1}(x)=\epsilon'(w(g))(x).
\end{equation*}
If $X$ is a Riemann surface, we say $G$ \emph{acts conformally on $X$} if there is an injective group homomorphism $G\xhookrightarrow{}\aut(X)$. Here, $\aut(X)$ denotes the group of conformal self-maps on $X$ with group operation given by composition of functions.

\begin{example}
Consider the three-holed compact surface $X$ in Figure \ref{fig:threehole}.
\begin{figure}[ht]
\includegraphics[page=2,scale=.35]{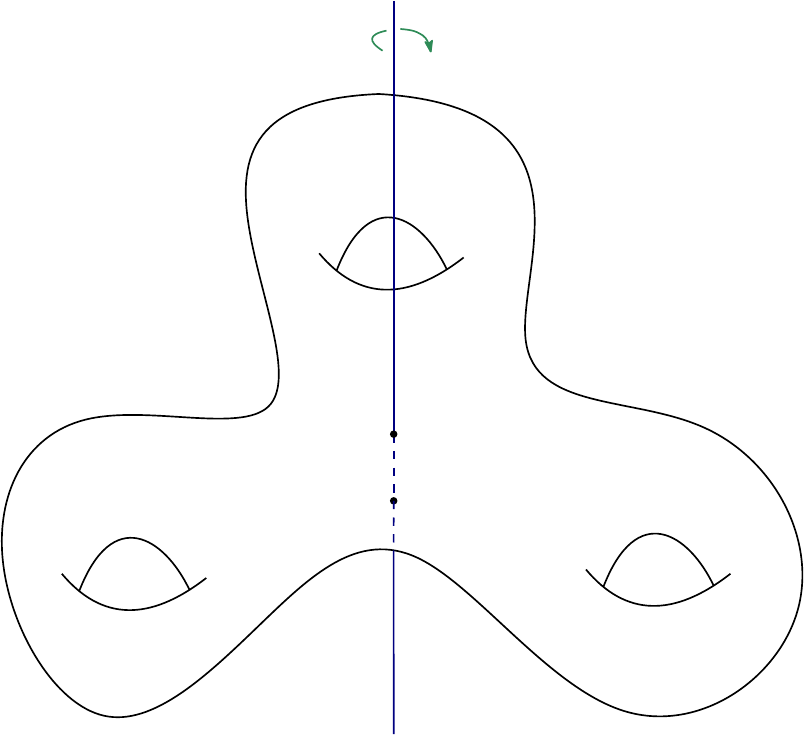}
\caption{Three-holed surface with $C_3$ rotational symmetry.}\label{fig:threehole}
\end{figure}
$X$ has three-fold rotational symmetry about the center vertical axis, i.e., there is a group action $\epsilon:C_3\xhookrightarrow{}\homeo^+(X)$. Every point on $X$, except for the two fixed points intercepting the vertical axis, has an orbit size of three under the action of $C_3$. Each hole of $X$ is mapped to each other. The resulting orbit space $X/C_3$ is a genus one surface with two branch points, see Figure \ref{fig:orbitspace} below.
\begin{figure}[ht]
\includegraphics[page=3,scale=.4]{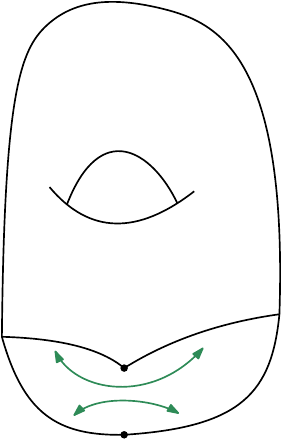}\qquad \includegraphics[page=4,scale=.4]{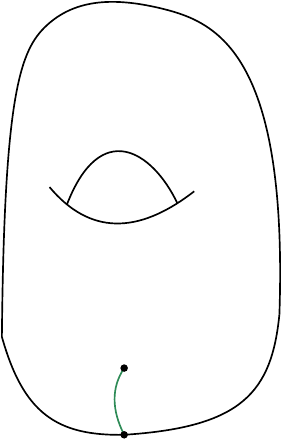}
\caption{The orbit space $X/C_3$ is a one-holed torus with two branch points. The left-hand figure shows the process of creating the orbit space, and the right-hand figure is the final surface.}\label{fig:orbitspace}
\end{figure}

\end{example}

\begin{example}
For a Riemann surface $X$ of genus $\sigma\geq 2$, it is a famous fact of Hurwitz that $|\aut(X)|\leq 84(\sigma-1)$. Among all genus three compact Riemann surfaces, there is a unique surface with the largest number of automorphisms, called the Klein quartic curve. It admits a $C_7$ conformal action, as well as a $\text{PSL}_2(7)$ conformal action.
\end{example}

Two topological $G$-actions on the same surface $X$ are equivalent if their images in $\homeo^+(X)$ are conjugate. The next section will describe why there are only finitely many $G$-actions on surfaces.

\subsection{Surface Kernel Epimorphisms and Generating Vectors}

One way to understand topological actions of a finite group $G$ on a surface is by realizing $G$ as a quotient group of automorphisms (i.e., orientation-preserving isometries) of the hyperbolic plane $\HH$ by a certain torsion-free, normal subgroup. 

A \emph{Fuchsian group} is a discrete subgroup of $\text{PSL}_2(\R)\cong\aut(\HH)$. Here, $\text{PSL}_2(\R)$ is given the subspace topology, being a subset of the two-by-two matrices with real entries with its usual topology. The Uniformization Theorem characterizes all Riemann surfaces of genus at least two as quotients of $\HH$ by the action of a torsion-free Fuchsian group. 

The following theorem gives a geometric interpretation of a topological $G$-action. The statement can be found in \cite{harvey}.
\begin{theorem}\label{thm:topactionfuchs}
A finite group $G$ acts on a (compact) surface $X$ of genus $\sigma\geq 2$ if and only if the $G$-action is topologically equivalent to the action of $\Delta/\Gamma$ on $X$, where $X\cong\HH/\Gamma$ and $\Gamma$ is a torsion-free Fuchsian group normally contained in some Fuchsian group $\Delta$ with $\Delta/\Gamma\cong G$.
\end{theorem}
Here, $\Gamma\cong\pi_1(X)$, and so $\Gamma$ is referred to as a \emph{surface group}. The overgroup $\Delta$ in the previous theorem is thus determined by the $G$-action. Let $g$ be the genus of the orbit space $X/G$. An application of the Siefert Van-Kampen Theorem gives a presentation for $\Delta$ as
\begin{equation}\label{eq:presentation}
\Delta\cong\left\langle a_1,b_1,\ldots,a_g,b_g, c_1,\ldots,c_k\mid c_1^{n_1}=\cdots=c_k^{n_k}=1, \prod_{i=1}^g[a_i,b_i]\cdot\prod_{i=1}^kc_i=1\right\rangle,
\end{equation}
where $[x,y]:=xyx^{-1}y^{-1}$ denotes the commutator. The Fuchsian group $\Delta$ is completely described by the integers $g$ and $n_1,\ldots, n_k$. Thus, we say that $\Delta$ has \emph{signature} $(g;n_1,\ldots,n_k)$ and that $X$ has signature $(n_1,\ldots,n_k)$. When $g=0$ and $k=3$, $\Delta$ is a \emph{triangle group}, written as
\begin{equation*}
\Delta(n_1,n_2,n_3)=\langle x,y,z\mid x^{n_1}=y^{n_2}=z^{n_3}=1,xyz=1\rangle.
\end{equation*}

By Theorem \ref{thm:topactionfuchs}, a topological $G$-action gives rise to a surjective group homomorphism $\phii:\Delta\twoheadrightarrow G$ with a presentation for $\Delta$ given as above and $\ker\phii\cong\pi_1(X)$ torsion-free. These epimorphisms are of great interest, and are called \emph{surface kernel epimorphisms}. In the special case when $g=0$ and $k=3$, we call such a $G$-action a \emph{quasiplatonic topological action}.

The map $X\rightarrow X/G$ from a surface $X$ of genus $\sigma$ is a ramified covering of the orbit space $X/G$ of genus $g$. Except for a finite number of points on $X$, i.e., the branch points of the $G$-action, the orbits under the $G$-action are all of size $|G|$. By considering the number of vertices, edges, and faces of a triangulation whose vertices include all the branch points on $X$, one can derive the Riemann-Hurwitz formula (e.g., Proposition 1.2, \cite{jones}).
\begin{theorem}[Riemann-Hurwitz Formula]
Let $f:X\rightarrow X/G$ be the projection map associated to a topological $G$-action on $X$. Write $X\cong\HH/\Gamma$ as in Theorem \ref{thm:topactionfuchs}, with $\Gamma\lhd\Delta$ and $\Delta$ having presentation as in \eqref{eq:presentation}. Then
\begin{equation*}
2\sigma-2=|G|(2g-2)+|G|\cdot\sum_{i=1}^k\left(1-\frac{1}{n_i}\right).
\end{equation*}
\end{theorem}



The notion of equivalent topological $G$-actions on $X$ translates to equivalences between the associated surface kernel epimorphisms in the following way (for proofs, see the discussion after Proposition 2.1 in \cite{broughton}, or Theorem 4 in \cite{lloyd1972}). The result for quasiplatonic actions is found as Theorem 2.3 in \cite{woottextend}. For the general case of a non-quasiplatonic topological action, see \cite{broughton}.

\begin{theorem}[Wootton, \cite{woottextend}]\label{thm:kernelepiequiv}
Two quasiplatonic topological $G$-actions $\epsilon$ and $\epsilon'$ on $X$ are equivalent if and only if there exists $\omega\in\aut(G)$ and $\gamma\in\aut(\Delta)$ such that the associated surface kernel epimorphisms $\phii$ and $\phii'$ of $\epsilon$ and $\epsilon'$, respectively, satisfy
\begin{equation*}
\varphi'=\omega\circ\phii\circ\gamma^{-1}.
\end{equation*}
\end{theorem}


Thus, the problem of counting topological $G$-actions up to equivalence is reduced to counting surface kernel epimorphisms of $G$. We can further describe the surface kernel epimorphisms by working only with $G$ itself, via generating vectors for $G$. A \emph{$(g;n_1,n_2,\ldots,n_k)$-generating vector} for $G$ is a tuple 
\begin{equation*}
(\alpha_1,\beta_1,\ldots,\alpha_g,\beta_g,\eta_1,\ldots,\eta_k)
\end{equation*}
of elements in $G$ that satisfy the following:
\begin{enumerate}
\item $\displaystyle\prod_{i=1}^g[\alpha_i,\beta_i]\cdot\prod_{i=1}^k\eta_i=1$;
\item $|\eta_i|=n_i$, where $|\cdot|$ denotes order;
\item $G=\langle\alpha_1,\beta_1,\ldots,\alpha_g,\beta_g,\eta_1,\ldots,\eta_k\rangle$.
\end{enumerate}

Suppose $G$ acts on $X$ with associated surface kernel epimorphism $\phii:\Delta\rightarrow G$. Then setting
\begin{align*}
\alpha_i&=\phii(a_i)\\
\beta_i&=\phii(b_i)\\
\eta_i&=\phii(c_i),
\end{align*}
yields a $(g;n_1,\ldots,n_k)$-generating vector for $G$. Conversely, if we begin with a $(g;n_1,\ldots,n_k)$-generating vector for $G$, we can build the Fuchsian group $\Delta$ so that $\phii:\Delta\twoheadrightarrow G$ is a surface kernel epimorphism. Equivalences among the surface kernel epimorphisms also translates to equivalences between generating vectors: we say two generating vectors $\nu$ and $\nu'$ for $G$ are equivalent if their associated surface kernel epimorphisms are equivalent in the sense of Theorem \ref{thm:kernelepiequiv}. Therefore, understanding distinct topological $G$-actions on a surface reduces to working exclusively with generating vectors for $G$.





If $G\cong\Delta/\Gamma$ acts topologically on $X$ where $\Delta$ is some hyperbolic triangle group $\Delta(n_1,n_2,n_3)$, then the action of $G$ on $X$ can also be described combinatorially by a bipartite map on the surface of $X$. These maps are described in the next sections.

\subsection{Dessins d'Enfants}

A \emph{dessin d'enfant} (or \emph{dessin} for short) is a pair $(X,D)$, where $X$ is a compact, connected surface and $D\subset X$ is a finite map such that
\begin{enumerate}
\item $D$ is connected.
\item $D$ is bicolored (i.e., bipartite).
\item $X\setminus D$ is the union of finitely many topological discs, called the \emph{faces} of $D$.
\end{enumerate}
We denote the two colors of the vertices of a dessin as white and black. Two dessins $(X,D)$ and $(Y,D')$ will be equivalent if there exists an orientation-preserving homeomorphism $f:X\rightarrow Y$ whose restriction $f\mid_D$ is isomorphic to $D'$ as bicolored graphs. See \cite{gir}, \cite{jones}, or \cite{wolf} for a thorough treatment of dessins d'enfants.



Equivalence between dessins can also be seen combinatorially by the use of permutations describing their symmetries. Because a dessin $D$ lies on a Riemann surface $X$ which is orientable, there is an induced ordering on the edges around the white and black vertices, and faces of $D$. Label the edges of the dessin $1,\ldots,N$. We will use the clockwise orientation on surfaces throughout the paper. Let $\sigma_0$ and $\sigma_1$ be the permutations in $S_N$, the permutation group on $N$ letters, describing the order of numbered edges about the white and black vertices of $D$, respectively. The pair $(\sigma_0,\sigma_1)$ is called the \emph{permutation representation pair} of $D$. The group $\langle\sigma_0,\sigma_1\rangle$ is a transitive permutation subgroup of $S_N$, and is called the \emph{monodromy group} $\mon(D)$ of $D$. The monodromy group is unique up to conjugation in $S_N$, due to renumbering of the edges of $D$.

The full set of symmetries of a dessin $D$ on $X$ is its \emph{automorphism group}, $\aut(D)$, defined as the group of orientation-preserving homeomorphisms $f:X\rightarrow X$ which permutes the edges of $D$ while preserving the cyclic order of edges around each vertex. Equivalently, the induced permutation on edges arising from the function $f$ commutes with $\sigma_0$ and $\sigma_1$. This means the automorphism group is the centralizer in $S_N$ of the monodromy group (Theorem 4.40, \cite{gir}):
\begin{equation*}
\aut(D)\cong Z(\mon(D)):=\lbrace \tau\in S_N:\tau\sigma=\sigma\tau, \sigma\in\mon(D)\rbrace.
\end{equation*}
A \emph{regular dessin} is a dessin $D$ such that $\aut(D)$ acts transitively on the edge set of $D$. In particular, $\aut(D)\cong\mon(D)$. Moreover, it turns out that $\aut(D)\cong\Delta(n_1,n_2,n_3)/\Gamma$ for some hyperbolic triangle group $\Delta(n_1,n_2,n_3)$ (e.g., \cite{jones1}). In fact, $n_1=|\sigma_0|, n_2=|\sigma_1|$ and $n_3=|\sigma_0\sigma_1|$, which are the orders of the white, black, and faces, respectively. We call $(n_1,n_2,n_3)$ the \emph{type} of the dessin.

\subsection{Regular Dessins and Surface Kernel Epimorphisms}

The combinatorial structure of a dessin is connected to the language of surface kernel epimorphisms and generating vectors of its automorphism group. Let $D$ be a regular dessin on $X\cong\HH/\Gamma$. Suppose $G:=\aut(D)$ is given as a quotient group $\Delta(n_1,n_2,n_3)/\Gamma$ for some hyperbolic triangle group $\Delta(n_1,n_2,n_3)$. Then $G\leq \aut(X)$ acts conformally on $X$. Moreover, there is the associated surface kernel epimorphism $\phii$ defined on the generators of $\Delta(n_1,n_2,n_3)$ as
\begin{align*}
\phii:\Delta(n_1,n_2,n_3)&\longrightarrow G\\
x&\longmapsto\sigma_0\\
y&\longmapsto\sigma_1\\
z&\longmapsto(\sigma_0\sigma_1)^{-1}.
\end{align*}
In this way, a combinatorial object on a compact Riemann surface of genus at least two encodes geometric information in the form of a generating vector for its symmetry group.





\subsection{Quasiplatonic Surfaces}

If $X$ has genus $\sigma\geq 2$, then $\aut(X)$ is a finite group of size at most $84(\sigma-1)$, by Hurwitz's Automorphism Theorem. The Uniformization Theorem then gives $X\cong\HH/\Gamma$ for a Fuchsian group $\Gamma$. We say $X$ is \emph{quasiplatonic} if $\Gamma\lhd\Delta(n_1,n_2,n_3)$ for some hyperbolic triangle group $\Delta(n_1,n_2,n_3)$. See Figure \ref{fig:quasiplatonicgenus3} for an example of a quasiplatonic surface.

The following theorem gives various ways to describe a quasiplatonic surface. We use $N(\cdot)$ to denote the normalizer in $\text{PSL}_2(\R)$.

\begin{theorem}[Jones and Wolfart \cite{jones}, Girondo and Gonz{\'a}lez-Diez \cite{gir}]
Let $X\cong\HH/\Gamma$ be a compact Riemann surface of genus at least two with surface group $\Gamma$. The following are equivalent.
\begin{enumerate}
\item $X$ is quasiplatonic.
\item $X$ admits a regular dessin d'enfant.
\item $X$ admits a Bely{\u\i} function $\beta:X\rightarrow\widehat{\C}$ which is a regular (ramified) covering branching over three values.
\item $N(\Gamma)$ is a triangle group.
\end{enumerate}
\end{theorem}

\begin{figure}
\includegraphics[page=8,scale=.5]{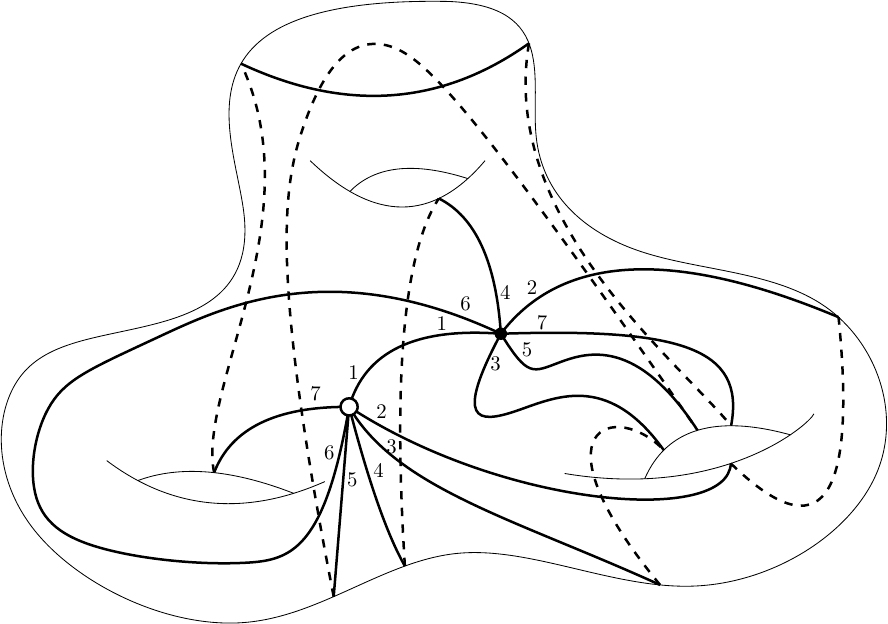}
\caption{A regular dessin d'enfant with $C_7$ as its group of automorphisms on a quasiplatonic surface of genus three.}
\label{fig:quasiplatonicgenus3}
\end{figure}



Let $D$ be a regular dessin on a quasiplatonic surface $X\cong\HH/\Gamma$ with $\aut(D)\cong G$. Then $G\cong\Delta(n_1,n_2,n_3)/\Gamma$ for some triangle group $\Delta(n_1,n_2,n_3)$. That gives an associated conformal $G$-action on the surface with orbit space $X/G$ homeomorphic to a sphere branched over three points. The map $X\rightarrow X/G$ is the Bely{\u\i} function of $D$. On the other hand, if we are given a topological $G$-action on the surface with a genus-zero orbit space branching over three points, then a result of Wootton says that it must be the conformal one given by this dessin \cite{woottextend}. We will now discuss these special topological actions in detail.


\subsection{Quasiplatonic Topological Action Formulas}

W. J. Harvey in \cite{harvey} gives necessary and sufficient conditions for $C_n$ to act topologically on a surface $X$ of genus at least two, which we refer to as Harvey's Theorem. The signatures $(n_1,n_2,n_3)$ satisfying Harvey's Theorem will be called \emph{admissible}. We state the general form of Harvey's Theorem below.

\begin{theorem}[Harvey \cite{harvey}]\label{thm:harvey}
Let $\Delta$ be a Fuchsian group with presentation as in \eqref{eq:presentation} and let $m=\text{lcm}\lbrace n_1,\ldots,n_k\rbrace$. Then there exists a surface kernel epimorphism $\varphi:\Delta\twoheadrightarrow C_n$ if and only if the following hold:
\begin{enumerate}
\item $m=lcm\lbrace n_1,\ldots,\widehat{n_i},\ldots,n_k\rbrace$ for each $i$, where $\widehat{n_i}$ denotes exclusion of $n_i$.
\item $m$ divides $n$, and if $g=0$, then $m=n$.
\item If $m$ is even, the number of periods divisible by the maximum power of 2 dividing $m$ is even.
\end{enumerate}
\end{theorem}

We are interested in the quasiplatonic case, i.e., when the quotient genus $g=0$ and there are three periods, so $k=3$. This the case which corresponds to actions of a group on a surface having a regular dessin, which are the quasiplatonic topological $G$-actions. 

We illustrate how the automorphism group $G$ of a dessin can describe a quasiplatonic topological $G$-action on a surface.

\begin{example}\label{ex:twoactions}
Let $G=C_7$. By Harvey's Theorem, the only admissible signature $(n_1,n_2,n_3)$ of a surface whose uniformizing Fuchsian group is  normally contained in triangle group $\Delta(n_1,n_2,n_3)$ is $(n_1,n_2,n_3)=(7,7,7)$. We thus look for $(0;7,7,7)$-generating vectors for $C_7$. Write $C_7$ as $\Z/7\Z$ under addition. Two such examples of $(0;7,7,7)$-generating vectors for $C_7$ are
\begin{equation*}
(3,3,1),\quad (4,2,1).
\end{equation*}
These two generating vectors correspond to distinct $C_7$-actions, because their associated surface kernel epimorphisms are not equivalent. More explicitly, there does not exist both an automorphism of $\Z/7\Z$ and an automorphism of $\Delta(7,7,7)$ which maps one generating vector to the other. This is because all automorphisms of $\Z/7\Z$ are multiplication by nonzero elements, while automorphisms of a triangle group are essentially permutations of the components of the vectors. 

Now, consider the following two dessins in Figure \ref{fig:twodessins}, drawn as the boundary of a regular, hyperbolic 14-gon inside the hyperbolic disc. The arrows indicate how to identify edges to form the Riemann surfaces. The glued edges become the edges of the dessin embedded on the surface.
\begin{figure}[ht]
\includegraphics[page=5,scale=.55]{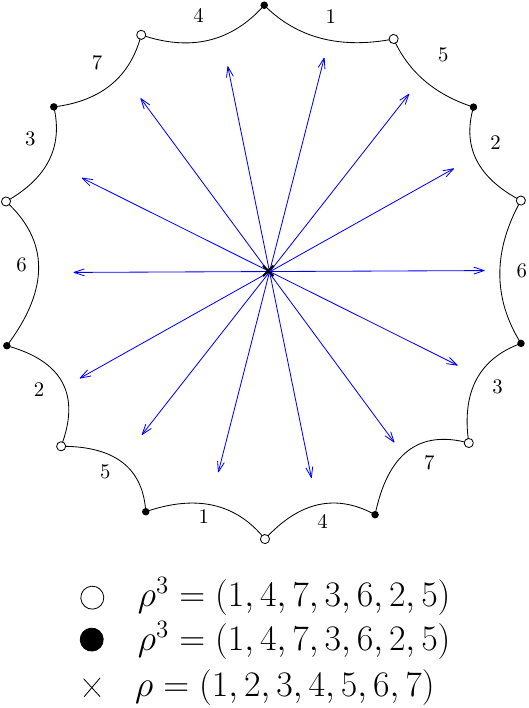}\qquad \includegraphics[page=7,scale=.55]{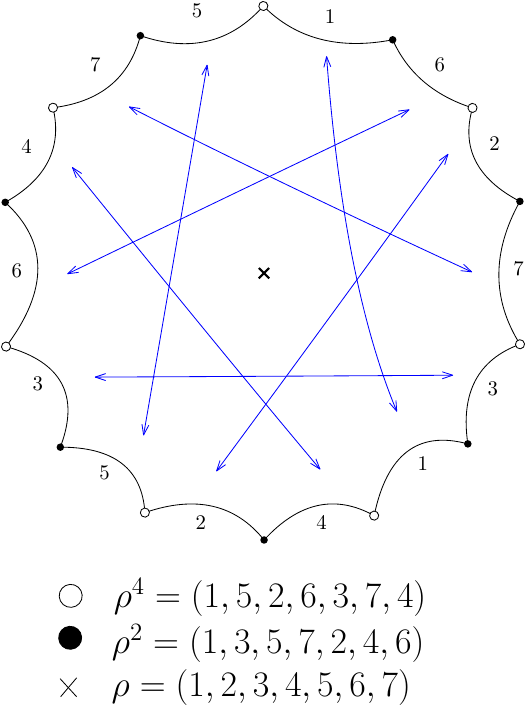}
\caption{Two regular dessins on genus three surfaces, both having $C_7$ as their automorphism group. Using clockwise orientation, the permutations written underneath the dessins describe the order of the edges around a white or black vertex, or a face.}\label{fig:twodessins}
\end{figure}
A rotation about the face center clockwise by $2\pi/7$ sends the edge numbered 1 to edge 2, edge 2 to edge 3, and so on. This is described by $\rho=(1,2,3,4,5,6,7)$. Reading clockwise again, the rotations about the white and black vertices, and the face center, respectively, are given in Figure \ref{fig:twodessins} as powers of $\rho$ next to $\circ, \bullet, \times$. The automorphism group of both dessins is indeed $C_7$, which can be verified by computing the centralizer in $S_7$ of their monodromy groups. The edge identifications specified in Figure \ref{fig:twodessins} result in two compact Riemann surfaces of genus three, which is given by the Riemann-Hurwitz formula for a $C_7$ quasiplatonic action with signature $(7,7,7)$:
\begin{equation*}
g=1+\frac{7}{2}\left(1-\frac{1}{7}-\frac{1}{7}-\frac{1}{7}\right)=3.
\end{equation*}
Both of these dessins are not equivalent and also describe two distinct topological actions with generating vectors $(3,3,1)$ and $(4,2,1)$.
\end{example}

R. Benim and A. Wootton in \cite{ben} derive explicit formulae which enumerates the number of distinct quasiplatonic topological actions of $C_n$ on a surface $X$ with a given admissible signature. They use the following definitions to describe their results.

\begin{definition}
Let $\tau_1(m,n)$ denote the number of nonzero, noncongruent solutions $x$ to the equation
\begin{equation*}
x^2+2x\equiv 0\modd m
\end{equation*}
where $gcd(x,m)=m/n$. Also, let $\tau_2(n)$ be the number of nonzero, noncongruent solutions $x$ to the equation
\begin{equation*}
x^2+x+1\equiv 0\modd n.
\end{equation*}
\end{definition}

\begin{theorem}[Benim and Wootton \cite{ben}]\label{thm:ben}
Let $n$ be a positive integer. Suppose $C_n$ acts on $X$ of signature $(n_1,n_2,n_3)$. Then the number $T$ of distinct topological actions of $C_n$ on $X$ is given in the three possibilities below.
\begin{enumerate}
\item If each $n_i$ is distinct, then
\begin{equation*}
T=\phi(\gcd(n_1,n_2,n_3))\prod_{i=1}^w\frac{p_i-2}{p_i-1}.
\end{equation*}
\item If two $n_i$ are equal so that the signature is $(n_1,n_2,n_3)=(n_1,n,n)$ with $n_1\neq n$, then
\begin{equation*}
T=\frac{1}{2}\left(\tau_1(n,n_1)+\phi(n_1)\prod_{i=1}^w\frac{p_i-2}{p_i-1}\right).
\end{equation*}
\item If all $n_i$ are equal so that the signature is $(n_1,n_2,n_3)=(n,n,n)$, then
\begin{equation*}
T=\frac{1}{6}\left(3+2\tau_2(n)+\phi(n)\prod_{i=1}^w\frac{p_i-2}{p_i-1}\right).
\end{equation*}
\end{enumerate}
\end{theorem}

When calculating $T$ for a given signature $(n_1,n_2,n_3)$, we will write this as $T(n_1,n_2,n_3)$. We may also refer to the number of topological actions of $C_n$ of signature $(n_1,n_2,n_3)$ as the \emph{$T$-value} of $(n_1,n_2,n_3)$.

\begin{example}\label{ex:twoactionverified}
Consider $C_7$ again with the signature $(7,7,7)$. According to Theorem \ref{thm:ben},
\begin{equation*}
T(7,7,7)=\frac{1}{6}\left(3+2\cdot \tau_2(7)+\phi(7)\cdot\frac{7-2}{7-1}\right)=\frac{1}{6}\left(3+2\cdot 2+5\right)=2.
\end{equation*}
(The value of $\tau_2(n)$ in general is given in Lemmas \ref{lem:tau2multi} and \ref{lem:tau2} in the next section.) Therefore, the two generating vectors of Example \ref{ex:twoactions} define the only two quasiplatonic $C_7$-actions on surfaces.
\end{example}

\section{Lemmas}\label{sec:lem}
To prove the first formula of our main result, we will use the following lemmas in order to use formulas from \cite{ben}. Lemmas \ref{lem:tau1} and \ref{lem:tau2multi} are taken from \cite{ben}, while Lemma \ref{lem:tau2} is new.

\begin{lemma}[Benim, Wootton \cite{ben}]\label{lem:tau1}
Let $m$ and $n$ be positive integers, written in their prime factorization as
\begin{equation*}
m=2^{k_0}\prod_{i=1}^\ell p_i^{k_i},\quad n=2^{h_0}\prod_{i=1}^\ell p_i^{h_i}.
\end{equation*}
Then the value of $\tau_1(m,n)$ is given as follows:
\begin{enumerate}
\item if $h_0\neq 0,1$ nor $k_0-1$, or $h_i\neq 0$ nor $k_i$ for some $i>0$, then $\tau_1(m,n)=0$;
\item if (i) above does not hold, and if $h_0=0$ or $h_0=1$, then $\tau_1(m,n)=1$. Otherwise, $\tau_1(m,n)=2$.
\end{enumerate}
\end{lemma}


\begin{lemma}[Benim, Wootton \cite{ben}]\label{lem:tau2multi}

For any positive integer $n=\prod_i p_i^{a_i}$ for distinct primes $p_i$ and positive integers $a_i$, we have $\tau_2(n)=\prod_i\tau_2(p_i^{k_i})$.
\end{lemma}

\begin{lemma}\label{lem:tau2}
Let $\tau_2(p^a)$ be the number of nonzero, noncongruent solutions $x$ to the equation
\begin{equation*}
x^2+x+1\equiv 0\modd p^a
\end{equation*}
for any integer $a\geq 1$ and odd prime $p$. Then
\begin{equation*}
\tau_2(p^a)=\left\lbrace\begin{array}{lr}
 0 &\quad p\equiv 5\modd 6, a\geq 1 \text{ or } p=3, a\geq 2\\
 1 &\quad p=3, a=1\\
 2 &\quad p\equiv 1\modd 6, a\geq 1
\end{array}\right..
\end{equation*}

\end{lemma}

To prove Lemma \ref{lem:tau2}, we need the following propositions. We say that an integer $a$ is a \emph{quadratic residue} modulo $p$, where $a$ and $p$ are relatively prime, if the congruence equation $x^2\equiv a\modd p$ has a solution for $x$. Also, the \emph{Legendre symbol} $\genfrac(){}{0}{a}{p}$ for odd prime $p$ and integer $a$ is defined to be either $1$ if $a$ is a quadratic residue modulo $p$, $-1$ if $a$ is a quadratic nonresidue modulo $p$, or 0 if $p\mid a$. (See, e.g., Chapter 3 of \cite{niven}.) We will make use of the quadratic reciprocity law: for odd primes $p$ and $q$,
\begin{equation}\label{eq:quadrec}
\genfrac(){}{0}{p}{q}\genfrac(){}{0}{q}{p}=(-1)^{\frac{p-1}{2}\cdot\frac{q-1}{2}}.
\end{equation}
We will also need Hensel's Lemma (see, e.g., Theorem 2.23 in \cite{niven}).

\begin{theorem}[Hensel's Lemma]
Let $f(x)$ be a polynomial with integer coefficients. If $f(a)\equiv 0\mod p^j$ and $f'(a)\not\equiv 0\mod p$, then there exists a unique integer $t$ modulo $p$ such that $f(a+tp^j)\equiv 0\modd p^{j+1}$.
\end{theorem}

\begin{proposition}\label{prop:2}
The congruence equation $x^2\equiv -3\modd 3^\alpha$ has exactly one solution for an integer $x$ if $\alpha=1$, and has no solutions if $\alpha\geq 2$.

\end{proposition}

The proof of Proposition \ref{prop:2} uses a general technique in solving prime-power modulus equations (see, e.g., Section 2.6 in \cite{niven}). 


\begin{proof}
Let $f(x)=x^2+3$. If $\alpha=1$, the only solution to $f(x)\equiv 0\modd 3$ is $x\equiv 0\modd 3$. Now let $\alpha\geq 2$. Because $f'(0)\equiv 0\modd 3$ (i.e., the root $x=0$ is singular) and $f(0)=3\not\equiv 0\modd 3^\alpha$, no solutions to $f(x)\equiv 0\modd 3^\alpha$ lie over the solution $x=0$ to $f(x)\equiv 0\modd 3$. Any solution to $f(x)\equiv 0\modd 3^\alpha$ must be a solution to $f(x)\equiv 0\modd 3$, and would thus necessarily be a root lying over the only solution $x=0$ to $f(x)\equiv 0\modd 3$. It follows that $f(x)\equiv 0\modd 3^\alpha$ has no solutions.
\end{proof}

\begin{proposition}\label{prop:1}
If $p$ is an odd prime and $\gcd(a,p)=1$, then, for any positive integer $\alpha$, $x^2\equiv a\modd p^\alpha$ has exactly $1+\genfrac(){}{0}{a}{p}$ integer solutions.
\end{proposition}

\begin{proof}
Let $f(x)=x^2-a$. Suppose first that $f(x)\equiv 0\modd p$ has no solutions, i.e., $\genfrac(){}{0}{a}{p}=-1$. Then $f(x)\equiv 0\modd p^\alpha$ cannot have any solutions either. Equivalently, if $f(x)\equiv 0\modd p$ has $1+\genfrac(){}{0}{a}{p}=0$ solutions, then so must $f(x)\equiv 0\modd p^\alpha$.

Now suppose $f(x)\equiv 0\modd p$ has a solution, i.e., $\genfrac(){}{0}{a}{p}=1$. Call this solution $x=b_1$. Because $p$ is odd, $b_1\not\equiv-b_1\modd p$, so $x=-b_1$ is another distinct solution modulo $p$. These two solutions are the only ones for $f(x)\equiv 0\modd p$, as $f(x)$ is a polynomial of degree two. That is, $f(x)\equiv 0\modd p$ has exactly $1+\genfrac(){}{0}{a}{p}=2$ solutions.

We next verify the hypotheses of Hensel's Lemma. Since $f'(x)=2x$, we claim $f'(b_1)=2b_1\not\equiv 0\modd p$. Assume for contradiction that $2b_1\equiv 0\modd p$, so that $2b_1=kp$ for some integer $k$. Using the fact that $f(b_1)=b_1^2-a\equiv 0\modd p$, multiplying both sides of $2b_1=kp$ by $b_1$ yields
\begin{equation*}
2a\equiv 2b_1^2=kb_1p\equiv 0\modd p.
\end{equation*}
But then $a\equiv 0\modd p$. This contradicts the assumption that $\gcd(a,p)=1$. Therefore, $f'(b_1)\not\equiv 0\modd p$. A similar argument shows $f'(-b_1)\not\equiv 0\modd p$.

We thus have $f(b_1)\equiv 0\modd p$ and $f'(b_1)\not\equiv 0\modd p$. Then by Hensel's Lemma, there exists a unique $t\modd p$ such that $f(b_1+tp)\equiv 0\modd p^2$. Let $b_2:=b_1+tp$ be the lift of $b_1$. Then $b_2\equiv b_1\modd p$, so that $f'(b_2)\equiv f'(b_1)\not\equiv 0\modd p$. Because $f(b_2)\equiv 0\modd p^2$, we can apply Hensel's Lemma again to lift to an integer $b_3$ such that $f(b_3)\equiv 0\modd p^3$. Continue this lifting process as well for $x=-b_1$, obtaining lifted solutions $\pm b_1,\pm b_2,\pm b_3,\ldots$ to $f(x)\equiv 0\modd p^\alpha$ for $\alpha=1,2,3,\ldots$. It follows that $f(x)\equiv 0\modd p^\alpha$ has at least $1+\genfrac(){}{0}{a}{p}=2$ solutions, namely $\pm b_\alpha$.

Moreover, $f(x)\equiv 0\modd p^\alpha$ must have exactly $2$ solutions. For otherwise, if $c$ is an integer such that $f(c)\equiv 0\modd p^\alpha$, then certainly $f(c)\equiv 0\modd p$. But then either $c\equiv b_1\modd p$ or $c\equiv -b_1\modd p$. By the unique extension property of Hensel's Lemma, either $c\equiv b_\alpha\modd p^\alpha$ or $c\equiv -b_\alpha\modd p^\alpha$. Thus, if $f(x)\equiv 0\modd p$ has $1+\genfrac(){}{0}{a}{p}=2$ solutions, then so must $f(x)\equiv 0\modd p^\alpha$.
\end{proof}

\begin{proposition}\label{prop:3}
For an odd prime $p$,
\begin{equation*}
\genfrac(){}{0}{-3}{p}=\left\lbrace\begin{array}{lr}
 1 & p\equiv 1\modd 6\\
 0 & p=3\\
 -1 & p\equiv 5\modd 6
\end{array}\right..
\end{equation*}
\end{proposition}

\begin{proof}
The result is given for $p=3$ by definition, so let $p\geq 5$. First,
\begin{equation*}
\genfrac(){}{0}{-3}{p}=\genfrac(){}{0}{-1}{p}\genfrac(){}{0}{3}{p}.
\end{equation*}
By the quadratic reciprocity law in Equation \eqref{eq:quadrec},
\begin{equation*}
\genfrac(){}{0}{3}{p}\genfrac(){}{0}{p}{3}=(-1)^{\frac{p-1}{2}}.
\end{equation*}
Since the value of the Legendre symbol, if it is nonzero, always squares to 1, we can rewrite the above as
\begin{equation*}
\genfrac(){}{0}{3}{p}=(-1)^{\frac{p-1}{2}}\genfrac(){}{0}{p}{3}.
\end{equation*}
Because
\begin{equation*}
\genfrac(){}{0}{-1}{p}=(-1)^\frac{p-1}{2},
\end{equation*}
we thus have
\begin{equation*}
\genfrac(){}{0}{-3}{p}=\genfrac(){}{0}{-1}{p}\genfrac(){}{0}{3}{p}=(-1)^\frac{p-1}{2}(-1)^{\frac{p-1}{2}}\genfrac(){}{0}{p}{3}=\genfrac(){}{0}{p}{3}=\left\lbrace\begin{array}{lr}
 1 & p\equiv 1\modd 3\\
 -1 & p\equiv 2\modd 3
\end{array}\right..
\end{equation*}
But $p\equiv 1\modd 3$ is equivalent to $p\equiv 1\modd 6$, while $p\equiv 2\modd 3$ is equivalent to $p\equiv 5\modd 6$. Therefore,
\begin{equation*}
\genfrac(){}{0}{-3}{p}=\left\lbrace\begin{array}{lr}
 1 & p\equiv 1\modd 6\\
 0 & p=3\\
 -1 & p\equiv 5\modd 6
\end{array}\right..
\end{equation*}

\end{proof}

\begin{proof}[Proof of Lemma \ref{lem:tau2}]
We wish to solve the congruence equation $x^2+x+1\equiv 0\modd p^a$. That is, we seek a solution to
\begin{equation*}
x^2+x+1-kp^a=0
\end{equation*}
for some integer $k$. Rewrite this equation as
\begin{equation*}
(2x+1)^2+3-4kp^a=0.
\end{equation*}
We thus seek solutions to $(2x+1)^2\equiv -3\modd p^a$. Writing $y:=2x+1$, we now wish to find $y$ such that $y^2\equiv -3\modd p^a$, which is asking for whether $-3$ is a quadratic residue modulo $p^a$. 

By Proposition \ref{prop:2}, if $p=3$ and $a=1$, then there is exactly one solution: $y\equiv 0\modd 3$, i.e., $x\equiv 1\modd 3$ solves $x^2+x+1\equiv 0\modd 3$. This shows $\tau_2(3)=1$. If $p=3$ and $a\geq 2$, then by Proposition \ref{prop:2}, $-3$ is quadratic nonresidue modulo $3^a$. This means $x^2+x+1\equiv 0\modd 3^a$ has no solutions, and hence $\tau_2(3^a)=0$ if $a\geq 2$. If $p\geq 5$, then by Proposition \ref{prop:1}, the number of solutions to $y^2\equiv -3\modd p^a$ is exactly $1+\genfrac(){}{0}{-3}{p}$. Then using Proposition \ref{prop:3}, we conclude that $y^2\equiv -3\modd p^a$ has two solutions if $p\equiv 1\modd 6$ or no solutions if $p\equiv 5\modd 6$. These are the claimed values of $\tau_2(p^a)$ for $p\geq 5$.
\end{proof}

We will also incorporate the following notation in the proofs of our main results.

\begin{definition}
Let $n=p^a$ for a prime $p$ and a positive integer $a$. If $a\geq 2$, define
\begin{equation*}
f(p^k)=\left\lbrace\begin{array}{lr}
 1 & k=0\\
 p^{k-1}(p-1) & 1\leq k\leq a-1\\
 p^{a-1}(p-2) & k=a
\end{array}\right..
\end{equation*}
If $a=1$, then define
\begin{equation*}
f(p^k)=\left\lbrace\begin{array}{lr}
 1 & k=0\\
 p-2 & k=1
\end{array}\right..
\end{equation*}
\end{definition}

The following lemma will be used many times in the computation of $QC(n)$.

\begin{lemma}
For odd prime $p$ and positive integer $a$,
\begin{enumerate}
\item $\displaystyle\sum_{k=0}^af(p^k)=p^{a-1}(p-1)$,
\item $\displaystyle\sum_{k=0}^{a-1}f(p^k)=p^{a-1}$,
\item $\displaystyle\sum_{k=1}^{a-1}f(p^k)=p^{a-1}-1$.
\end{enumerate}
\end{lemma}

\begin{proof}
We prove part (i) only (since (ii) and (iii) follow easily from (i)).
\begin{align*}
\sum_{k=0}^a f(p^k)=1+\sum_{k=1}^{a-1}p^{k-1}(p-1)+p^{a-1}(p-2)&=1+p^{a-1}-1+p^{a-1}(p-2)\\
&=p^{a-1}(p-1).
\end{align*}
\end{proof}

For the following derivations of $QC(n)$, we will write the cyclic group $C_n$ as $\Z/n\Z$, the integers modulo $n$.

\section{Proof of Theorem \ref{thm:qceven}: $QC(n)$ for Even $n$}\label{sec:qceven}

Because Theorem \ref{thm:qceven} contains the formula for $QC(n)$ based on the exponent of two, we split up the proof of Theorem \ref{thm:qceven} into two cases: when $n=2^{a_1}\prod_{i=2}^r p_i^{a_i}$ for odd primes $p_i$ and positive integers $a_i$, where either $a_1=1$ or $a_1\geq 2$.

\subsection{When $n=2^{a_1}\prod_{i=2}^rp_i^{a_i}$ with $a_1=1$}

We prove first the following recursive formula by induction on the number $r$ of distinct primes.
\begin{theorem}\label{thm:recursive2}
For $n=2\cdot\prod_{i=2}^rp_i^{a_i}$,
\begin{equation*}
QC(n\cdot p_{r+1}^{a_{r+1}})=\big(QC(n)+1-2^{r-2}\big)p_{r+1}^{a_{r+1}-1}(p_{r+1}+1)-1+2^{r-1}.
\end{equation*}
\end{theorem}

We prove the base case when $r=2$ in Theorem \ref{prop:base} below. This means we consider $n=2p^aq^b$ for odd primes $p$ and $q$, and positive integers $a$ and $b$. Before we prove Theorem \ref{prop:base}, we discuss the main technique: extend an admissible signature of $C_{2p^a}$ to an admissible signature of $C_{2p^aq^b}$. 

The number of distinct topological actions of each signature type for $C_{2p^aq^b}$ is given in Table \ref{tab:2paqb} below. Each quantity in the right-hand column is computed by applying Theorem \ref{thm:ben} to each signature type. We have split the different signatures for quasiplatonic topological $C_{2p^aq^b}$-actions into five types: $1a, 1b, 2a, 2b,$ and $2c$. These are defined by whether new signatures are extended from the two admissible signatures for $C_{2p^a}$, given as the following two types:
\begin{align*}
&\text{Type 1: } (p^k,2p^a,2p^a)\text{ with } 1\leq k\leq a;\\
&\text{Type 2: }(2p^k,p^a,2p^a) \text{ with } 0\leq k\leq a-1.
\end{align*}
New signatures for $C_{2p^aq^b}$-actions are extended from the signatures for $C_{2p^a}$-actions by placing $q^h$ in each period of one of the above signatures. In particular:
\begin{itemize}
\item Type $1a$ signatures arise by extending the signatures $(p^k,2p^a,2p^a)$ to 

$(p^kq^h,2p^aq^b,2p^aq^b)$.
\item Type $1b$ signatures arise by extending the signatures $(p^k,2p^a,2p^a)$ to 

$(p^kq^b,2p^aq^h,2p^aq^b)$.
\item Type $2a$ signatures arise by extending the signatures $(2p^k,p^a,2p^a)$ to 

$(2p^kq^h,p^aq^b,2p^aq^b)$.
\item Type $2b$ signatures arise by extending the signatures $(2p^k,p^a,2p^a)$ to 

$(2p^kq^b,p^aq^h,2p^aq^b)$.
\item Type $2c$ signatures arise by extending the signatures $(2p^k,p^a,2p^a)$ to 

$(2p^kq^b,p^aq^b,2p^aq^h)$.
\end{itemize}
The possible ranges $k$ and $h$ for the exponents of $p$ and $q$, respectively, must be considered carefully when computing the $T$-values of the above five signature types.
\begin{figure}
\begin{center}
\begin{tabular}{|l|}\hline
Type $1a$ Signatures $(p^kq^h,2p^aq^b,2p^aq^b)$ \\ \hline
\end{tabular}
\begin{tabular}{|l|l|} \hline
Ranges & Number of Topological Actions $T$ \\ \hline \hline
$k=0,1\leq h\leq b-1$ & $\frac{1}{2}q^{h-1}(q-1)$ \\ \hline
$k=0,h=b$ & $\frac{1}{2}(1+q^{b-1}(q-2))$ \\ \hline
$1\leq k\leq a-1,h=0$ & $\frac{1}{2}p^{k-1}(p-1)$ \\ \hline
$1\leq k\leq a-1,1\leq h\leq b-1$ & $\frac{1}{2}p^{k-1}(p-1)q^{h-1}(q-1)$ \\ \hline
$1\leq k\leq a-1,h=b$ & $\frac{1}{2}p^{k-1}(p-1)q^{b-1}(q-2)$ \\ \hline
$k=a,h=0$ & $\frac{1}{2}(1+p^{a-1}(p-2))$ \\ \hline
$k=a,1\leq h\leq b-1$ & $\frac{1}{2}p^{a-1}(p-2)q^{h-1}(q-1)$ \\ \hline
$k=a,h=b$ & $\frac{1}{2}(1+p^{a-1}(p-2)q^{b-1}(q-2))$ \\ \hline
\end{tabular}

\begin{tabular}{|l|}\hline
Type $1b$ Signatures $(p^kq^b,2p^aq^h,2p^aq^b)$ \\ \hline
\end{tabular}
\begin{tabular}{|l|l|} \hline
Ranges & Number of Topological Actions $T$ \\ \hline \hline
$k=0,h=0$ & $1$ \\ \hline
$k=0,1\leq h\leq b-1$ & $q^{h-1}(q-1)$ \\ \hline
$1\leq k\leq a-1,h=0$ & $p^{k-1}(p-1)$ \\ \hline
$1\leq k\leq a-1,1\leq h\leq b-1$ & $p^{k-1}(p-1)q^{h-1}(q-1)$ \\ \hline
$k=a,h=0$ & $p^{a-1}(p-2)$ \\ \hline
$k=a,1\leq h\leq b-1$ & $p^{a-1}(p-2)q^{h-1}(q-1)$ \\ \hline
\end{tabular}

\begin{tabular}{|l|}\hline
Type $2a$ Signatures $(2p^kq^h,p^aq^b,2p^aq^b)$ \\ \hline
\end{tabular}
\begin{tabular}{|l|l|} \hline
Ranges & Number of Topological Actions $T$ \\ \hline \hline
$k=0,h=0$ & $1$ \\ \hline
$k=0,1\leq h\leq b-1$ & $q^{h-1}(q-1)$ \\ \hline
$k=0,h=b$ & $q^{b-1}(q-2)$ \\ \hline
$1\leq k\leq a-1,h=0$ & $p^{k-1}(p-1)$ \\ \hline
$1\leq k\leq a-1,1\leq h\leq b-1$ & $p^{k-1}(p-1)q^{h-1}(q-1)$ \\ \hline
$1\leq k\leq a-1,h=b$ & $p^{k-1}(p-1)q^{b-1}(q-2)$ \\ \hline
\end{tabular}

\begin{tabular}{|l|}\hline
Type $2b$ Signatures $(2p^kq^b,p^aq^h,2p^aq^b)$ \\ \hline
\end{tabular}
\begin{tabular}{|l|l|} \hline
Ranges & Number of Topological Actions $T$ \\ \hline \hline
$k=0,h=0$ & $1$ \\ \hline
$k=0,1\leq h\leq b-1$ & $q^{h-1}(q-1)$ \\ \hline
$1\leq k\leq a-1,h=0$ & $p^{k-1}(p-1)$ \\ \hline
$1\leq k\leq a-1,1\leq h\leq b-1$ & $p^{k-1}(p-1)q^{h-1}(q-1)$ \\ \hline
\end{tabular}

\begin{tabular}{|l|}\hline
Type $2c$ Signatures $(2p^kq^b,p^aq^b,2p^aq^h)$ \\ \hline
\end{tabular}
\begin{tabular}{|l|l|} \hline
Ranges & Number of Topological Actions $T$ \\ \hline \hline
$k=0,h=0$ & $1$ \\ \hline
$k=0,1\leq h\leq b-1$ & $q^{h-1}(q-1)$ \\ \hline
$1\leq k\leq a-1,h=0$ & $p^{k-1}(p-1)$ \\ \hline
$1\leq k\leq a-1,1\leq h\leq b-1$ & $p^{k-1}(p-1)q^{h-1}(q-1)$ \\ \hline
\end{tabular}
\end{center}
\caption{The $T$-values for the five signature types for quasiplatonic topological $C_{2p^aq^b}$-actions.}
\label{tab:2paqb}
\end{figure}

Any extensions of the signatures of type $1a,1b,2a,2b,$ or $2c$ by a prime power (such as extending $2p^a$ signatures to $2p^aq^b$ signatures by multiplying each period by $q^h$) will be denoted as a new signature of the same type for the larger cyclic action. This motivates the following definition.

\begin{definition}
For $n=2p^a$, let $T^i_{k,2p^a}$ be the \emph{$T$-value for a type $i$ signature} for $C_n$ for $i=1$ or $2$. Now consider $n=2p_2^{a_2}\cdots p_r^{a_r}$ with $r\geq 3$ and a given admissible signature for $C_n$ of Type $i$ for $i=$ 1a, 1b, 2a, 2b, or 2c. Then we let
\begin{equation*}
T^i_{k_1,\ldots,k_r,n}
\end{equation*}
be the \emph{$T$-value for a type $i$ signature} for $C_n$.
\end{definition}

If $n=2p^a$, then the $T$-values for signatures for $C_n$ are as follows:
\begin{equation*}
T^1_{k,2p^a}=\left\lbrace\begin{array}{lr}
 0 & k=0\\
 \frac{1}{2}f(p^k) & 1\leq k\leq a-1\\
 \frac{1}{2}f(p^a)+\frac{1}{2} & k=a
\end{array}\right.
\end{equation*}
and, for $0\leq k\leq a-1$,
\begin{equation*}
T^2_{k,2p^a}=f(p^k).
\end{equation*}
From the above definition, it follows that, for $r\geq 3$ and $n=2\prod_{i=2}^r p_i^{a_i}$,
\begin{equation*}
QC(n)=\sum_i\sum_{k_1,\ldots,k_r}T^i_{k_1,\ldots,k_r,n},
\end{equation*}
where the $i$ index ranges over all five admissible signature types 1a, 1b, 2a, 2b, and 2c, and the indices $k_1,\ldots,k_r$ run over all possible ranges. If $r=2$ so that $n=2p^a$, we have
\begin{equation*}
QC(2p^a)=\sum_{k=1}^{a}\left(T^1_{k,2p^a}\right)+\sum_{k=0}^{a-1}\left(T^2_{k,2p^a}\right).
\end{equation*}

We now show the recursive formula Theorem \ref{thm:recursive2} when $r=2$. 

\begin{theorem}\label{prop:base}
\begin{equation*}
QC(2p^aq^b)=QC(2p^a)\cdot q^{b-1}(q+1)+1.
\end{equation*}
\end{theorem}

\begin{proof}
Let $k$ and $h$ be parameters corresponding to $p^k$ and $q^h$, respectively. For a fixed $k$, we add all $T$-values for admissible ranges for $h$ . Then, we sum together $T$-values over all admissible ranges for $k$ to obtain $QC(2p^aq^b)$ in terms of $T^1_{k,2p^a}$ and $T^2_{k,2p^a}$. We need to consider each of the five types of signatures for $C_{2p^aq^b}$-actions: $1a,1b,2a,2b,$ and $2c$. For the $T$-values of each signature, refer to Table \ref{tab:2paqb}.

\underline{For Type $2a$}: The possible range for $k$ for a type $2a$ signature is $0\leq k\leq a-1$. For each such fixed $k$,
\begin{equation*}
T(2p^kq^h,p^aq^b,2p^aq^b)=f(p^k)f(q^h),
\end{equation*}
where $0\leq h\leq b$. Thus, the total number of topological actions for a fixed $k$ in this range and all $h$ for signatures of type $2a$ is
\begin{align*}
f(p^k)\left(\sum_{h=0}^{b}f(q^h)\right)=f(p^k)q^{b-1}(q-1)=T^2_{k,2p^a}\cdot q^{b-1}(q-1).
\end{align*}

\underline{For Type $2b$ or $2c$}: Signatures of types 2b and 2c have the same possible $k$ ranges ($0\leq k\leq a-1$) and the same $h$ ranges ($0\leq h\leq b-1$). Their $T$-values are all $f(p^k)f(q^h)$. It follows that the total number of topological actions for a fixed $k$ and all $h$ for signatures of type 2b or 2c is
\begin{align*}
f(p^k)\left(\sum_{h=0}^{b-1}f(q^h)\right)=f(p^k)q^{b-1}=T^2_{k,2p^a}\cdot q^{b-1}.
\end{align*}

\underline{For Type $1b$}: The range for $k$ is $0\leq k\leq a$, while $h$ can range as $0\leq h\leq b-1$. Because we are considering extensions of signatures of type 1, the $T$-value $T^1_{k,2p^a}$ depends on the parameter $k$. The $T$-value for type $1b$ signatures is $f(p^k)f(q^h)$. Then the total number of topological actions for a fixed $k$ and all $h$ for signatures of type $1b$ is
\begin{align*}
f(p^k)\left(\sum_{h=0}^{b-1}f(q^h)\right)&=f(p^k)q^{b-1}=\left\lbrace\begin{array}{lr}
 q^{b-1} & k=0\\
 T^1_{k,2p^a}\cdot 2q^{b-1} & 1\leq k\leq a-1\\
 T^1_{a,2p^a}\cdot 2q^{b-1}-q^{b-1} & k=a
\end{array}\right..
\end{align*}

\underline{For Type $1a$}: Signatures of type $1a$ are the most difficult to consider, because their $T$-values require a different formula to compute, according to Theorem \ref{thm:ben}. In this case, there are two subcases to consider: when $1\leq k\leq a-1$, and either $k=0$ or $k=a$. This must be taken into account, because $T^1_{k,2p^a}$ depends on these ranges for $k$.
\begin{itemize}
\item If $1\leq k\leq a-1$, then the $T$-value of type $1a$ signatures is $(1/2)f(p^k)f(q^h)$. Then ranging over all possible $h$, the total number of topological actions of type $1a$ signatures with fixed $k$ in the given range is
\begin{align*}
\frac{1}{2}f(p^k)\left(\sum_{h=0}^b f(q^h)\right)=\frac{1}{2}f(p^k)q^{b-1}(q-1)=T^1_{k,2p^a}\cdot q^{b-1}(q-1).
\end{align*}

\item If $k=0$, then $h\neq 0$, since otherwise, we would be calculating the $T$-value for the signature $(p^0q^0,2p^aq^b,2p^aq^b)$, which is zero. Refer to Table \ref{tab:2paqb} for the $T$-values of type $1a$ signatures when $k=0$. The total number of topological actions of type 1a signatures with $k=0$ is then
\begin{align*}
f(p^0)\left(\sum_{h=1}^{b-1}\left(\frac{1}{2}f(q^h)\right)+\frac{1}{2}f(q^b)+\frac{1}{2}\right)&=\frac{1}{2}\cdot 1\cdot\left(q^{b-1}-1+q^{b-1}(q-2)+1\right)\\
&=\frac{1}{2}q^{b-1}(q-1).
\end{align*}
A $T$-value of a type 1 signature for $C_{2p^a}$ cannot be incorporated in this case, since $(p^0,2p^a,2p^a)$ is not an admissible signature for $C_{2p^a}$-actions.

\item If $k=a$, then $h$ may range as $0\leq h\leq b$. When $h$ is in the range $1\leq h\leq b-1$, we do not need an extra $\frac{1}{2}$ term, since $f(q^h)$ is even. In particular, if $h=0$ or $h=a$, we will need the extra $\frac{1}{2}$ term, since then both $f(q^h)$ and $f(p^a)$ are odd. Refer to Table \ref{tab:2paqb} for the $T$-values of type $1a$ signatures when $k=a$. The total number of topological actions of type $1a$ signatures with $k=a$ is then
\begin{align*}
\frac{1}{2}f(p^a)&+\frac{1}{2}+f(p^a)\left(\sum_{h=1}^{b-1}\frac{1}{2}f(q^h)\right)+\frac{1}{2}f(q^b)f(p^a)+\frac{1}{2}\\
&\quad=1+\frac{1}{2}f(p^a)+\frac{1}{2}f(p^a)\left(q^{b-1}(q-1)\right)+\frac{1}{2}f(p^a)q^{b-1}(q-2)\\
&\quad=1+\frac{1}{2}f(p^a)q^{b-1}(q-1)\\
&\quad=T^1_{a,2p^a}\cdot q^{b-1}(q-1)+1-\frac{1}{2}q^{b-1}(q-1).
\end{align*}
\end{itemize}

\begin{figure}
\begin{center}
\begin{tabular}{|l|l|l|}\hline
Type & Ranges & Topological Actions of $2p^aq^b$ for fixed $k$ \\ \hline
$1a$ & $k=0$ & $\frac{1}{2}q^{b-1}(q-1)$ \\ \hline
$1a$ & $1\leq k\leq a-1$ & $T^1_{k,2p^kq^h}\cdot q^{b-1}(q-1)$ \\ \hline
$1a$ & $k=a$ & $T^1_{a,2p^a}\cdot q^{b-1}(q-1)+1-\frac{1}{2}q^{b-1}(q-1)$ \\ \hline
$1b$ & $k=0$ & $q^{b-1}$ \\ \hline
$1b$ & $1\leq k\leq a-1$ & $T^1_{k,2p^aq^b}\cdot 2q^{b-1}$ \\ \hline
$1b$ & $k=a$ & $T^1_{a,2p^aq^b}\cdot 2q^{b-1}-q^{b-1}$ \\ \hline
$2a$ & $0\leq k\leq a-1$ & $T^2_{k,2p^a}\cdot q^{b-1}(q-1)$ \\ \hline
$2b$ & $0\leq k\leq a-1$ & $T^2_{k,2p^aq^b}\cdot q^{b-1}$ \\ \hline
$2c$ & $0\leq k\leq a-1$ & $T^2_{k,2p^aq^b}\cdot q^{b-1}$ \\ \hline
\end{tabular}
\end{center}
\caption{Quasiplatonic topological $C_{2p^aq^b}$-actions for ranges of $k$.}
\label{tab:2paqbfixedk}
\end{figure}

Table \ref{tab:2paqbfixedk} summarizes the calculations from above. Finally, adding all the $T$-values for possible $k$-ranges, $QC(2p^aq^b)$ is computed to be
\begin{align*}
QC(2p^aq^b)&=\frac{1}{2}q^{b-1}+q^{b-1}+\left(\sum_{k=1}^{a-1}T^1_{k,2p^a}\right)\left(q^{b-1}(q-1)+2q^{b-1}\right)\\
&\quad +T^1_{a,2p^a}\left(q^{b-1}(q-1)+2q^{b-1}\right)+1-\frac{1}{2}q^{b-1}(q-1)-q^{b-1}\\
&\quad +\left(\sum_{k=0}^{a-1}T^2_{k,2p^a}\right)\left(q^{b-1}(q-1)+2q^{b-1}\right)\\
&=\left(\sum_{k=1}^{a}\left(T^1_{k,2p^a}\right)+\sum_{k=0}^{a-1}\left(T^2_{k,2p^a}\right)\right)q^{b-1}(q+1)+1\\
&=QC(2p^a)\cdot q^{b-1}(q+1)+1.
\end{align*}

\end{proof}

Proving the recursive formula 
\begin{equation*}
QC(n\cdot p_{r+1}^{a_{r+1}})=\left(QC(n)-2^{r-2}+1\right)p_{r+1}^{a_{r+1}-1}(p_{r+1}+1)+2^{r-1}-1
\end{equation*}
of Theorem \ref{thm:recursive2} will then allow us to derive $QC(n)$ for $n=2\prod_{i=2}^r p_i^{a_i}$.

\begin{proof}[Proof of Theorem \ref{thm:recursive2}]
The base case of the formula is proven already in Theorem \ref{prop:base} when $r=2$. So assume $r\geq 3$. We use the same technique as before: we extend signatures for $C_n$ to signatures for $C_{n\cdot p_{r+1}^{a_{r+1}}}$. This is done by multiplying each period of a signature for $C_n$ by $p_{r+1}^t$, for some nonnegative integer $t\leq a_{r+1}$. The signatures for $C_n$ are of the form
\begin{equation*}
\left(2^{k_1}p_2^{k_2}\cdots p_r^{k_r},2^{\ell_1}p_2^{\ell_2}\cdots p_r^{\ell_r},2^{h_1}p_2^{h_2}\cdots p_r^{h_r}\right)
\end{equation*}
for nonnegative integers $k_i,\ell_i$ and $h_i$. The ranges for the exponents is determined by Harvey's Theorem. In particular, each of $k_i,\ell_i, $ and $h_i$ are at most $a_{r+1}$.

Consider the $T$-value of an extension of a signature for $C_n$ by the odd prime $p_{r+1}^{a_{r+1}}$,i.e., the $T$-value of a signature for $C_{n\cdot p_{r+1}^{a_{r+1}}}$-actions. Signatures for $C_{n\cdot p_{r+1}^{a_{r+1}}}$-actions consist of the following three forms:
\begin{align}
&\left(2^{k_1}p_2^{k_2}\cdots p_r^{k_r}p_{r+1}^t,2^{\ell_1}p_2^{\ell_2}\cdots p_r^{\ell_r}p_{r+1}^{a_{r+1}},2^{h_1}p_2^{h_2}\cdots p_r^{h_r}p_{r+1}^{a_{r+1}}\right)\label{eq:sign1}, 0\leq t\leq a_{r+1}\\
&\left(2^{k_1}p_2^{k_2}\cdots p_r^{k_r}p_{r+1}^{a_{r+1}},2^{\ell_1}p_2^{\ell_2}\cdots p_r^{\ell_r}p_{r+1}^{t},2^{h_1}p_2^{h_2}\cdots p_r^{h_r}p_{r+1}^{a_{r+1}}\right)\label{eq:sign2},0\leq t\leq a_{r+1}-1\\
&\left(2^{k_1}p_2^{k_2}\cdots p_r^{k_r}p_{r+1}^{a_{r+1}},2^{\ell_1}p_2^{\ell_2}\cdots p_r^{\ell_r}p_{r+1}^{a_{r+1}},2^{h_1}p_2^{h_2}\cdots p_r^{h_r}p_{r+1}^{t}\right)\label{eq:sign3}0\leq t\leq a_{r+1}-1.
\end{align}
By Harvey's Theorem, exactly two of $k_1,\ell_1,$ or $h_1$ (which are the exponents of 2 in the signatures above) must be equal to one. The remaining exponents for $p_i$ are determined similarly by Harvey's Theorem: at least two of $k_i,\ell_i$ or $h_i$ are equal to their maximum power, $a_i$. Because $T$-values of signatures differ depending on whether two periods are equal or not according to Theorem \ref{thm:ben}, there are two cases to consider: when the exponents $\ell_1,\ldots,\ell_r$ and $h_1,\ldots,h_r$ are either equal or not all equal.

First suppose that the exponents $\ell_1,\ldots,\ell_r$ and $h_1,\ldots,h_r$ are not all equal. Then the $T$-value of an admissible signature for $C_n$ of the form
\begin{equation*}
\left(2^{k_1}p_2^{k_2}\cdots p_r^{k_r},2^{\ell_1}p_2^{\ell_2}\cdots p_r^{\ell_r},2^{h_1}p_2^{h_2}\cdots p_r^{h_r}\right)
\end{equation*}
is given by $\prod_{i=2}^rf(p_i^{k_i})$. That is,
\begin{equation*}
T^2_{k_2,\ldots,k_r,n}=\prod_{i=2}^rf(p_i^{k_i}).
\end{equation*}
Hold $k_i$ fixed for $i=2,\ldots, r$. The $T$-values of the signatures \eqref{eq:sign1}-\eqref{eq:sign3} with their respective $t$-ranges is then given as
\begin{align*}
T&\left(2^{k_1}p_2^{k_2}\cdots p_r^{k_r}p_{r+1}^t,2^{\ell_1}p_2^{\ell_2}\cdots p_r^{\ell_r}p_{r+1}^{a_{r+1}},2^{h_1}p_2^{h_2}\cdots p_r^{h_r}p_{r+1}^{a_{r+1}}\right)\\
&=\prod_{i=2}^r\left(f(p_i^{k_i})\right)\cdot \sum_{t=0}^{a_{r+1}}f(p_{r+1}^t),\\
T&\left(2^{k_1}p_2^{k_2}\cdots p_r^{k_r}p_{r+1}^{a_{r+1}},2^{\ell_1}p_2^{\ell_2}\cdots p_r^{\ell_r}p_{r+1}^{t},2^{h_1}p_2^{h_2}\cdots p_r^{h_r}p_{r+1}^{a_{r+1}}\right)\\
&=T\left(2^{k_1}p_2^{k_2}\cdots p_r^{k_r}p_{r+1}^{a_{r+1}},2^{\ell_1}p_2^{\ell_2}\cdots p_r^{\ell_r}p_{r+1}^{a_{r+1}},2^{h_1}p_2^{h_2}\cdots p_r^{h_r}p_{r+1}^{t}\right)\\
&=\prod_{i=2}^r\left(f(p_i^{k_i})\right)\cdot \sum_{t=0}^{a_{r+1}-1}f(p_{r+1}^t).\\
\end{align*}
In other words, the $T$-values of the signatures \eqref{eq:sign1}-\eqref{eq:sign3} is multiplicative. Summing over these $T$-values, the total number of topological actions of $n\cdot p_{r+1}^{a_{r+1}}$ for fixed $k_2,\ldots,k_r$, in this case, is thus
\begin{align}
&\prod_{i=2}^rf(p_i^{k_i})\cdot\left(\sum_{t=0}^{a_{r+1}}(f(p_{r+1}^t))+2\sum_{t=0}^{a_{r+1}-1}(f(p_{r+1}^t))\right)\nonumber\\
&\quad =T^2_{k_2,\ldots,k_r,n}\cdot\left(1+\sum_{t=1}^{a_{r+1}-1}(p_{r+1}^t)+p_{r+1}^{a_{r+1}-1}+2+2\sum_{t=1}^{a_{r+1}-1}(p_{r+1}^t)\right)\nonumber\\
&\quad =T^2_{k_2,\ldots,k_r,n}\cdot\left(1+p_{r+1}^{a_{r+1}-1}-1+p_{r+1}^{a_{r+1}-1}+2+2(p_{r+1}^{a_{r+1}-1}-1)\right)\nonumber\\
&\quad =T^2_{k_2,\ldots,k_r,n}\cdot\left(p_{r+1}^{a_{r+1}-1}(p_{r+1}+1)\right)\label{eq:case2}.
\end{align}

Next, suppose that the exponents $\ell_1,\ldots,\ell_r$ and $h_1,\ldots,h_r$ are all equal. By Harvey's Theorem, each of these exponents achieves its maximum bound, i.e., $\ell_i=h_i=a_i$. This also implies that $k_1=0$, since we cannot have $k_1=\ell_1=h_1=a_1=1$ (otherwise, the three periods of a signature would have the highest power of two dividing $n$, contrary to Harvey's Theorem). Hence, the signatures for $C_n$ that we are extending are of the form
\begin{equation}\label{eq:signallequal}
\left(p_2^{k_2}\cdots p_r^{k_r},2p_2^{a_2}\cdots p_r^{a_r},2p_2^{a_2}\cdots p_r^{a_r}\right).
\end{equation}
We thus consider signatures for $C_{n\cdot p_{r+1}^{a_{r+1}}}$ of the following two possible extended forms:
\begin{align}
&\left(p_2^{k_2}\cdots p_r^{k_r}p_{r+1}^t,2p_2^{a_2}\cdots p_r^{a_r}p_{r+1}^{a_{r+1}},2p_2^{a_2}\cdots p_r^{a_r}p_{r+1}^{a_{r+1}}\right)\label{eq:sign4}, 0\leq t\leq a_{r+1}\\
&\left(p_2^{k_2}\cdots p_r^{k_r}p_{r+1}^{a_{r+1}},2p_2^{a_2}\cdots p_r^{a_r}p_{r+1}^{t},2p_2^{a_2}\cdots p_r^{a_r}p_{r+1}^{a_{r+1}}\right)\label{eq:sign5},0\leq t\leq a_{r+1}-1.
\end{align}
There is no third possible signature in this case, because the second and third periods of the signatures \eqref{eq:signallequal} for $C_n$ are now equal. Equivalently, extending by $t$ in the third component of \eqref{eq:signallequal} would be identical to those extensions by $t$ in the second component, given by \eqref{eq:sign5}.

Signatures of the form \eqref{eq:sign4} have two equal periods. Hence, $T$-values will incur factors of $1/2$ depending on whether all the $k_i$ ranges are maximal (i.e., $k_i=a_i$ or $k_i=0$ for all $i=2,\ldots,r$) or at least one $k_i$ falls in a mid-range (i.e., $1\leq k_i\leq a_i-1$ for some $i$). We must consider these cases separately.


\underline{If $1\leq k_i\leq a_i-1$ for some $2\leq i\leq r$}: The $T$-value of an admissible signature for $C_n$ of the form as in \eqref{eq:signallequal} is $(1/2)\cdot\prod_{i=2}^rf(p_i^{k_i})$. An identical formula also holds when two periods for a signature of $n\cdot p_{r+1}^{a_{r+1}}$ are equivalent. It follows that
\begin{equation*}
T^1_{k_2,\ldots,k_r,n}=\frac{1}{2}\prod_{i=2}^rf(p_i^{k_i}).
\end{equation*}

First, consider signatures of the form \eqref{eq:sign4} above, and hold each $k_i$ fixed for $i=2,\ldots,r$. The $T$-value of this signature type is
\begin{equation*}
T\left(p_2^{k_2}\cdots p_r^{k_r}p_{r+1}^t,2p_2^{a_2}\cdots p_r^{a_r}p_{r+1}^{a_{r+1}},2p_2^{a_2}\cdots p_r^{a_r}p_{r+1}^{a_{r+1}}\right)=\frac{1}{2}\prod_{i=2}^{r}f(p_i^{k_i})\cdot f(p_{r+1}^t).
\end{equation*}
Then summing all $T$-values over possible $t$ ranges and using a similar calculation as in the previous case above, the total number of topological actions of $C_{n\cdot p_{r+1}^{a_{r+1}}}$ of the form \eqref{eq:sign4} is
\begin{align}
\frac{1}{2}\prod_{i=2}^rf(p_i^{k_i})\left(\sum_{t=0}^{a_{r+1}}f(p_{r+1}^t)\right)=T^1_{k_2,\ldots,k_r,n}\cdot p_{r+1}^{a_{r+1}-1}(p_{r+1}-1).\label{eq:case1part1}
\end{align}

Now consider signatures of the form \eqref{eq:sign5}. These signatures have all distinct periods, and hence no factor of $1/2$ occurs in their $T$-values. In other words, the $T$-value of these signatures is identical to the $T$-value computed in the previous signature type above, but without the $1/2$ factor. However, when $p_{r+1}^{a_{r+1}}$ is removed from each period, the resulting signature for $C_n$ will have two equal periods. In order to incorporate $T^1_{k_2,\ldots,k_r,n}$, we will need to adjust our resulting $T$-value sum by multiplying and diving by $2$. Holding each $k_i$ fixed for $i=2,\ldots,r$ and summing over the possible $t$ ranges, the total number of topological actions for $C_{n\cdot p_{r+1}^{a_{r+1}}}$ of the form \eqref{eq:sign5} is then
\begin{align}
\prod_{i=2}^rf(p_i^{k_i})\left(\sum_{t=0}^{a_{r+1}-1}f(p_{r+1}^t)\right)&=\frac{1}{2}\prod_{i=2}^rf(p_i^{k_i})\cdot 2p_{r+1}^{a_{r+1}-1}\nonumber\\
&=T^1_{k_2,\ldots,k_r,n}\cdot 2p_{r+1}^{a_{r+1}-1}\label{eq:case1part2}.
\end{align}
Summing \eqref{eq:case1part1} and \eqref{eq:case1part2} gives us the number of distinct quasiplatonic topological actions of $C_{n\cdot p_{r+1}^{a_{r+1}}}$ in this case as 
\begin{equation}
T^1_{k_2,\ldots,k_r,n}\cdot p_{r+1}^{a_{r+1}-1}(p_{r+1}+1)\label{eq:case1total1}.
\end{equation}

\underline{If none of $k_i$ satisfy $1\leq k_i\leq a_i-1$ for $2\leq i\leq r$}: This case requires several possible subcases to consider.
\begin{itemize}
\item Suppose all $k_i=0$ for $i=2,\ldots,r$. Signatures of the form \eqref{eq:sign4} will then become
\begin{equation}\label{eq:signallzero}
\left(p_{r+1}^t,2p_2^{a_2}\cdots p_r^{a_r}p_{r+1}^{a_{r+1}},2p_2^{a_2}\cdots p_r^{a_r}p_{r+1}^{a_{r+1}}\right).
\end{equation}
The range for $t$ here is $1\leq t\leq a_{r+1}$. We cannot have $t=0$, as then the signature \eqref{eq:signallzero} becomes $(1,2^{a_1}p_2^{a_2}\cdots p_r^{a_r}p_{r+1}^{a_{r+1}},2^{a_1}p_2^{a_2}\cdots p_r^{a_r}p_{r+1}^{a_{r+1}})$, which is not admissible by Harvey's Theorem. This signature is not obtained by extending any signature for $C_n$, because the signature 
\begin{equation}
(1,2^{a_1}p_2^{a_2}\cdots p_r^{a_r},2^{a_1}p_2^{a_2}\cdots p_r^{a_r}),
\end{equation}
which is the only candidate for a signature for $C_n$ to be extended to obtain \eqref{eq:signallzero}, is also not admissible. Thus, once the $T$-value of the signature \eqref{eq:signallzero} is computed, there is no associated $T$-value for $C_n$ in this case. The number of topological actions of this signature for the given $t$ ranges is then
\begin{align}
\sum_{t=0}^{a_{r+1}-1}\frac{1}{2}(f(p_{r+1}^t))+\frac{1}{2}f(p_{r+1}^{a_{r+1}})+\frac{1}{2}\label{eq:0a}.
\end{align}
Next, signatures of the form \eqref{eq:sign5} then become
\begin{equation*}
\left(p_{r+1}^{a_{r+1}},2p_2^{a_2}\cdots p_r^{a_r}p_{r+1}^{t},2p_2^{a_2}\cdots p_r^{a_r}p_{r+1}^{a_{r+1}}\right).
\end{equation*}
The range for $t$ here is $0\leq t\leq a_{r+1}-1$. Then the number of topological actions of this signature for the given $t$-ranges is
\begin{equation}
\sum_{t=0}^{a_{r+1}-1}f(p_{r+1}^t)\label{eq:0b}.
\end{equation}
Combining \eqref{eq:0a} and \eqref{eq:0b} yields the total number of topological actions of this signature for the case when each $k_i=0$ for $2\leq i\leq r$:
\begin{align}
\sum_{t=0}^{a_{r+1}-1}\frac{1}{2}(f(p_{r+1}^t))+&\frac{1}{2}f(p_{r+1}^{a_{r+1}})+\frac{1}{2}+\sum_{t=0}^{a_{r+1}-1}f(p_{r+1}^t)\nonumber\\
&=\frac{1}{2}p_{r+1}^{a_{r+1}-1}(p_{r+1}+1)\label{eq:case1total2}.
\end{align}

\item Suppose not all $k_i=0$, and keep each $k_i$ fixed for $i=2,\ldots,r$. Consider signatures of the form \eqref{eq:sign4}, i.e.,
\begin{equation*}
\left(p_2^{k_2}\cdots p_r^{k_r}p_{r+1}^t,2p_2^{a_2}\cdots p_r^{a_r}p_{r+1}^{a_{r+1}},2p_2^{a_2}\cdots p_r^{a_r}p_{r+1}^{a_{r+1}}\right).
\end{equation*}
The range for $t$ is $0\leq t\leq a_{r+1}$. However, the $T$-values of these signatures depends on whether $t$ satisfies $1\leq t\leq a_{r+1}-1$ or not. We describe the three different possible $T$-values for $t$, and sum those up. In particular, we want to relate our resulting quantity with the $T$-value of the associated signature for $C_n$. In this case, the signature for $C_n$ would be of the form
\begin{equation*}
\left(p_2^{k_2}\cdots p_r^{k_r},2^{a_1}p_2^{a_2}\cdots p_r^{a_r},2^{a_1}p_2^{a_2}\cdots p_r^{a_r}\right),
\end{equation*}
whose $T$-value is given by
\begin{equation*}
T^1_{k_2,\ldots,k_r,n}=\frac{1}{2}\prod_{i=2}^r(f(p_i^{k_i}))+\frac{1}{2}.
\end{equation*}
This is true, because Theorem \ref{thm:ben} gives this $T$-value when two periods of a signature for $C_n$ are equal. Now, the $T$-values of the signatures \eqref{eq:sign4} is given below.
\begin{itemize}
\item When $t=0$,
\begin{align*}
&T\left(p_2^{k_2}\cdots p_r^{k_r},2p_2^{a_2}\cdots p_r^{a_r}p_{r+1}^{a_{r+1}},2p_2^{a_2}\cdots p_r^{a_r}p_{r+1}^{a_{r+1}}\right)\\
&\quad=\frac{1}{2}\prod_{i=2}^r(f(p_i^{k_i}))+\frac{1}{2}.
\end{align*}
\item When $1\leq t\leq a_{r+1}-1$,
\begin{align*}
&T\left(p_2^{k_2}\cdots p_r^{k_r}p_{r+1}^t,2p_2^{a_2}\cdots p_r^{a_r}p_{r+1}^{a_{r+1}},2p_2^{a_2}\cdots p_r^{a_r}p_{r+1}^{a_{r+1}}\right)\\
&\quad=\frac{1}{2}\prod_{i=2}^r(f(p_i^{k_i}))f(p_{r+1}^t).
\end{align*}
\item When $t=a_{r+1}$,
\begin{align*}
&T\left(p_2^{k_2}\cdots p_r^{k_r}p_{r+1}^{a_{r+1}},2p_2^{a_2}\cdots p_r^{a_r}p_{r+1}^{a_{r+1}},2p_2^{a_2}\cdots p_r^{a_r}p_{r+1}^{a_{r+1}}\right)\\
&\quad=\frac{1}{2}\prod_{i=2}^r(f(p_i^{k_i}))\cdot f(p_{r+1}^{a_{r+1}})+\frac{1}{2}.
\end{align*}
\end{itemize}
Summing the above $T$-values over the given $t$ ranges, the total number of quasiplatonic topological actions for $C_{n\cdot p_{r+1}^{a_{r+1}}}$ of this signature for the case when not all $k_i=0$ for $2\leq i\leq r$ of the form \eqref{eq:sign4} is then
\begin{align}
&\left(\frac{1}{2}\prod_{i=2}^r(f(p_i^{k_i}))+\frac{1}{2}\right)+\frac{1}{2}\prod_{i=2}^r(f(p_i^{k_i}))\left(\sum_{t=1}^{a_{r+1}-1}f(p_{r+1}^t)\right)\nonumber\\
&\hspace{33mm}+\left(\frac{1}{2}\prod_{i=2}^r(f(p_i^{k_i}))\cdot f(p_{r+1}^{a_{r+1}})+\frac{1}{2}\right)\nonumber\\
&=1+\frac{1}{2}\prod_{i=2}^r(f(p_i^{k_i}))\cdot p_{r+1}^{a_{r+1}-1}(p_{r+1}-1)\nonumber\\
&=1+\left(T^1_{k_2,\ldots,k_r,n}-\frac{1}{2}\right)p_{r+1}^{a_{r+1}}(p_{r+1}-1)\label{eq:a1}.
\end{align}
Next, consider signatures of the form \eqref{eq:sign5}. The $T$-value for the signature for $C_n$ which extends to \eqref{eq:sign5} is similar to the case just described. For signatures \eqref{eq:sign5}, the $T$-values do not include any factors of $1/2$, because all periods are distinct. Here, the $t$-ranges are $0\leq t\leq a_{r+1}-1$. Thus, the total number of topological actions of this signature for the case when not all $k_i=0$ for $i=2,\ldots,r$ of the form \eqref{eq:sign5} is
\begin{align}
\prod_{i=2}^r(f(p_i^{k_i}))\cdot\left(\sum_{t=0}^{a_{r+1}-1}f(p_{r+1}^t)\right)&=\prod_{i=2}^r(f(p_i^{k_i}))\cdot p_{r+1}^{a_{r+1}-1}\nonumber\\
&=\left(T^1_{k_2,\ldots,k_r,n}-\frac{1}{2}\right)\cdot 2p_{r+1}^{a_{r+1}-1}\label{eq:a2}.
\end{align}
Summing \eqref{eq:a1} and \eqref{eq:a2} gives the total number of topological actions of $C_{n\cdot p_{r+1}^{a_{r+1}}}$ with signatures having not all $k_i=0$ as
\begin{equation}
1+\left(T^1_{k_2,\ldots,k_r,n}-\frac{1}{2}\right)p_{r+1}^{a_{r+1}-1}(p_{r+1}+1)\label{eq:case1total3}.
\end{equation}
\end{itemize}
This completes the second case.

We are now finally in position to compute $QC(n\cdot p_{r+1}^{a_{r+1}})$. We will sum up all possible signatures, whose $T$-values are given in either \eqref{eq:case2}, \eqref{eq:case1total1} \eqref{eq:case1total2}, or \eqref{eq:case1total3}. That is, we compute the total sum over all possible signatures ranging over possible $k_2,\ldots,k_r$ ranges. The interesting case occurs when signatures have periods with exponents having not all $k_i=0$, whose corresponding $T$-values are given in \eqref{eq:case1total3}. Here, there are $2^{r-1}-1$ such possible $(r-1)$-tuples of the numbers $0$ or $1$ when excluding the case of all $0$. Hence, we will add \eqref{eq:case1total3} to itself $2^{r-1}-1$ times. As cases are added, the terms $T^i_{k_2,\ldots,k_r,n}$ involving $T$-values of signatures for $C_n$ do not get added together, since they are representing functions with different values. Using the fact that
\begin{equation*}
QC(n)=\sum_{i=1}^2\sum_{k_2,\ldots,k_r}T^i_{k_2,\ldots,k_r,n},
\end{equation*}
we conclude that
\begin{align*}
QC(n\cdot p_{r+1}^{a_{r+1}})&=\left(\sum_{k_2,\ldots,k_r}T^2_{k_2,\ldots,k_r,n}\right)p_{r+1}^{a_{r+1}-1}(p_{r+1}+1)\\
&\quad +\left(\sum_{\substack{k_2,\ldots,k_r\\1\leq k_i\leq a_i-1\\ \text{for some }i}}T^1_{k_2,\ldots,k_r,n}\right)p_{r+1}^{a_{r+1}-1}(p_{r+1}+1)+\frac{1}{2}p_{r+1}^{a_{r+1}-1}(p_{r+1}+1)\\
&\quad+2^{r-1}-1+\left(T^1_{k_2,\ldots,k_r,n}-\frac{1}{2}\cdot(2^{r-1}-1)\right)\cdot p_{r+1}^{a_{r+1}-1}(p_{r+1}+1)\\
&=\left(\sum_{i=1}^2\sum_{k_2,\ldots,k_r}T^i_{k_2,\ldots,k_r}\right)p_{r+1}^{a_{r+1}-1}(p_{r+1}+1)\\
&\quad-\frac{1}{2}(2^{r-1}-1)p_{r+1}^{a_{r+1}-1}(p_{r+1}+1)+\frac{1}{2}p_{r+1}^{a_{r+1}-1}(p_{r+1}+1)+2^{r-1}-1\\
&=\left(QC(n)-2^{r-2}+1\right)p_{r+1}^{a_{r+1}-1}(p_{r+1}+1)+2^{r-1}-1.
\end{align*}
\end{proof}

Using Theorem \ref{thm:recursive2}, we can prove the formula for $QC(n)$ when $n$ is an even number not divisible by $4$.

\begin{proof}[Proof of Theorem \ref{thm:qceven} when $n$ is even and not divisible by 4]
We prove that
\begin{equation}\label{eq:qceven1}
QC(n)=\frac{1}{2}\prod_{i=2}^rp_i^{a_i-1}(p_i+1)+2^{r-2}-1
\end{equation}
by induction on the number of odd primes of $n=2p_2^{a_2}\cdots p_r^{a_r}$.

The base case is one odd prime, so that $n=2p^a$ for an odd prime $p$, $a\geq 1$, and $r=2$ in the formula above. All signatures for $C_n$ are either of the form $(p^k,2p^a,2p^a)$ for $1\leq k\leq a$ or $(2p^k,p^a,2p^a)$ for $0\leq k\leq a-1$. The associated $T$-values are then
\begin{align*}
T(p^k,2p^a,2p^a)=\left\lbrace\begin{array}{lr}
 \frac{1}{2}f(p^k) & 1\leq k\leq a-1\\
 \frac{1}{2}f(p^a)+\frac{1}{2} & k=a
\end{array}\right.
\end{align*}
and
\begin{equation*}
T(2p^k,p^a,2p^a)=f(p^k).
\end{equation*}
It follows that
\begin{align*}
QC(2p^a)=\sum_{i=1}^{a-1}\left(\frac{1}{2}f(p^k)\right)+\frac{1}{2}f(p^a)+\frac{1}{2}+\sum_{i=0}^{a-1}f(p^k)=\frac{1}{2}p^{a-1}(p+1),
\end{align*}
which is the desired formula when $r=2$.

Suppose for induction that \eqref{eq:qceven1} holds for $n=2p_2^{a_2}\cdots p_r^{a_r}$ with $r\geq 3$. Consider an odd prime $p_{r+1}^{a_{r+1}}$ not dividing $n$. Using Theorem \ref{thm:recursive2} and the induction hypothesis,
\begin{align*}
QC(n\cdot p_{r+1}^{a_{r+1}})&=(QC(n)-2^{r-2}+1)p_{r+1}^{a_{r+1}}(p_{r+1}+1)+2^{r-1}-1\\
&=\frac{1}{2}\prod_{i=2}^{r+1}p_i^{a_i-1}(p_i+1)+2^{r-1}-1.
\end{align*}
By mathematical induction, we conclude that \eqref{eq:qceven1} is true for even $n$ which are not divisible by 4.
\end{proof}

\subsection{When $n$ is of the form $n=2^{a_1}\prod_{i=2}^rp_i^{a_i}$ with $a_1\geq 2$}

In order to prove the formula for $QC(n)$ when $n$ is even of the form $n=2^{a_1}\prod_{i=2}^r p_i^{a_i}$ with $a_1\geq 2$, we will incorporate another recursive formula. The following fact will be used often in the proof:
\begin{equation}\label{phi}
\sum_{k=0}^{n-1}\phi(2^k)=\phi(2^n).
\end{equation}

\begin{theorem}\label{thm:recursiveqc}
\begin{equation*}
QC(2^{a_1}p_2^{a_2}\cdots p_r^{a_r})=2\cdot QC(2^{a_1-1}p_2^{a_2}\cdots p_r^{a_r})+1+\left\lbrace\begin{array}{lr}
 0 & 2\leq a_1\leq 3\\
 -2^r & 4\leq a_1
\end{array}\right..
\end{equation*}
\end{theorem}

\begin{proof}
First let $a_1=2$. We calculate $QC(2p_2^{a_2}\cdots p_r^{a_r})$ in terms of $QC(2^2p_2^{a_2}\cdots p_r^{a_r})$ by extending each signature type for $C_{2p_2^{a_2}\cdots p_r^{a_r}}$ to a corresponding signature for $C_{2^2p_2^{a_2}\cdots p_r^{a_r}}$ by multiplying certain periods by 2 (this is described later in the proof). 

The admissible signatures for $C_{2p_2^{a_2}\cdots p_r^{a_r}}$ are given below as two types: type 1 and type 2 signatures.
\begin{align*}
&\text{(Type 1)}\qquad (p_2^{k_2}\cdots p_r^{k_r},2p_2^{a_2}\cdots p_r^{a_r},2p_2^{a_2}\cdots p_r^{a_r}),\\
&\text{(Type 2)}\qquad (2p_2^{k_2}\cdots p_r^{k_r},p_2^{a_2}\cdots p_r^{a_r},2p_2^{a_2}\cdots p_r^{a_r}).
\end{align*} 
The ranges for type 1 signatures are $0\leq k_i\leq a_i$, except for when each $k_i=0$. The ranges for type 2 signatures are $0\leq k_i\leq a_i$, except for when each $k_i=a_i$. The $T$-values of these types is then
\begin{align*}
&T(p_2^{k_2}\cdots p_r^{k_r},2p_2^{a_2}\cdots p_r^{a_r},2p_2^{a_2}\cdots p_r^{a_r})\\
&\quad=\left\lbrace\begin{array}{lr}
 \displaystyle\frac{1}{2}\prod_{i=2}^rf(p_i^{k_i}) &\quad 1\leq k_i\leq a_i-1 \text{ for some } i\\
 \displaystyle\frac{1}{2}\prod_{i=2}^rf(p_i^{k_i})+\frac{1}{2} &\quad \text{ otherwise}
\end{array}\right.,\\
&T(2p_2^{k_2}\cdots p_r^{k_r},p_2^{a_2}\cdots p_r^{a_r},2p_2^{a_2}\cdots p_r^{a_r})=\prod_{i=2}^rf(p_i^{k_i}).
\end{align*}

We derive the claimed formula first when $a_1=2$ by extending both type 1 and type 2 signatures.

Consider type 1 signatures. We can extend type 1 signatures to signatures for $C_{2^2p_2^{a_2}\cdots p_r^{a_r}}$ by multiplying either all periods by 2 or multiplying the second and third periods by 2. This gives us the following two extended signatures and respective ranges for $C_{2^2p_2^{a_2}\cdots p_r^{a_r}}$:
\begin{align}
&(p_2^{k_2}\cdots p_r^{k_r},2^2p_2^{a_2}\cdots p_r^{a_r},2^2p_2^{a_2}\cdots p_r^{a_r}), \qquad 0\leq k_i\leq a_i, \text{ not all }k_i=a_i,\label{recursive1a}\\
&(2p_2^{k_2}\cdots p_r^{k_r},2^{2}p_2^{a_2}\cdots p_r^{a_r},2^2p_2^{a_2}\cdots p_r^{a_r}), \qquad 0\leq k_i\leq a_i\label{recursive1b}.
\end{align}
Except for when each $k_i=0$ for $i=2,\ldots,r$, Theorem \ref{thm:ben} implies that the $T$-values of the above signatures \eqref{recursive1a} and \eqref{recursive1b} both coincide with the $T$-value of a type 1 signature. Specifically,
\begin{align}\label{recursivex}
T(p_2^{k_2}\cdots p_r^{k_r},2^2p_2^{a_2}\cdots p_r^{a_r},2^2p_2^{a_2}\cdots p_r^{a_r})&=T(2p_2^{k_2}\cdots p_r^{k_r},2^2p_2^{a_2}\cdots p_r^{a_r},2^2p_2^{a_2}\cdots p_r^{a_r})\nonumber\\
&=T(p_2^{k_2}\cdots p_r^{k_r},2p_2^{a_2}\cdots p_r^{a_r},2p_2^{a_2}\cdots p_r^{a_r}).
\end{align}
If each $k_i=0$, then we have one extra contribution from the extended signature $(2,2^2p_2^{a_2}\cdots p_r^{a_r},2^2p_2^{a_2}\cdots p_r^{a_r})$:
\begin{equation*}
T(2,2^2p_2^{a_2}\cdots p_r^{a_r},2^2p_2^{a_2}\cdots p_r^{a_r})=\frac{1}{2}\big(1+\phi(2)\big)=1.
\end{equation*}
Therefore, the sum over all $k_i$ of the $T$-values of the two extended signatures \eqref{recursive1a} and \eqref{recursive1b} is
\begin{align}
\sum_{k_2,\ldots,k_r}\Big\lbrace&T(p_2^{k_2}\cdots p_r^{k_r},2^2p_2^{a_2}\cdots p_r^{a_r},2^2p_2^{a_2}\cdots p_r^{a_r})\nonumber\\
&\quad +T(2p_2^{k_2}\cdots p_r^{k_r},2^2p_2^{a_2}\cdots p_r^{a_r},2^2p_2^{a_2}\cdots p_r^{a_r})\Big\rbrace\nonumber\\
&=2\cdot\sum_{k_2,\ldots,k_r}\Big(T(p_2^{k_2}\cdots p_r^{k_r},2p_2^{a_2}\cdots p_r^{a_r},2p_2^{a_2}\cdots p_r^{a_r})\Big)+1\label{recursivea}
\end{align}

Now consider type 2 signatures. Extending Type 2 signatures as in the previous type yields two extended signatures with their respective ranges for $C_{2^2p_2^{a_2}\cdots p_r^{a_r}}$:
\begin{align}
&(2^2p_2^{k_2}\cdots p_r^{k_r},p_2^{a_2}\cdots p_r^{a_r},2^2p_2^{a_2}\cdots p_r^{a_r}), \qquad 0\leq k_i\leq a_i, \text{ not all }k_i=a_i,\label{recursive1c}\\
&(2^2p_2^{k_2}\cdots p_r^{k_r},2p_2^{a_2}\cdots p_r^{a_r},2^2p_2^{a_2}\cdots p_r^{a_r}), \qquad 0\leq k_i\leq a_i \text{ not all }k_i=a_i \label{recursive1d}.
\end{align}
The $T$-values of these extended signatures agree with the $T$-values of their associated signature for $C_{2p_2^{a_2}\cdots p_r^{a_r}}$, that is,
\begin{align}\label{recursivexx}
T(2^2p_2^{k_2}\cdots p_r^{k_r},p_2^{a_2}\cdots p_r^{a_r},2^2p_2^{a_2}\cdots p_r^{a_r})&=T(2^2p_2^{k_2}\cdots p_r^{k_r},2p_2^{a_2}\cdots p_r^{a_r},2^2p_2^{a_2}\cdots p_r^{a_r})\nonumber\\
&=T(2p_2^{k_2}\cdots p_r^{k_r},p_2^{a_2}\cdots p_r^{a_r},2p_2^{a_2}\cdots p_r^{a_r}).
\end{align}
Therefore, the sum over all $k_i$ of the two extended signatures \eqref{recursive1c} and \eqref{recursive1d} gives
\begin{align}
\sum_{k_2,\ldots,k_r}\Big\lbrace&T(2^2p_2^{k_2}\cdots p_r^{k_r},p_2^{a_2}\cdots p_r^{a_r},2^2p_2^{a_2}\cdots p_r^{a_r})\nonumber\\
&\quad +T(2^2p_2^{k_2}\cdots p_r^{k_r},2p_2^{a_2}\cdots p_r^{a_r},2^2p_2^{a_2}\cdots p_r^{a_r})\Big\rbrace\nonumber\\
&=2\cdot\sum_{k_2,\ldots,k_r}\Big(T(2p_2^{k_2}\cdots p_r^{k_r},p_2^{a_2}\cdots p_r^{a_r},2p_2^{a_2}\cdots p_r^{a_r})\Big)\label{recursiveb}.
\end{align}
Adding together \eqref{recursivea} and \eqref{recursiveb}, we conclude that
\begin{align*}
QC(2^2p_2^{a_2}\cdots p_r^{a_r})&=2 QC(2p_2^{a_2}\cdots p_r^{a_r})+1,
\end{align*}
which is the desired formula when $a_1=2$.

\vspace{2mm}

Next, let $a_1=3$. We consider first those signatures with two periods that are equal. By \eqref{recursivex} above, we know the $T$-values for the admissible signatures for $C_{2^2p_2^{a_2}\cdots p_r^{a_r}}$ having two equal periods. On the other hand, there are three admissible signature forms for $C_{2^3p_2^{a_2}\cdots p_r^{a_r}}$ with two equal periods. They are as below, with their respective $k_i$ ranges:
\begin{align*}
&(p_2^{k_2}\cdots p_r^{k_r},2^3p_2^{a_2}\cdots p_r^{a_r},2^3p_2^{a_2}\cdots p_r^{a_r}),\quad 0\leq k_i\leq a_i, \text{ not all }k_i=0,\\
&(2p_2^{k_2}\cdots p_r^{k_r},2^3p_2^{a_2}\cdots p_r^{a_r},2^3p_2^{a_2}\cdots p_r^{a_r}), \quad 0\leq k_i\leq a_i,\\
&(2^2p_2^{k_2}\cdots p_r^{k_r},2^3p_2^{a_2}\cdots p_r^{a_r},2^3p_2^{a_2}\cdots p_r^{a_r}), \quad 0\leq k_i\leq a_i.
\end{align*}
Calculating their $T$-values in terms of the $T$-values of their associated signatures for $C_{2^2p_2^{a_2}\cdots p_r^{a_r}}$,
\begin{align}
&T(p_2^{k_2}\cdots p_r^{k_r},2^3p_2^{a_2}\cdots p_r^{a_r},2^3p_2^{a_2}\cdots p_r^{a_r})\nonumber\\
&\quad=T(p_2^{k_2}\cdots p_r^{k_r},2^2p_2^{a_2}\cdots p_r^{a_r},2^2p_2^{a_2}\cdots p_r^{a_r}),\label{recursivey1}\\
&T(2p_2^{k_2}\cdots p_r^{k_r},2^3p_2^{a_2}\cdots p_r^{a_r},2^3p_2^{a_2}\cdots p_r^{a_r})\nonumber\\
&\quad=T(2p_2^{k_2}\cdots p_r^{k_r},2^2p_2^{a_2}\cdots p_r^{a_r},2^2p_2^{a_2}\cdots p_r^{a_r}),\label{recursivey2}\\
&T(2^2p_2^{k_2}\cdots p_r^{k_r},2^3p_2^{a_2}\cdots p_r^{a_r},2^3p_2^{a_2}\cdots p_r^{a_r})\nonumber\\
&\quad=\left\lbrace\begin{array}{lr}
 2\cdot T(p_2^{k_2}\cdots p_r^{k_r},2^2p_2^{a_2}\cdots p_r^{a_r},2^2p_2^{a_2}\cdots p_r^{a_r}) & k_i\neq 0 \text{ for some }i\\
 \phi(4) & k_i=0 \text{ for all }i
\end{array}\right..\label{recursiveyy}
\end{align}
Looking closely at the case when $k_i=0$ for each $i=2,\ldots,r$ in \eqref{recursiveyy}, we have
\begin{align}\label{recursivey3}
T(2^2,2^3p_2^{a_2}\cdots p_r^{a_r},2^3p_2^{a_2}\cdots p_r^{a_r})&=\phi(4)\nonumber\\
&=2\cdot T(2,2^2p_2^{a_2}\cdots p_r^{a_r},2^2p_2^{a_2}\cdots p_r^{a_r})+1.
\end{align}

Next, consider signatures for $C_{2^3p_2^{a_2}\cdots p_r^{a_r}}$ having all distinct periods. By \eqref{recursivexx} above, we know the $T$-values of the admissible signatures for $C_{2^2p_2^{a_2}\cdots p_r^{a_r}}$ having all distinct periods. In fact, they all equal $\prod_{i=2}^rf(p_i^{k_i})$. On the other hand, we have three admissible signatures for $C_{2^3p_2^{a_2}\cdots p_r^{a_r}}$ having distinct peiods. They are as below, with all their $k_i$ parameters satisfying $0\leq k_i\leq a_i$ for $i=2,\ldots,r$, but not all $k_i=a_i$.
\begin{align*}
&(2^3p_2^{k_2}\cdots p_r^{k_r},p_2^{a_2}\cdots p_r^{a_r},2^3p_2^{a_2}\cdots p_r^{a_r}),\\
&(2^3p_2^{k_2}\cdots p_r^{k_r},2p_2^{a_2}\cdots p_r^{a_r},2^3p_2^{a_2}\cdots p_r^{a_r}),\\
&(2^3p_2^{k_2}\cdots p_r^{k_r},2^2p_2^{a_2}\cdots p_r^{a_r},2^3p_2^{a_2}\cdots p_r^{a_r}).
\end{align*}
Calculating their $T$-values,
\begin{align}
&T(2^3p_2^{k_2}\cdots p_r^{k_r},p_2^{a_2}\cdots p_r^{a_r},2^3p_2^{a_2}\cdots p_r^{a_r})\nonumber\\
&\quad=T(2^2p_2^{k_2}\cdots p_r^{k_r},p_2^{a_2}\cdots p_r^{a_r},2^3p_2^{a_2}\cdots p_r^{a_r})\label{recursivey4}\\
&T(2^3p_2^{k_2}\cdots p_r^{k_r},2p_2^{a_2}\cdots p_r^{a_r},2^3p_2^{a_2}\cdots p_r^{a_r})\nonumber\\
&\quad=T(2^2p_2^{k_2}\cdots p_r^{k_r},2p_2^{a_2}\cdots p_r^{a_r},2^2p_2^{a_2}\cdots p_r^{a_r})\label{recursivey5}\\
&T(2^3p_2^{k_2}\cdots p_r^{k_r},2^2p_2^{a_2}\cdots p_r^{a_r},2^3p_2^{a_2}\cdots p_r^{a_r})\nonumber\\
&\quad=\phi(4)\cdot T(2^2p_2^{k_2}\cdots p_r^{k_r},p_2^{a_2}\cdots p_r^{a_r},2^2p_2^{a_2}\cdots p_r^{a_r})\nonumber\\
&\quad=2\cdot T(2^2p_2^{k_2}\cdots p_r^{k_r},p_2^{a_2}\cdots p_r^{a_r},2^2p_2^{a_2}\cdots p_r^{a_r})\label{recursivey6}.
\end{align}
We can therefore combine all terms from \eqref{recursivey1}-\eqref{recursivey6} to conclude
\begin{equation*}
QC(2^3p_2^{a_2}\cdots p_r^{a_r})=2\cdot QC(2^2p_2^{a_2}\cdots p_r^{a_r})+1.
\end{equation*}

Finally, let $a_1\geq 4$. We have two types of admissible signatures for $C_{2^{a_1}p_2^{a_2}\cdots p_r^{a_r}}$, which we label as types I and II, depending on whether two periods are equal or not.
\begin{align*}
&\text{(Type I)}\qquad (2^{k_1}p_2^{k_2}\cdots p_r^{k_r},2^{a_1}p_2^{a_2}\cdots p_r^{a_r},2^{a_1}p_2^{a_2}\cdots p_r^{a_r}),\\
&\text{(Type II)}\qquad (2^{a_1}p_2^{k_2}\cdots p_r^{k_r},2^{k_1}p_2^{a_2}\cdots p_r^{a_r},2^{a_1}p_2^{a_2}\cdots p_r^{a_r}).
\end{align*}
As before, each integer $k_i$ represents the parameter for each exponent $p_i$, where $2\leq i\leq r$. The ranges for type I signatures are $0\leq k_i\leq a_i$ for $2\leq i\leq r$and $0\leq k_1\leq a_1-1$. We exclude the case when $k_i=0$ for $i=1,2,\ldots,r$. The ranges for type II signatures are $0\leq k_i\leq a_i-1$ for each $1\leq i\leq r$. 

For a fixed $r$-tuple of $k_i$'s and a fixed signature type I or II, we calculate that signature's $T$-value. Then, we sum over all allowable ranges for the $k_i$ to obtain $QC(2^{a_1}p_2^{a_2}\cdots p_r^{a_r})$ written in terms of $QC(2^{a_1-1}p_2^{a_2}\cdots p_r^{a_r})$.

\underline{For Type I Signatures}: We will denote
\begin{equation*}
\hat{T}:=\frac{1}{2}\prod_{i=2}^rf(p_i^{k_i}),
\end{equation*}
which is a quantity dependent on $k_2,\ldots,k_r$. The $T$-values of type I signatures depend on whether or not $1\leq k_i\leq a_i-1$ for some $i=2,\ldots,r$. We consider these cases separately.
\begin{itemize}
\item If $1\leq k_i\leq a_i-1$ for some $2\leq i\leq r$, then
\begin{equation*}
T(2^{k_1}p_2^{k_2}\cdots p_r^{k_r},2^{a_1}p_2^{a_2}\cdots p_r^{a_r},2^{a_1}p_2^{a_2}\cdots p_r^{a_r})=\phi(2^{k_1})\cdot\hat{T}.
\end{equation*}
Summing over $k_1$ for $0\leq k_1\leq a_1-1$ and using \eqref{phi},
\begin{align}\label{recursive1}
\sum_{k_1=0}^{a_1-1}&T(2^{k_1}p_2^{k_2}\cdots p_r^{k_r},2^{a_1}p_2^{a_2}\cdots p_r^{a_r},2^{a_1}p_2^{a_2}\cdots p_r^{a_r})\nonumber\\
&=\hat{T}\cdot\sum_{k_1=0}^{a_1-1}\phi(2^{k_1})\nonumber\\
&=\hat{T}\cdot\left(\sum_{k_1=0}^{a_1-2}\big(\phi(2^{k_1})\big)+\phi(2^{a_1-1})\right)\nonumber\\
&=\hat{T}\cdot 2\sum_{k_1=0}^{a_1-2}\phi(2^{k_1}).
\end{align}
\item Suppose now all $k_i$ do not satisfy $1\leq k_i\leq a_i-1$, where $2\leq i\leq r$. The $T$-values of the respective signatures depends additionally on whether or not all $k_i=0$, for $2\leq i\leq r$.

For $2\leq i\leq r$, suppose all $k_i=0$. This implies $k_1\neq 0$, as otherwise, one of the periods of this signature is 1, contrary to Harvey's Theorem. Then,
\begin{align*}
T(2^{k_1}p_2^{0}\cdots p_r^{0},&2^{a_1}p_2^{a_2}\cdots p_r^{a_r},2^{a_1}p_2^{a_2}\cdots p_r^{a_r})\\
&=\left\lbrace\begin{array}{lr}
 1 & k_1=1\\
 \frac{1}{2}\phi(2^{k_1}) & 2\leq k_1\leq a_1-2\\
 \frac{1}{2}\phi(2^{k_1})+1 & k_1=a_1-1.
\end{array}\right..
\end{align*}
Summing over $k_1$,
\begin{align}\label{recursive2}
T(&2^{k_1}p_2^{0}\cdots p_r^{0},2^{a_1}p_2^{a_2}\cdots p_r^{a_r},2^{a_1}p_2^{a_2}\cdots p_r^{a_r})\nonumber\\
&=1+\frac{1}{2}\left(\phi(4)+\phi(8)+\cdots+\phi(2^{a_1-2})\right)+\frac{1}{2}\phi(2^{a_1-1})+1\nonumber\\
&=2+\frac{1}{2}\sum_{k_1=2}^{a_1-1}\phi(2^{k_1})\nonumber\\
&=2+\frac{1}{2}(2^{a_1-1}-2)\nonumber\\
&=2^{a_1-2}+1.
\end{align}

Now suppose not all $k_i=0$ for $2\leq i\leq r$. Then,
\begin{align*}
T(2^{k_1}p_2^{k_2}\cdots p_r^{k_r},&2^{a_1}p_2^{a_2}\cdots p_r^{a_r},2^{a_1}p_2^{a_2}\cdots p_r^{a_r})\\
&=\left\lbrace\begin{array}{lr}
 \hat{T}+\frac{1}{2} & k_1=0,1\\
 \phi(2^{k_1})\hat{T} & 2\leq k_1\leq a_1-2\\
 \frac{1}{2}\hat{T}+1 & k_1=a_1-1.
\end{array}\right..
\end{align*}
Next, we sum over $k_1$. The extra terms of $1/2$ and $1$ which appear in the $T$-values above will be added $2^{r-1}-1$ times. This is because there are $2^{r-1}-1$ possible $r-1$-tuples of either 0's or 1's, excluding the case of all $k_i=0$. Thus,
\begin{align}\label{recursive3}
\sum_{k_1=0}^{a_1-1}&T(2^{k_1}p_2^{k_2}\cdots p_r^{k_r},2^{a_1}p_2^{a_2}\cdots p_r^{a_r},2^{a_1}p_2^{a_2}\cdots p_r^{a_r})\nonumber\\
&=\hat{T}+\frac{1}{2}(2^{r-1}-1)+\hat{T}+\frac{1}{2}(2^{r-1}-1)\nonumber\\
&\quad+\phi(4)\hat{T}+\phi(8)\hat{T}+\cdots \phi(2^{a_1-2})\hat{T}+\phi(2^{a_1-1})\hat{T}+2^{r-1}-1\nonumber\\
&=2^r-2+\hat{T}\cdot\sum_{k=0}^{a_1-1}\phi(2^k)\nonumber\\
&=2^r-2+\hat{T}\cdot 2\sum_{k=0}^{a_1-2}\phi(2^k).
\end{align}

\end{itemize}

\underline{For Type II Signatures}: For a fixed $r$-tuple of $k_i$'s with $0\leq k_i\leq a_i-1$ and $1\leq i\leq r$,
\begin{equation*}
T(2^{a_1}p_2^{k_2}\cdots p_r^{k_r},2^{k_1}p_2^{a_2}\cdots p_r^{a_r},2^{a_1}p_2^{a_2}\cdots p_r^{a_r})=\phi(2^{k_1})\cdot\prod_{i=2}^rf(p_i^{k_i}).
\end{equation*}
Then summing over $k_1$ and using \eqref{phi},
\begin{align}\label{recursive4}
\sum_{k_1=0}^{a_1-1}&T(2^{a_1}p_2^{k_2}\cdots p_r^{k_r},2^{k_1}p_2^{a_2}\cdots p_r^{a_r},2^{a_1}p_2^{a_2}\cdots p_r^{a_r})\nonumber\\
&=\prod_{i=2}^rf(p_i^{k_i})\cdot\sum_{k_1=0}^{a_1-1}\phi(2^{k_1})\nonumber\\
&=\prod_{i=2}^rf(p_i^{k_i})\cdot\left(\sum_{k_1=0}^{a_1-2}\big(\phi(2^{k_1})\big)+\phi(2^{a_1-1})\right)\nonumber\\
&=\prod_{i=2}^rf(p_i^{k_i})\cdot2\sum_{k_1=0}^{a_1-2}\phi(2^{k_1}).
\end{align}

Therefore, the total sum over all $k_i$ with $i=2,\ldots,r$ for all possible admissible signatures for $C_{2^{a_1}p_2^{a_2}\cdots p_r^{a_r}}$ is the sum of the quantities \eqref{recursive1}-\eqref{recursive4} over the parameters $k_2,\ldots,k_r$, which yields
\begin{equation*}
QC(2^{a_1}p_2^{a_2}\cdots p_r^{a_r})=2\left(\sum_{k_2,\ldots,k_r}\hat{T}\right)+2^{r}+2^{a_1-2}-1.
\end{equation*}
On the other hand, repeating the entire set of calculations for $C_{2^{a_1-1}p_2^{a_2}\cdots p_r^{a_r}}$ instead and summing over all $k_i$ for $i=2,\ldots,r$ yields
\begin{equation*}
QC(2^{a_1-1}p_2^{a_2}\cdots p_r^{a_r})=\left(\sum_{k_2,\ldots,k_r}\hat{T}\right)+2^r+2^{a_1-3}-1.
\end{equation*}
We can thus conclude
\begin{equation*}
QC(2^{a_1}p_2^{a_2}\cdots p_r^{a_r})=2\cdot QC(2^{a_1-1}p_2^{a_2}\cdots p_r^{a_r})+1-2^r.
\end{equation*}

\end{proof}

With the additional use of the recursive formula Theorem \ref{thm:recursiveqc}, we can finally prove Theorem \ref{thm:qceven} in general.

\begin{proof}[Proof of Theorem \ref{thm:qceven}]
Using our recursive formula Theorem \ref{thm:recursiveqc}, we calculate 
\begin{equation*}
QC(2^{a_1}p_2^{a_2}\cdots p_r^{a_r})
\end{equation*}
when $a_1\geq 1$. The case when $a_1=1$ was done in the previous section to this paper, which showed that
\begin{equation*}
QC(2p_2^{a_2}\cdots p_r^{a_r})=\frac{1}{2}\prod_{i=2}^rp_i^{a_i-1}(p_i+1)-1+2^{r-2}.
\end{equation*}
Recall Theorem \ref{thm:recursiveqc}:
\begin{equation*}
QC(2^{a_1}p_2^{a_2}\cdots p_r^{a_r})=2\cdot QC(2^{a_1-1}p_2^{a_2}\cdots p_r^{a_r})+1+\left\lbrace\begin{array}{lr}
 0 & 2\leq a_1\leq 3\\
 -2^r & 4\leq a_1
\end{array}\right..
\end{equation*}

First let $a_1=2$. Then,
\begin{align*}
QC(2^{2}p_2^{a_2}\cdots p_r^{a_r})&=2\cdot QC(2p_2^{a_2}\cdots p_r^{a_r})+1\\
&=2\left(\prod_{i=2}^rp_i^{a_i-1}(p_i+1)-1+2^{r-2}\right)+1\\
&=2\prod_{i=2}^rp_i^{a_i-1}(p_i+1)-1+2^{r-1}.
\end{align*}

Now let $a_1=3$. Then,
\begin{align*}
QC(2^{3}p_2^{a_2}\cdots p_r^{a_r})&=2\cdot QC(2^2p_2^{a_2}\cdots p_r^{a_r})+1\\
&=2\left(2\prod_{i=2}^rp_i^{a_i-1}(p_i+1)-1+2^{r-1}\right)+1\\
&=2^2\prod_{i=2}^rp_i^{a_i-1}(p_i+1)-1+2^{r}.
\end{align*}

Next, let $a_1\geq 4$. We proceed by induction on $a_1$. The base case is when $a_1=4$. So,
\begin{align*}
QC(2^{4}p_2^{a_2}\cdots p_r^{a_r})&=2\cdot QC(2^3p_2^{a_2}\cdots p_r^{a_r})+1-2^r\\
&=2\left(2^2\prod_{i=2}^rp_i^{a_i-1}(p_i+1)-1+2^{r}\right)+1-2^r\\
&=2^3\prod_{i=2}^rp_i^{a_i-1}(p_i+1)-1+2^{r},
\end{align*}
as desired. This completes the base case. Now assume the result holds true up until some $a_1\geq 4$. Then
\begin{align*}
QC(2^{a_1}p_2^{a_2}\cdots p_r^{a_r})&=2\cdot QC(2^{a_1-1}p_2^{a_2}\cdots p_r^{a_r})+1-2^r\\
&=2\left(2^{a_1-2}\prod_{i=2}^rp_i^{a_i-1}(p_i+1)-1+2^{r}\right)+1-2^r\\
&=2^{a_1-1}\prod_{i=2}^rp_i^{a_i-1}(p_i+1)-1+2^{r}.
\end{align*}
This completes the inductive proof.

We have therefore shown that
\begin{equation*}
QC(2^{a_1}p_2^{a_2}\cdots p_r^{a_r})=2^{a_1-2}\left(\prod_{i=2}^rp_i^{a_i-1}(p_i+1)\right)-1+\left\lbrace\begin{array}{lr}
 2^{r-2} & a_1=1\\
 2^{r-1} & a_1=2\\
 2^{r} & a_1\geq 3
\end{array}\right..
\end{equation*}

\end{proof}

\section{Proof of Theorem \ref{thm:qcodd}: $QC(n)$ for Odd $n$}\label{sec:qcodd}

The form of $QC(n)$ for odd $n$ depends on the arithmetic properties of $n$. The structure of the proof is split into the four cases: for $n=\prod_{i=1}^r p_i^{a_i}$ and $p_i$ odd, distinct primes, we consider when either
\begin{enumerate}
\item $p_i\equiv 1\modd 6$ for $1\leq i\leq r$;
\item $p_1=3, a_1=1$ and $p_i\equiv 1\modd 6$ for $2\leq i\leq r$;
\item $p_1=3, a_1\geq 2$ and $p_i\equiv 1\modd 6$ for $2\leq i\leq r$;
\item or $p_i\equiv 5\modd 6$ for some $1\leq i\leq r$.
\end{enumerate}

\subsection{When $p_i\equiv 1\modd 6$ for $1\leq i\leq r$}

We now prove the main theorem for $QC(n)$ when each prime dividing $n$ is one modulo six, which we call the third case of Theorem \ref{thm:qcodd}. We use a proof by induction, with the base case given first below. Recall the fact that
\begin{equation*}
\tau_2(p^a)=\left\lbrace\begin{array}{lr}
 0 &\quad p\equiv 5\modd 6, a\geq 1 \text{ or } p=3, a\geq 2\\
 1 &\quad p=3, a=1\\
 2 &\quad p\equiv 1\modd 6, a\geq 1
\end{array}\right..
\end{equation*}

\begin{theorem}\label{lem:onemodsix}
If $p\equiv 1\modd 6$ and $a\geq 1$ is an integer, then
\begin{equation*}
QC(p^a)=\frac{1}{6}p^{a-1}(p+1)+\frac{2}{3}.
\end{equation*}
\end{theorem}

\begin{proof}
By Harvey's Theorem, the only admissible signatures for $p^a$ are of the form
\begin{equation*}
(p^k,p^a,p^a),\quad 1\leq k\leq a.
\end{equation*}
The $T$-values are then
\begin{align*}
T(p^k,p^a,p^a)&=\left\lbrace\begin{array}{lr}
 \frac{1}{2}f(p^k) & 1\leq k\leq a-1\\
 \frac{1}{6}\left(3+2\cdot\tau_2(p^a)+p^{a-1}(p-2)\right) & k=a
\end{array}\right.\\
&=\left\lbrace\begin{array}{lr}
 \frac{1}{2}p^{k-1}(p-1) & 1\leq k\leq a-1\\
 \frac{7}{6}+\frac{1}{6}p^{a-1}(p-2) & k=a
\end{array}\right..
\end{align*}
Summing over $k$, we conclude that
\begin{align*}
QC(p^a)=\sum_{k=1}^aT(p^k,p^a,p^a)&=\frac{1}{2}(p-1)\left(\sum_{k=1}^{a-1}p^{k-1}\right)+\frac{7}{6}+\frac{1}{6}p^{a-1}(p-2)\\
&=\frac{1}{2}(p^{a-1}-1)+\frac{7}{6}+\frac{1}{6}p^{a-1}(p-2)\\
&=\frac{1}{6}p^{a-1}(p+1)+\frac{2}{3}.
\end{align*}

\end{proof}

For subsequent proofs, we make the following convention to ease notation: for a given extended signature for $C_{nq}$ of the form $(aq^x,bq^y,cq^z)$ for odd prime $q$ not dividing $n$ and nonnegative integers $x,y,z$, let $T:=T(a,b,c)$ be the $T$-value of the signature $(a,b,c)$ for $C_n$.

\begin{theorem}\label{thm:onemodsixrecursion}
Suppose $n=\prod_{i=1}^r p_i^{a_i}$ for odd primes $p_i\equiv 1\modd 6$. Then the following recursive formula holds:
\begin{equation*}
QC(n\cdot p_{r+1}^{a_{r+1}})=\left(QC(n)+1-\frac{5}{3}\cdot 2^{r-1}\right)p_{r+1}^{a_{r+1}-1}\left(p_{r+1}+1\right)-1+\frac{5}{3}\cdot 2^r.
\end{equation*}
\end{theorem}

\begin{proof}
We prove this by induction on the number $r$ of odd primes. The base case is when $r=1$. Let $n=p^aq^b$ for odd primes $p$ and $q$ congruent to one mod six, and positive integers $a$ and $b$. We show that
\begin{equation*}
QC(p^aq^b)=\left(QC(p^a)-\frac{2}{3}\right)q^{b-1}(q+1)+\frac{7}{3}.
\end{equation*}
We extend the admissible signatures for $C_{p^a}$ to signatures for $C_{p^aq^b}$. The only such signatures for $C_{p^a}$ are of the form $(p^k,p^a,p^a)$ for $1\leq k\leq a-1$. Their $T$-values are given in the previous lemma above. Extending these signatures to signatures for $C_{p^aq^b}$, there are the following two types
\begin{align*}
&(\text{Type } 1) \quad (p^kq^h,p^aq^b,p^aq^b) \quad 0\leq k\leq a, 0\leq h\leq b, k\neq h\neq 0 \text{ simultaneously},\\
&(\text{Type } 2) \quad (p^kq^b,p^aq^h,p^aq^b) \quad 0\leq k\leq a-1, 0\leq h\leq b-1.
\end{align*}

For fixed values of $k$, we calculate the $T$-values of both signature types 1 and 2 above, and then sum over the parameter $h$. We have three cases for values of $k$ for extended signatures: $k=0$, $1\leq k\leq a-1$, or $k=a$.

We focus on type 1 signatures first.

\underline{When $k=0$}: The $T$-values are
\begin{equation*}
T(q^h,p^aq^b,p^aq^b)=\left\lbrace\begin{array}{lr}
 \frac{1}{2}f(q^h) & 1\leq h\leq b-1\\
 \frac{1}{2}f(q^b)+\frac{1}{2} & h=b
\end{array}\right..
\end{equation*}
These $T$-values are not related to a signature for $C_{p^a}$, since $k=0$ is not valid for the signatures $(p^k,p^a,p^a)$ of $C_{p^a}$ by Harvey's Theorem.

\underline{When $1\leq k\leq a-1$}: The $T$-values are
\begin{equation*}
T(p^kq^h,p^aq^b,p^aq^b)=\frac{1}{2}f(p^k)f(q^h)=T(p^k,p^a,p^a)\cdot f(q^h)=T\cdot f(q^h),
\end{equation*}
for $0\leq h\leq b$.

\underline{When $k=a$}: The $T$-values are
\begin{equation*}
T(p^aq^h,p^aq^b,p^aq^b)=\left\lbrace\begin{array}{lr}
 \frac{1}{2}f(p^a)+\frac{1}{2} & h=0\\
 \frac{1}{2}f(p^a)f(q^h) & 1\leq h\leq b-1\\
 \frac{1}{6}(3+2\cdot\tau_2(p^aq^b)+f(p^a)f(q^b)) & h=b
\end{array}\right..
\end{equation*}
Recall that $T:=T(p^a,p^a,p^a)=\frac{7}{6}+\frac{1}{6}f(p^a)$, and $\tau_2(p^aq^b)=2^2=4$ because $p$ and $q$ are one modulo six. Then we can rewrite the above $T$-value in terms of the $T$-value for the corresponding original signature for $C_{p^a}$. In particular,
\begin{equation*}
T(p^aq^h,p^aq^b,p^aq^b)=\left\lbrace\begin{array}{lr}
 3\left(T-\frac{7}{6}\right)+\frac{1}{2} & h=0\\
 3\left(T-\frac{7}{6}\right)\cdot f(q^h) & 1\leq h\leq b-1\\
 3\left(T-\frac{7}{6}\right)\cdot f(q^b)+\frac{11}{6} & h=b
\end{array}\right..
\end{equation*}

Therefore, summing the above $T$-values over $h$ in the case of type 1 signatures, we have
\begin{align}
&\left(\sum_{h=1}^b\frac{1}{2}f(q^h)\right)+\frac{1}{2}+T(p^k,p^a,p^a)\left(\sum_{h=0}^bf(q^h)\right)+3\left(T(p^a,p^a,p^a)-\frac{7}{6}\right)+\frac{1}{2}\nonumber\\
&\quad \left(\sum_{h=1}^{b-1}3\left(T(p^a,p^a,p^a)-\frac{7}{6}\right)\right)+\left(T(p^a,p^a,p^a)-\frac{7}{6}\right)f(q^b)+\frac{11}{6}\nonumber\\
&=\frac{1}{2}q^{b-1}+T(p^k,p^a,p^a)\cdot q^{b-1}(q-1)+\left(T(p^a,p^a,p^a)-\frac{7}{6}\right)\cdot q^{b-1}(q+1)+\frac{7}{3}\label{mod6type1}
\end{align}

Now we consider the type 2 signatures $(p^kq^b,p^aq^h,p^aq^b)$ with $0\leq k\leq a-1$ and $0\leq h\leq b-1$. Their $T$-values are
\begin{equation*}
T(p^kq^b,p^aq^h,p^aq^b)=f(p^k)f(q^h)=\left\lbrace\begin{array}{lr}
 f(q^h) & k=0\\
 2T\cdot f(q^h) & 1\leq k\leq a-1
\end{array}\right.,
\end{equation*}
where $T:=T(p^k,p^a,p^a)$. Summing over $h$, we have
\begin{equation}\label{mod6type2}
\left(\sum_{h=0}^{b-1}f(q^h)\right)+2T\cdot \left(\sum_{h=0}^{b-1}f(q^h)\right)=(2T+1)q^{b-1}.
\end{equation}

If we sum over the parameter $k$ in both \eqref{mod6type1} and \eqref{mod6type2}, we will have
\begin{align*}
QC(p^aq^b)&=\frac{1}{2}q^{b-1}(q-1)+\left(\sum_{k=1}^{a-1}T(p^k,p^a,p^a)\right)\cdot(2q^{b-1}+q^{b-1}(q-1))\\
&\quad+q^{b-1}+\left(T(p^a,p^a,p^a)-\frac{7}{6}\right)\cdot q^{b-1}(q+1)+\frac{7}{3}\\
&=\left(\sum_{k=1}^aT(p^k,p^a,p^a)\right)\cdot q^{b-1}(q+1)+q^{b-1}+\frac{1}{2}q^{b-1}(q-1)\\
&\quad-\frac{7}{6}q^{b-1}(q+1)+\frac{7}{3}\\
&=\left(QC(p^a)-\frac{2}{3}\right)q^{b-1}(q+1)+\frac{7}{3}.
\end{align*}
This completes the base case when the number of odd primes $r$ is one.

Now assume the statement holds for some $r\geq 2$. Signatures for $C_{p_1^{a_1}\cdots p_r^{a_r}}$ are of the form
\begin{equation*}
(p_1^{k_1}\cdots p_r^{k_r},p_1^{h_1}\cdots p_r^{h_r},p_1^{\ell_1}\cdots p_r^{\ell_r})
\end{equation*}
for nonnegative integers $k_i,h_i,\ell_i$. By Harvey's Theorem, exactly two of $k_i,h_i$, or $\ell_i$ must equal $a_i$ (otherwise, it violates the least common multiple conditions). After renaming, we will have parameters $k_1,\ldots,k_r$ as the exponents of the $r$ primes $p_1,\ldots,p_r$, and these exponents may be distributed over the three periods. The exponents of all other primes, in each period, will then be maximized at $a_i$. Since the $T$-values of a signature depend on whether two periods agree or not, we consider two cases of where we may place $k_1,\ldots,k_r$:
\begin{enumerate}
\item If $h_i=\ell_i=a_i$ for $1\leq i\leq r$, the signatures are of the form
\begin{equation*}
(p_1^{k_1}\cdots p_r^{k_r},p_1^{a_1}\cdots p_r^{a_r},p_1^{a_1}\cdots p_r^{a_r});
\end{equation*}
That is, all $k_i$ are the exponents of the $p_1,\ldots,p_r$ primes in the first period. These signatures will be called type (i) signatures.
\item If not all $h_i=\ell_i=a_i$ for $1\leq i\leq r$, then there exists subindices $r_1,\ldots,r_{m_1}, s_1,\ldots,s_{m_2},t_1,\ldots,t_{m_3}$ such that the primes in the first period with subindices $r_j$ have exponents $k_{r_j}$; the primes in the second period  with subindices $s_j$ have exponents $k_{s_j}$; and  the primes in the third period with subindices $t_j$ have exponents $k_{t_j}$. These subindices together form a partition of the subindices $1,\ldots,r$ into three parts. This case occurs when not all $k_i$ are the exponents of the $p_1,\ldots,p_r$ primes in the first period. We suppress the primes $p_i$ in the signature for this case, and simply write the exponents. As described above, once a ranging parameter $k_i$ is present on a prime $p_i$ in one period, all other exponents of $p_i$ in the other two periods must equal $a_i$. We simplify the notation by not including these $a_i$ in the period of a signature of this type. Thus, signatures of this case are of the form
\begin{equation*}
(k_{r_1}\cdots k_{r_{m_1}},k_{s_1}\cdots k_{s_{m_2}},k_{t_1}\cdots k_{t_{m_3}}),
\end{equation*}
and will be called type (ii) signatures.
\end{enumerate}

We now extend the signatures for $C_{p_1^{a_1}\cdots p_r^{a_r}}$ of the two types above to signatures for $C_{p_1^{a_1}\cdots p_{r+1}^{a_{r+1}}}$  by multiplying each period by $p_{r+1}$ to some power $t$. Then we calculate the $T$-values of these extended signatures in terms of their previous signatures, and sum over $k_{r+1}$. We will write these $T$-values in terms of their previous signatures for $C_{p_1^{a_1}\cdots p_r^{a_r}}$. The $T$-values of an original signature for a fixed tuple $k_1,\ldots k_r$, following the previous convention, will be simply written as $T$.

\underline{Extensions for Signatures of Type (i)}: Extend the Type (i) signatures 
\begin{equation*}
(p_1^{k_1}\cdots p_r^{k_r},p_1^{a_1}\cdots p_r^{a_r},p_1^{a_1}\cdots p_r^{a_r})
\end{equation*}
to two types 1a and 1b of signatures for $C_{p_1^{a_1}\cdots p_{r+1}^{a_{r+1}}}$, with their respective ranges noted:
\begin{align*}
&(\text{1a})\quad (p_1^{k_1}\cdots p_r^{k_r}p_{r+1}^{k_{r+1}},p_1^{a_1}\cdots p_r^{a_r}p_{r+1}^{a_{r+1}},p_1^{a_1}\cdots p_r^{a_r}p_{r+1}^{a_{r+1}}),\\&\quad\quad 0\leq k_i\leq a_i \text{ with } k_i\neq 0 \text{ simultaneously for } 1\leq i\leq r+1;\\
&(\text{1b})\quad (p_1^{k_1}\cdots p_r^{k_r}p_{r+1}^{a_{r+1}},p_1^{a_1}\cdots p_r^{a_r}p_{r+1}^{k_{r+1}},p_1^{a_1}\cdots p_r^{a_r}p_{r+1}^{a_{r+1}}),\\&\quad\quad 0\leq k_i\leq a_i \text{ with } k_i\neq a_i \text{ simultaneously for } 1\leq i\leq r,\\
&\quad\quad 0\leq k_{r+1}\leq a_{r+1}-1.
\end{align*}

\underline{Extensions for Signatures of Type (ii)}: Extend the type (ii) signatures
\begin{equation*}
(k_{r_1}\cdots k_{r_{m_1}},k_{s_1}\cdots k_{s_{m_2}},k_{t_1}\cdots k_{t_{m_3}}).
\end{equation*}
to three types 2a, 2b, and 2c of signatures for $C_{p_1^{a_1}\cdots p_{r+1}^{a_{r+1}}}$, with their respective ranges noted:
\begin{align*}
&(\text{2a})\quad (k_{r_1}\cdots k_{r_{m_1}}k_{r+1},k_{s_1}\cdots k_{s_{m_2}},k_{t_1}\cdots k_{t_{m_3}}),\\&\quad\quad 0\leq k_{r_i}\leq a_i-1,0\leq k_{s_i}\leq a_i-1, 0\leq k_{t_i}\leq a_i-1 ,\\&\quad\quad 0\leq k_{r+1}\leq a_{r+1};\\
&(\text{2b})\quad (k_{r_1}\cdots k_{r_{m_1}},k_{s_1}\cdots k_{s_{m_2}}k_{r+1},k_{t_1}\cdots k_{t_{m_3}}),\\&\quad\quad 0\leq k_{i}\leq a_i-1 \text{ for all subindices};\\
&(\text{2c})\quad (k_{r_1}\cdots k_{r_{m_1}},k_{s_1}\cdots k_{s_{m_2}},k_{t_1}\cdots k_{t_{m_3}}k_{r+1}),\\&\quad\quad 0\leq k_{i}\leq a_i-1 \text{ for all subindices}.
\end{align*}

Given the five types of extended signatures above, we compute their $T$-values for a fixed $r$-tuple $k_1,\ldots,k_r$. The $T$-value of any signature depends on whether some $k_i=0, k_i=a_i$ or otherwise. We consider these cases separately, starting with the simplest case when some $k_i$ falls in mid-range.

\underline{If $1\leq k_i\leq a_i-1$ for some $1\leq i\leq r$}: The $T$-value of a type (i) signature for $p_1^{a_1}\cdots p_r^{a_r}$ in this case is
\begin{equation*}
T:=T(p_1^{k_1}\cdots p_r^{k_r},p_1^{a_1}\cdots p_r^{a_r},p_1^{a_1}\cdots p_r^{a_r})=\frac{1}{2}\prod_{i=1}^rf(p_i^{k_i}).
\end{equation*}
The $T$-value of the 1a extension is then, for $0\leq k_{r+1}\leq a_{r+1}$,
\begin{equation}\label{modsixmid1a}
T(p_1^{k_1}\cdots p_r^{k_r}p_r^{k_{r+1}},p_1^{a_1}\cdots p_r^{a_r}p_{r+1}^{a_{r+1}},p_1^{a_1}\cdots p_r^{a_r}p_{r+1}^{a_{r+1}})=\frac{1}{2}\prod_{i=1}^{r+1}f(p_i^{k_i})=T\cdot f(p_{r+1}^{k_{r+1}}).
\end{equation}
Also, the $T$-value of the 1b extension is, for $0\leq k_{r+1}\leq a_{r+1}-1$,
\begin{equation}\label{modsixmid1b}
T(p_1^{k_1}\cdots p_r^{k_r}p_r^{a_{r+1}},p_1^{a_1}\cdots p_r^{a_r}p_{r+1}^{k_{r+1}},p_1^{a_1}\cdots p_r^{a_r}p_{r+1}^{a_{r+1}})=\prod_{i=1}^{r+1}f(p_i^{k_i})=2T\cdot f(p_{r+1}^{k_{r+1}}).
\end{equation}

Next, the $T$-value of a type (ii) signature for $p_1^{a_1}\cdots p_r^{a_r}$ in this case is
\begin{equation*}
T:=T(k_{r_1}\cdots k_{r_{m_1}},k_{s_1}\cdots k_{s_{m_2}},k_{t_1}\cdots k_{t_{m_3}})=\prod_{i=1}^rf(p_i^{k_i}).
\end{equation*}
The $T$-value of the 2a extension is then, for $0\leq k_{r+1}\leq a_{r+1}$,
\begin{equation}\label{modsixmid2a}
T(k_{r_1}\cdots k_{r_{m_1}}k_{r+1},k_{s_1}\cdots k_{s_{m_2}},k_{t_1}\cdots k_{t_{m_3}})=\prod_{i=1}^{r+1}f(p_i^{k_i})=T\cdot f(p_{r+1}^{k_{r+1}}).
\end{equation}
Also, the $T$-value of both the 2b and 2c extensions agree. Their value, for $0\leq k_{r+1}\leq a_{r+1}-1$, is
\begin{equation}\label{modsixmid2bc}
\prod_{i=1}^{r+1}f(p_i^{k_i})=T\cdot f(p_{r+1}^{k_{r+1}}).
\end{equation}
We can then sum over the $T$-values of the extensions 1a and 1b, given in \eqref{modsixmid1a} and \eqref{modsixmid1b}. This yields
\begin{align}\label{modsixtotal1}
T\cdot&\left(\sum_{k_{r+1}=0}^{a_{r+1}}f(p_{r+1}^{k_{r+1}})\right)+2T\cdot\left(\sum_{k_{r+1}=0}^{a_{r+1}-1}f(p_{r+1}^{k_{r+1}})\right)\nonumber\\
&=T\cdot p_{r+1}^{a_{r+1}-1}(p_{r+1}-1)+2T\cdot p_{r+1}^{a_{r+1}-1}\nonumber\\
&=T\cdot p_{r+1}^{a_{r+1}-1}(p_{r+1}+1).
\end{align}
A similar calculation gives us the sum over the $T$-values of the extensions 2a, 2b, and 2c. Summing the quantity \eqref{modsixmid2a} and twice the quantity in \eqref{modsixmid2bc} (to count 2b and 2c signature $T$-values), we have
\begin{equation}\label{modsixtotal2}
T\cdot\left(\sum_{k_{r+1}=0}^{a_{r+1}}f(p_{r+1}^{k_{r+1}})\right)+2T\cdot\left(\sum_{k_{r+1}=0}^{a_{r+1}-1}f(p_{r+1}^{k_{r+1}})\right)=T\cdot p_{r+1}^{a_{r+1}-1}(p_{r+1}+1).
\end{equation}
This completes the case when $1\leq k_i\leq a_i-1$ for some $1\leq i\leq r$.

\underline{If $k_i=0$ for all $1\leq i\leq r$}: Signatures for this case will either not have any corresponding previous signature for which they were extended from or have a previous $T$-value of one (for 2a, 2b, or 2c signatures). As such, the following $T$-value calculations do not involve any previous $T$-values, except for the 2a, 2b, and 2c signatures.

The $T$-value of the 1a extensions are
\begin{equation}\label{modsix01a}
T(p_{r+1}^{k_{r+1}},p_1^{a_1}\cdots p_{r+1}^{a_{r+1}},p_1^{a_1}\cdots p_{r+1}^{a_{r+1}})=\left\lbrace\begin{array}{lr}
 \frac{1}{2}f(p_{r+1}^{k_{r+1}}) & 1\leq k_{r+1}\leq a_{r+1}-1\\
 \frac{1}{2}f(p_{r+1}^{a_{r+1}})+\frac{1}{2} & k_{r+1}=a_{r+1}
\end{array}\right..
\end{equation}
The $T$-values of the 1b extensions are, for $0\leq k_{r+1}\leq a_{r+1}-1$,
\begin{equation}\label{modsix01b}
T(p_{r+1}^{a_{r+1}},p_1^{a_1}\cdots p_{r+1}^{k_{r+1}},p_1^{a_1}\cdots p_{r+1}^{a_{r+1}})=f(p_{r+1}^{k_{r+1}}).
\end{equation}
The sum over $k_{r+1}$ for the 1a and 1b signatures given in \eqref{modsix01a} and \eqref{modsix01b} is thus
\begin{align}\label{modsixtotal3}
\frac{1}{2}&\left(\sum_{k_{r+1}=1}^{a_{r+1}}f(p_{r+1}^{k_{r+1}})\right)+\frac{1}{2}+\sum_{k_{r+1}=1}^{a_{r+1}-1}f(p_{r+1}^{k_{r+1}})\nonumber\\
&=\frac{1}{2}(p_{r+1}^{a_{r+1}-1}-1)+\frac{1}{2}+p_{r+1}^{a_{r+1}-1}+p_{r+1}^{a_{r+1}-1}(p_{r+1}-2)\nonumber\\
&=\frac{1}{2}p_{r+1}^{a_{r+1}-1}(p_{r+1}+1).
\end{align}

Next, for type (ii) signatures for $C_{p_1^{a_1}\cdots p_{r}^{a_{r}}}$, since all periods are distinct with all ranging parameters $k_i=0$, it follows that the $T$-value is simply one. We write this $T$-value as $T=1$. Now, the $T$-value for each of the type 2a, 2b, and 2c extended signatures for $C_{p_1^{a_1}\cdots p_{r+1}^{a_{r+1}}}$ are all $f(p_{r+1}^{k_{r+1}})$. The ranges here are $0\leq k_{r+1}\leq a_{r+1}$ for the 2a signatures, and $0\leq k_{r+1}\leq a_{r+1}-1$ for the 2b and 2c signatures. Summing these $T$-values over $k_{r+1}$ yields
\begin{equation}\label{modsixtotal4}
T\cdot\left(\sum_{k_{r+1}=0}^{a_{r+1}}f(p_{r+1}^{k_{r+1}})\right)+2T\cdot\left(\sum_{k_{r+1}=0}^{a_{r+1}-1}f(p_{r+1}^{k_{r+1}})\right)=T\cdot p_{r+1}^{a_{r+1}-1}(p_{r+1}+1).
\end{equation}
This completes the case when $k_i=0$ for all $1\leq i\leq r$.

\underline{If $k_i=a_i$ for all $1\leq i\leq r$}: The $T$-value for type (i) signatures in this case when all ranging parameters $k_i$ are maximal is given as
\begin{align*}
T:=T(p_1^{a_1}\cdots p_r^{a_{r}}&, p_1^{a_1}\cdots p_{r}^{a_{r}},p_1^{a_1}\cdots p_{r}^{a_{r}})\\
&=\frac{1}{6}\left(3+2\cdot \tau_2(p_1^{a_1}\cdots p_{r+1}^{a_{r+1}})+\prod_{i=1}^rf(p_i^{a_i})\right)\\
&=\frac{1}{2}+\frac{1}{3}\cdot 2^r+\frac{1}{6}\prod_{i=1}^r f(p_i^{a_i}).
\end{align*}
The $T$-value for the 1a extensions are
\begin{align*}
&T(p_1^{a_1}\cdots p_{r+1}^{k_{r+1}}, p_1^{a_1}\cdots p_{r+1}^{a_{r+1}},p_1^{a_1}\cdots p_{r+1}^{a_{r+1}})\\
&\quad=\left\lbrace\begin{array}{lr}
 \frac{1}{2}\left(\prod_{i=1}^rf(p_i^{a_i})\right)+\frac{1}{2} & k_{r+1}=0\\
 \frac{1}{2}\left(\prod_{i=1}^rf(p_i^{a_i})\right)\cdot f(p_{r+1}^{k_{r+1}}) & 1\leq k_{r+1}\leq a_{r+1}-1\\
 \frac{1}{2}+\frac{1}{3}\cdot 2^{r+1}+\frac{1}{6}\prod_{i=1}^{r+1}f(p_i^{a_i}) & k_{r+1}=a_{r+1}
\end{array}\right.\\
&\quad=\left\lbrace\begin{array}{lr}
 \left(3T-\frac{3}{2}-2^r\right)+\frac{1}{2} & k_{r+1}=0\\
 \left(3T-\frac{3}{2}-2^r\right)\cdot f(p_{r+1}^{k_{r+1}}) & 1\leq k_{r+1}\leq a_{r+1}-1\\
 \frac{1}{2}+\frac{1}{3}\cdot 2^{r+1}+\left(T-\frac{1}{2}-\frac{1}{3}\cdot 2^r\right)f(p_{r+1}^{a_{r+1}}) & k_{r+1}=a_{r+1}
\end{array}\right..
\end{align*}
There are no other $T$-value of extended signatures possible in this case (see the possible ranges for the extended signatures 1b, 2a, 2b, and 2c). Hence, the sum over $k_{r+1}$ of extended $T$-values for this case is
\begin{align}\label{modsixtotal5}
\left(3T-\frac{3}{2}-2^r\right)&\left(\sum_{k_{r+1}=0}^{a_{r+1}-1}f(p_{r+1}^{k_{r+1}})\right)+\frac{1}{2}+\left(T-\frac{1}{2}-\frac{1}{3}\cdot 2^r\right)f(p_{r+1}^{a_{r+1}})+\frac{1}{2}+\frac{1}{3}\cdot 2^{r+1}\nonumber\\
&=\left(T-\frac{1}{2}-\frac{1}{3}\cdot 2^r\right)\cdot \left(3p_{r+1}^{a_{r+1}-1}+p_{r+1}^{a_{r+1}-1}(p_{r+1}-2)\right)+1+\frac{1}{3}\cdot 2^{r+1}\nonumber\\
&=\left(T-\frac{1}{2}-\frac{1}{3}\cdot 2^r\right)\cdot \left(p_{r+1}^{a_{r+1}-1}(p_{r+1}+1)\right)+1+\frac{1}{3}\cdot 2^{r+1}.
\end{align}
This completes the case when $k_i=a_i$ for all $1\leq i\leq r$.

\underline{If, for $1\leq i\leq r$, at least one $k_i\neq 0$, no $k_i$ satisfy $1\leq k_i\leq a_i-1$, and not all }\\\underline{$k_i=a_i$}: Of the ranging parameters $k_1,\ldots,k_r$ satisfying the above condition, consider the subindices $x_1,\ldots,x_m$ of $1,\ldots,r$ such that $k_{x_i}=a_i$. The fact that $k_i\neq 0$ for some $i$ and no $k_i$ lies in a mid-range implies that such subindices do exist.

The $T$-value of signatures of type (i) are
\begin{equation*}
T:=\frac{1}{2}\prod_{i=1}^mf(p_{x_i}^{k_{x_i}})+\frac{1}{2}.
\end{equation*}
The $T$-value of the 1a extensions are then
\begin{align}\label{modsixx1}
&\left\lbrace\begin{array}{lr}
 \frac{1}{2}\prod_{i=1}^mf(p_{x_i}^{k_{x_i}})+\frac{1}{2} & k_{r+1}=0\\
 \frac{1}{2}\left(\prod_{i=1}^mf(p_{x_i}^{k_{x_i}})\right)\cdot f(p_{r+1}^{k_{r+1}}) & 1\leq k_{r+1}\leq a_{r+1}-1\\
 \frac{1}{2}\left(\prod_{i=1}^mf(p_{x_i}^{k_{x_i}})\right)\cdot f(p_{r+1}^{a_{r+1}})+\frac{1}{2} & k_{r+1}=a_{r+1}
\end{array}\right.\nonumber\\
&=\left\lbrace\begin{array}{lr}
 T & k_{r+1}=0\\
 \left(T-\frac{1}{2}\right)\cdot f(p_{r+1}^{k_{r+1}}) & 1\leq k_{r+1}\leq a_{r+1}-1\\
  \left(T-\frac{1}{2}\right)\cdot f(p_{r+1}^{a_{r+1}})+\frac{1}{2} & k_{r+1}=a_{r+1}\\
\end{array}\right..
\end{align}
Also, the $T$-value of the 1b extensions are, for $0\leq k_{r+1}\leq a_{r+1}-1$,
\begin{equation}\label{modsixx2}
\left(\prod_{i=1}^mf(p_{x_i}^{k_{x_i}})\right)\cdot f(p_{r+1}^{k_{r+1}})=(2T-1)\cdot f(p_{r+1}^{k_{r+1}}).
\end{equation}
Summing over all $k_{r+1}$ given above in \eqref{modsixx1} and \eqref{modsixx2},
\begin{align}\label{modsixtotal6}
T+&\left(T-\frac{1}{2}\right)\cdot\left(\sum_{k_{r+1}=1}^{a_{r+1}}f(p_{r+1}^{k_{r+1}})\right)+\frac{1}{2}+\left(2T-1\right)\cdot\left(\sum_{k_{r+1}=0}^{a_{r+1}-1}f(p_{r+1}^{k_{r+1}})\right)\nonumber\\
&=T+\left(T-\frac{1}{2}\right)\cdot(p_{r+1}^{a_{r+1}-1}(p_{r+1}-1)-1)+\frac{1}{2}+\left(2T-1\right)\cdot p_{r+1}^{a_{r+1}-1}\nonumber\\
&=\left(T-\frac{1}{2}\right)p_{r+1}^{a_{r+1}-1}(p_{r+1}+1)+1.
\end{align}
This completes the final case. 

\underline{Total sum}: Now we collect all terms \eqref{modsixtotal1},\eqref{modsixtotal2},\eqref{modsixtotal3},\eqref{modsixtotal4},\eqref{modsixtotal5}, and\eqref{modsixtotal6} calculated above, and sum over all possible $k_1,\ldots, k_r$. In the sum, we factor $p_{r+1}^{a_{r+1}-1}(p_{r+1}+1)$ from all the collected terms whenever present. We also let $\sum_{k_1,\ldots,k_r} T$ be the sum of all $T$-values for signatures for $C_{p_1^{a_1}\cdots p_r^{a_r}}$. By definition,
\begin{equation*}
QC(p_1^{a_1}\cdots p_r^{a_r})=\sum_{k_1,\ldots,k_r}T.
\end{equation*}
There are $2^r-2$ signatures of the final form presented above, that is, where each $k_i$ is either $0$ or $a_i$, and cannot satisfy $1\leq k_i\leq a_i-1$, nor can all $k_i=a_i$. Each of these signatures has $T$-value given in \eqref{modsixtotal6}. Finally, we can now sum over $k_1,\ldots,k_r$ to obtain
\begin{align*}
QC(p_1^{a_1}\cdots p_{r+1}^{a_{r+1}})&=\left(\sum_{k_1\ldots,k_r}T-\frac{1}{2}(2^r-2)-\frac{1}{3}\cdot 2^r\right)\cdot p_{r+1}^{a_{r+1}-1}(p_{r+1}+1)\\
&\quad+1+\frac{1}{3}\cdot 2^{r+1}+2^r-2\\
&=\left(QC(p_1^{a_1}\cdots p_{r}^{a_{r}})+1-\frac{5}{3}\cdot 2^{r-1}\right)\cdot p_{r+1}^{a_{r+1}-1}(p_{r+1}+1)\\
&\quad -1+\frac{5}{3}\cdot 2^r.
\end{align*}

\end{proof}

We are now in position to prove Theorem \ref{thm:qcodd} when each prime $p_i$ dividing $n$ is one modulo six.

\begin{proof}[Proof of Third Case of Theorem \ref{thm:qcodd}]
Let $n=\prod_{i=1}^rp_i^{a_i}$ with each $p_i\equiv 1\modd 6$. We prove
\begin{equation*}
QC(n)=\frac{1}{6}\prod_{i=1}^{r}p_i^{a_i-1}(p_i+1)-1+\frac{5}{3}\cdot 2^{r-1}
\end{equation*}
by induction on the number $r$ of odd primes dividing $n$. The base case is when $r=1$, which is shown in Theorem \ref{lem:onemodsix}.

Now suppose the statement to be proven holds true for some $r\geq 1$. We show the statement holds for $n\cdot p_{r+1}^{a_{r+1}}$ for some odd prime $p_{r+1}\equiv 1\modd 6$ which is not a factor of $n$, and $a_{r+1}$ is a positive integer. Then using Theorem \ref{thm:onemodsixrecursion}, we have
\begin{align*}
QC(p_1^{a_1}\cdots p_{r+1}^{a_{r+1}})&=\left(QC(p_1^{a_1}\cdots p_{r}^{a_{r}})+1-\frac{5}{3}\cdot 2^{r-1}\right)\cdot p_{r+1}^{a_{r+1}-1}(p_{r+1}+1)\\
&\quad -1+\frac{5}{3}\cdot 2^r\\
&=\left(\frac{1}{6}\prod_{i=1}^{r}p_i^{a_i-1}(p_i+1)\right)\cdot p_{r+1}^{a_{r+1}-1}(p_{r+1}+1)-1+\frac{5}{3}\cdot 2^r.
\end{align*}
We have shown the statement true for $n\cdot p_{r+1}^{a_{r+1}}$. This completes the inductive step. Therefore, for any integer $n$ of the form $n=\prod_{i=1}^rp_i^{a_i}$ and each $p_i\equiv 1\modd 6$,
\begin{equation*}
QC(n)=\frac{1}{6}\prod_{i=1}^{r}p_i^{a_i-1}(p_i+1)-1+\frac{5}{3}\cdot 2^{r-1}.
\end{equation*}
\end{proof}


\subsection{When $p_1=3, a_1=1$ and $p_i\equiv 1\modd 6$ for $2\leq i\leq r$}

We now proceed to derive the formula for $QC(n)$ when $n=\prod_{i=1}^r p_i^{a_i}$ with $p_1=3, a_1=1$ and $p_i\equiv 1\modd 6$ for $2\leq i\leq r$. As in the previous proofs, we incorporate a recursive formula first.

\begin{theorem}\label{thm:recursive3n}
Suppose $n=\prod_{i=1}^r p_i^{a_i}$ where $p_i\equiv 1\modd 6$. Then the following recursive formula holds:
\begin{align*}
QC(3n)=4QC(n)+3-2^{r+1}.
\end{align*}
\end{theorem}

\begin{proof}
The main technique is to extend admissible signatures for $C_n$ to admissible signatures for $C_{3^an}$ by multiplying each period by an appropriate power of 3.

We first prove $QC(3n)=4QC(n)+3-2^{r+1}$. Recall the two signature types for $n$, using the notation from the proof of Theorem \ref{thm:onemodsixrecursion}:
\begin{center}
\begin{align*}
&\text{Type (i) } (p_1^{k_1}\cdots p_r^{k_r},p_1^{a_1}\cdots p_r^{a_r},p_1^{a_1}\cdots p_r^{a_r}),\\
&\text{Type (ii) } (k_{r_1}\cdots k_{r_{m_1}},k_{s_1}\cdots k_{s_{m_2}},k_{t_1}\cdots k_{t_{m_3}}).
\end{align*}
\end{center}
Then their respective $T$-values are
\begin{align}
T(p_1^{k_1}\cdots p_r^{k_r},&p_1^{a_1}\cdots p_r^{a_r},p_1^{a_1}\cdots p_r^{a_r})\nonumber\\
&=\left\lbrace\begin{array}{lr}
 \frac{1}{2}\prod_{i=1}^r f(p_i^{k_i}) & 1\leq k_i\leq a_i-1\\
  & \text{ for some } 1\leq i\leq r\\
  & \\
 \frac{1}{2}+\frac{1}{3}\cdot 2^r+\frac{1}{6}\prod_{i=1}^r f(p_i^{a_i}) & k_i=a_i\\
 & \text{ for all } 1\leq i\leq r\\
 & \\
 \frac{1}{2}\prod_{i=1}^mf(p_{x_i}^{k_{x_i}})+\frac{1}{2} & k_{x_i}=a_i,\\
 & \text{ not all }k_i=a_i,\\
 & \text{ not all } k_i=0,\\
 & \text{ no } k_i \text{ satisfy } 1\leq k_i\leq a_i-1
\end{array}\right.,\label{3.7.1}\\
T(k_{r_1}\cdots k_{r_{m_1}},&k_{s_1}\cdots k_{s_{m_2}},k_{t_1}\cdots k_{t_{m_3}})=\prod_{i=1}^rf(p_i^{k_i})\label{3.7.2}.
\end{align}

We now calculate the $T$-values of the extensions of these two types of signatures. As before, we label the two possible extension types stemming from type (i) above as types 1a and 1b, while the three extensions coming from type (ii) will be labeled 2a, 2b, and 2c. Specifically, these extended signatures are as follows.

\underline{Extensions for Signatures of Type (i)}: Extend the type (i) signatures 
\begin{equation*}
(p_1^{k_1}\cdots p_r^{k_r},p_1^{a_1}\cdots p_r^{a_r},p_1^{a_1}\cdots p_r^{a_r})
\end{equation*}
to the following two types 1a and 1b of signatures for $C_{3p_1^{a_1}\cdots p_{r}^{a_{r}}}$, with their respective ranges noted:
\begin{align*}
&(\text{1a})\quad (3^kp_1^{k_1}\cdots p_r^{k_r},3p_1^{a_1}\cdots p_r^{a_r},3p_1^{a_1}\cdots p_r^{a_r}),\\&\quad\quad 0\leq k_i\leq a_i \text{ with } k_i\neq 0 \text{ simultaneously for all } 1\leq i\leq r,\\
&\quad\quad k=0 \text{ or } k=1;\\
&(\text{1b})\quad (3p_1^{k_1}\cdots p_r^{k_r},p_1^{a_1}\cdots p_r^{a_r},p_1^{a_1}\cdots p_r^{a_r}),\\&\quad\quad 0\leq k_i\leq a_i \text{ with } k_i\neq a_i \text{ simultaneously for all } 1\leq i\leq r.
\end{align*}

\underline{Extensions for Signatures of type (ii)}: Extend the type (ii) signatures
\begin{equation*}
(k_{r_1}\cdots k_{r_{m_1}},k_{s_1}\cdots k_{s_{m_2}},k_{t_1}\cdots k_{t_{m_3}}).
\end{equation*}
to the following three types 2a, 2b, and 2c of signatures for $C_{3p_1^{a_1}\cdots p_{r}^{a_{r}}}$. The only restriction on the $k$ parameters is that no two periods can have all maximal prime powers, for otherwise, you would have a signature of type 1a.
\begin{align*}
&(\text{2a})\quad (k\cdot k_{r_1}\cdots k_{r_{m_1}},k_{s_1}\cdots k_{s_{m_2}},k_{t_1}\cdots k_{t_{m_3}}),\\&\quad\quad k=0 \text{ or } k=1;\\
&(\text{2b})\quad (k_{r_1}\cdots k_{r_{m_1}},k\cdot k_{s_1}\cdots k_{s_{m_2}},k_{t_1}\cdots k_{t_{m_3}}),\\&\quad\quad k=0;\\
&(\text{2c})\quad (k_{r_1}\cdots k_{r_{m_1}},k_{s_1}\cdots k_{s_{m_2}},k\cdot k_{t_1}\cdots k_{t_{m_3}}),\\&\quad\quad k=0.
\end{align*}

We calculate the $T$-values of these extended signatures, and compare them to the $T$-values of the signatures they were extended from.

The $T$-value of a type 1a signature for $C_{3n}$ is
\begin{equation*}
\left\lbrace\begin{array}{lr}
 \frac{1}{2}\prod_{i=1}^r f(p_i^{k_i})\cdot f(3^k) & 1\leq k_i\leq a_i-1 \text{ for some } i,\\
 & k=0 \text{ or } k=1\\
 & \\
 \frac{1}{2}f(3)+\frac{1}{2} & k_i=0 \text{ for all } i\\
 & \\
 \frac{1}{2}\prod_{i=1}^{r}f(p_i^{a_i})+\frac{1}{2} & k_i=a_i \text{ for all } i,\\
 & k=0\\
 & \\
 \frac{1}{2}+\frac{1}{3}\cdot 2^r+\frac{1}{6}\prod_{i=1}^{r}f(p_i^{a_i}) & k_i=a_i \text{ for all } i,\\
 & k=1\\
 & \\
 \frac{1}{2}\prod_{i=1}^mf(p_{x_i}^{k_{x_i}})\cdot f(3^k)+\frac{1}{2} & k_{x_i}=a_i,\\
 & \text{ not all }k_i=a_i,\\
 & \text{ not all } k_i=0,\\
 & \text{ no } k_i \text{ satisfy } 1\leq k_i\leq a_i-1,\\
 & k=0 \text{ or } k=1
\end{array}\right.. 
\end{equation*}
By the definition of $f(x)$, we know $f(3^k)=1$ for $k=0$ and $k=1$. Recall that we use $T$ to denote the $T$-value of the pre-extended signature for $C_{n}$. Comparing each $T$-value of a type 1a signature for $C_{3n}$ above with the corresponding $T$-values for $C_n$ found in \eqref{3.7.1}, we see that the $T$-values of the 1a signatures are given by
\begin{equation}\label{3.7.3}
\left\lbrace\begin{array}{lr}
 T & 1\leq k_i\leq a_i-1 \text{ for some } i,\\
 & k=0 \text{ or } k=1\\
 & \\
 1 & k_i=0 \text{ for all } i\\
 & \\
 3\left(T-\frac{1}{2}-\frac{1}{3}\cdot 2^r\right)+\frac{1}{2} & k_i=a_i \text{ for all } i,\\
 & k=0\\
 & \\
 T & k_i=a_i \text{ for all } i,\\
 & k=1\\
 & \\
 T & k_{x_i}=a_i,\\
 & \text{ not all }k_i=a_i,\\
 & \text{ not all } k_i=0,\\
 & \text{ no } k_i \text{ satisfy } 1\leq k_i\leq a_i-1,\\
 & k=0 \text{ or } k=1
\end{array}\right.. 
\end{equation}

The $T$-value of a type 1b signature for $C_{3n}$ is
\begin{equation*}
\prod_{i=1}^rf(p_i^{k_i}).
\end{equation*}
We compare this value with the $T$-value for signature of type (i), considering each case carefully. This gives the $T$-value of 1b signatures for $C_{3n}$ in terms of the $T$-values of type (i) for $C_n$ as
\begin{equation}\label{3.7.5}
\prod_{i=1}^rf(p_i^{k_i})=\left\lbrace\begin{array}{lr}
 2T & 1\leq k_i\leq a_i-1 \text{ for some } i\\
 1 & k_i=0 \text{ for all } i\\
 2T-1 & \text{otherwise}
\end{array}\right..
\end{equation}

The type 2a, 2b, and 2c signatures are much simpler to analyze, because the $T$-value of type (ii) signatures for $C_n$ (for fixed $r$-tuple $k_1,\ldots,k_r)$ is given in \eqref{3.7.2} as $T:=\prod_{i=1}^r f(p_i^{k_i})$. Then the $T$-value of a type 2a signature for $C_{3n}$ is
\begin{equation}\label{3.7.4}
T(k\cdot k_{r_1}\cdots k_{r_{m_1}},k_{s_1}\cdots k_{s_{m_2}},k_{t_1}\cdots k_{t_{m_3}})=\prod_{i=1}^rf(p_i^{k_i})\cdot f(3^k)=\prod_{i=1}^rf(p_i^{k_i})=T.
\end{equation}
We note for later use that the above $T$-value occurs when $k=0$ and when $k=1$. That is, we will need to count this $T$-value twice.

Both of the remaining signatures types 2b and 2c have the same $T$-value, namely $\prod_{i=1}^rf(p_i^{k_i})=T$. We will count this $T$-value twice in the following total summation.

\underline{Total summation}: Using \eqref{3.7.3}, \eqref{3.7.5}, \eqref{3.7.4}, and the remark regarding 2b and 2c signatures in the previous paragraph, we count total sums over each possible combination of $k_1,\ldots,k_r$. The total of all $T$-values of type (i) extensions is
\begin{equation}\label{3.7.6}
\left\lbrace\begin{array}{lr}
 4T & 1\leq k_i\leq a_i-1 \text{ for some } i\\
 2 & k_i=0 \text{ for all } i\\
 4T-\frac{3}{2}-2^r+\frac{1}{2} & k_i=a_i \text{ for all } i\\
 4T-1 & \text{otherwise}
\end{array}\right..
\end{equation}
The total of all $T$-values of type (ii) extensions is $4T$. For the final case of \eqref{3.7.6} above, there are $2^r-2$ possible choices of $r$-tuples $k_1,\ldots,k_r$ such that at least one $k_i\neq 0$ and not all $k_i=a_i$. Therefore, taking the sum over all $T$-values of admissible signatures $(n_1',n_2',n_3')$ for $C_{3n}$ in terms of the $T$-values of signatures $(n_1,n_2,n_3)$ for $C_n$ is
\begin{align*}
QC(3n)&=\sum_{(n_1',n_2',n_3')}T(n_1',n_2',n_3')\\
&=4\cdot\sum_{(n_1,n_2,n_3)}T(n_1,n_2,n_3)+2-1-2^r-(2^r-2)\\
&=4\cdot QC(n)+3-2^{r+1}.
\end{align*}

\end{proof}

Because we have already derived $QC(n)$ when $n$ is divisible by primes congruent to one mod six, we can easily prove the formula for $QC(3n)$.

\begin{proof}[Proof of Second Case of Theorem \ref{thm:qcodd}]
We have shown before that, for $n=\prod_{i=1}^rp_i^{a_i}$ with $p_i\equiv 1\modd 6$ for all $i$,
\begin{equation*}
QC(n)=\frac{1}{6}\prod_{i=1}^rp_i^{a_i-1}(p_i+1)-1+\frac{5}{3}\cdot 2^{r-1}.
\end{equation*}
Then by Theorem \ref{thm:recursive3n},
\begin{align*}
QC(3n)&=4\cdot QC(n)+3-2^{r+1}\\
&=\frac{1}{6}\prod_{i=1}^r4\cdot p_i^{a_i-1}(p_i+1)-4+\frac{5}{3}\cdot 2^{r+1}+3-2^{r+1}\\
&=\frac{1}{6}\prod_{i=1}^r4\cdot p_i^{a_i-1}(p_i+1)-1+\frac{4}{3}\cdot 2^{r},
\end{align*}
as desired.
\end{proof}


\subsection{When $p_1=3, a_1\geq 2$ and $p_i\equiv 1\modd 6$ for $2\leq i\leq r$}

We can now prove the $QC(3^an)$ formula of Theorem \ref{thm:qcodd} for $a\geq 2$ and $n$ divisible by primes congruent to one mod six.

\begin{theorem}\label{thm:recursive3an}
Suppose $n=\prod_{i=1}^r p_i^{a_i}$ where $p_i\equiv 1\modd 6$. Then the following recursive formula holds:
\begin{align*}
QC(3^an)&=4\cdot 3^{a-1}\cdot QC(n)+4\cdot 3^{a-1}-1+(1-10\cdot 3^{a-2})\cdot 2^r, \text{ for }a\geq 2.
\end{align*}
\end{theorem}

\begin{proof}
We follow the same method and notation as in the proof of $QC(3n)$. For the two types of signatures for $n$,
\begin{center}
\begin{enumerate}
\item $(p_1^{k_1}\cdots p_r^{k_r},p_1^{a_1}\cdots p_r^{a_r},p_1^{a_1}\cdots p_r^{a_r}),$
\item $(k_{r_1}\cdots k_{r_{m_1}},k_{s_1}\cdots k_{s_{m_2}},k_{t_1}\cdots k_{t_{m_3}}),$
\end{enumerate}
\end{center}
we extend to signatures of types 1a, 1b, 2a, 2b, and 2c for $C_{3^an}$ as before:
\begin{align*}
&(\text{1a})\quad (3^kp_1^{k_1}\cdots p_r^{k_r},3^ap_1^{a_1}\cdots p_r^{a_r},3^ap_1^{a_1}\cdots p_r^{a_r}),\\&\quad\quad 0\leq k_i\leq a_i \text{ with } k_i\neq 0 \text{ and } k\neq 0 \text{ simultaneously for all } 1\leq i\leq r,\\
&\quad\quad 0\leq k\leq a;\\
&(\text{1b})\quad (3^ap_1^{k_1}\cdots p_r^{k_r},3^kp_1^{a_1}\cdots p_r^{a_r},3^ap_1^{a_1}\cdots p_r^{a_r}),\\&\quad\quad 0\leq k_i\leq a_i \text{ with } k_i\neq a_i \text{ simultaneously for all } 1\leq i\leq r,\\
&\quad\quad 0\leq k\leq a-1;\\
&(\text{2a})\quad (k\cdot k_{r_1}\cdots k_{r_{m_1}},k_{s_1}\cdots k_{s_{m_2}},k_{t_1}\cdots k_{t_{m_3}}),\\&\quad\quad 0\leq k\leq a;\\
&(\text{2b})\quad (k_{r_1}\cdots k_{r_{m_1}},k\cdot k_{s_1}\cdots k_{s_{m_2}},k_{t_1}\cdots k_{t_{m_3}}),\\&\quad\quad 0\leq k\leq a-1;\\
&(\text{2c})\quad (k_{r_1}\cdots k_{r_{m_1}},k_{s_1}\cdots k_{s_{m_2}},k\cdot k_{t_1}\cdots k_{t_{m_3}}),\\&\quad\quad 0\leq k\leq a-1.
\end{align*}
The $T$-value of a type 1a signature is
\begin{equation*}
\left\lbrace\begin{array}{lr}
 \frac{1}{2}\left(\prod_{i=1}^r f(p_i^{k_i})\right)f(3^k) & 1\leq k_i\leq a_i-1 \text{ for some }1\leq i\leq r,\\
 & 0\leq k\leq a\\
 & \\
 \frac{1}{2}\left(\prod_{i=1}^rf(p_i^{k_i})\right)f(3) & 1\leq k\leq a-1\\
 & \\ 
 \frac{1}{2}f(3^k) & 1\leq k\leq a-1 \text{ and } k_i=0 \text{ for all } i\\
 & \\
 \frac{1}{2}\left(\prod_{i=1}^r f(p_i^{a_i})\right)f(3^k) & 1\leq k\leq a-1\text{ and } k_i=a_i \text{ for all }i\\
 & \\
 \frac{1}{2}\prod_{i=1}^rf(p_i^{a_i})+\frac{1}{2} & k=0 \text{ and } k_i=a_i \text{ for all }i\\
 & \\
 \frac{1}{6}(3+2\cdot \tau_2(3^an)+\left(\prod_{i=1}^rf(p_i^{a_i})\right)f(3^a)) & k=a \text{ and } k_i=a_i \text{ for all }i\\
 & \\
 \frac{1}{2}f(3^a)+\frac{1}{2} & k=a \text{ and } k_i=0 \text{ for all }i\\
 & \\
 \frac{1}{2}\prod_{i=1}^mf(p_{x_i}^{k_{x_i}})+\frac{1}{2} & k=0 \text{ and }k_{x_i}=a_i \\
 & \\
 \frac{1}{2}\left(\prod_{i=1}^mf(p_{x_i}^{k_{x_i}})\right)f(3^a)+\frac{1}{2} & k=a \text{ and }k_{x_i}=a_i
\end{array}\right..
\end{equation*}
(As always, the $k_{x_i}$ denote those parameters $k_i$ which achieve their maximum $a_i$, with the condition that not all $k_i$ reach the maximum $a_i$, but there is at least one $k_i$ that does.) Now compare the $T$-value for a type 1a signature with the corresponding $T$-value for a type (i) signature (found in \eqref{3.7.1}). Recall, by the definition of $f(x)$, that
\begin{equation*}
f(3^k)=\left\lbrace\begin{array}{lr}
 1 & k=0\\
 2\cdot 3^{k-1} & 1\leq k\leq a-1\\
 3^{a-1} & k=a
\end{array}\right..
\end{equation*}
Writing the original $T$-value as simply $T$, we have the $T$-value of a type 1a signature written as
\begin{equation}\label{3.8.1}
\left\lbrace\begin{array}{lr}
 T\cdot f(3^k) & 1\leq k_i\leq a_i-1 \text{ for some }1\leq i\leq r,\\
 & 0\leq k\leq a\\
 & \\
 \left(T-\frac{1}{2}\right)\cdot f(3^k) & 1\leq k\leq a-1\text{ and } k_{x_i}=a_i\\
 & \\ 
 \frac{1}{2}f(3^k) & 1\leq k\leq a-1 \text{ and } k_i=0 \text{ for all } i\\
 & \\
 3f(3^k)\left(T-\frac{1}{2}-\frac{1}{3}\cdot 2^r\right) & 1\leq k\leq a-1\text{ and } k_i=a_i \text{ for all }i\\
 & \\
 3\left(T-\frac{1}{2}-\frac{1}{3}\cdot 2^r\right)+\frac{1}{2} & k=0 \text{ and } k_i=a_i \text{ for all }i\\
 & \\
 f(3^a)\left(T-\frac{1}{2}-\frac{1}{3}\cdot 2^r\right)+\frac{1}{2} & k=a \text{ and } k_i=a_i \text{ for all }i\\
 & \\
 \frac{1}{2}f(3^a)+\frac{1}{2} & k=a \text{ and } k_i=0 \text{ for all }i\\
 & \\
 T & k=0 \text{ and }k_{x_i}=a_i \\
 & \\
 f(3^a)\left(T-\frac{1}{2}\right)+\frac{1}{2} & k=a \text{ and }k_{x_i}=a_i
\end{array}\right..
\end{equation}

Next, the $T$-value of a type 1b signature is $\prod_{i=1}^rf(p_i^{k_i})\cdot f(3^k)$ for $0\leq k\leq a-1$ and not all $k_i=a_i$. Since the $T$-value of a type (i) signature has several cases, we need to rewrite the $T$-value of a type 1b signature in terms of the $T$-value of a type (i) signature in each case. Thus, the $T$-value of a type 1b signature is written as
\begin{equation}\label{3.8.2}
\prod_{i=1}^rf(p_i^{k_i})\cdot f(3^k)=\left\lbrace\begin{array}{lr}
 2T & k=0 \text{ and } 1\leq k_i\leq a_i-1 \text{ for some }i\\
 & \\
 1 & k=0 \text{ and } k_i=0 \text{ for all }i\\
 & \\
 2\left(T-\frac{1}{2}\right) & k=0 \text{ and } k_{x_i}=a_i\\
 & \\
 2f(3^k)T & 1\leq k\leq a-1 \text{ and } 1\leq k_i\leq a_i-1 \text{ for some }i\\
 & \\
 f(3^k) & 1\leq k\leq a-1 \text{ and } k_i=0 \text{ for all }i\\
 & \\
 2f(3^k)\left(T-\frac{1}{2}\right) & 1\leq k\leq a-1 \text{ and } k_{x_i}=a_i
\end{array}\right..
\end{equation}

Thus, for a fixed $r$-tuple $k_1,\ldots,k_r$, we sum over $k$ in each of the cases for $k_1,\ldots,k_r$. We see that the total sum over $k$ of type 1a and 1b signatures from \eqref{3.8.1} and \eqref{3.8.2} becomes
\begin{equation*}
\left\lbrace\begin{array}{lr}
 T\cdot\left(\left(\sum_{k=0}^a f(3^k)\right)+2+\left(\sum_{k=1}^{a-1}2f(3^k)\right)\right) & 1\leq k_i\leq a_i-1 \text{ for some }i\\
 & \\
 \left(T-\frac{1}{2}-\frac{1}{3}\cdot 2^r\right)\cdot\left(3+3^{a-1}+3\sum_{k=1}^{a-1}f(3^k)\right)+1 & k_i=a_i \text{ for all }i\\
 & \\
 2\cdot 3^{a-1} & k_i=0 \text{ for all } i\\
 & \\ 
 \left(T-\frac{1}{2}\right)\left(4\cdot 3^{a-1}-1\right)+T+\frac{1}{2} & \text{otherwise}
\end{array}\right..
\end{equation*}
This simplifies to
\begin{equation}\label{3.8.4}
\left\lbrace\begin{array}{lr}
 4\cdot 3^{a-1}T & 1\leq k_i\leq a_i-1 \text{ for some }i\\
 & \\
 \left(T-\frac{1}{2}-\frac{1}{3}\cdot 2^r\right)\cdot\left(4\cdot 3^{a-1}\right)+1 & k_i=a_i \text{ for all }i\\
 & \\
 2\cdot 3^{a-1} & k_i=0 \text{ for all } i\\
 & \\ 
 \left(T-\frac{1}{2}\right)\left(4\cdot 3^{a-1}-1\right)+T+\frac{1}{2} & \text{otherwise}
\end{array}\right..
\end{equation}

Next consider type 2a, 2b, and 2c signatures. Recall that the $T$-value of a type (ii) signature for $C_n$ is simply $T:=\prod_{i=1}^r f(p_i^{k_i})$. Now, the $T$-value of a type 2a signature is $\prod_{i=1}^rf(p_i^{k_i})\cdot f(3^k)$ for $0\leq k\leq a$. Then the $T$-value of type 2a signature is just $T\cdot f(3^k)$. Similarly, the $T$-value of types 2b and 2c signatures is $\prod_{i=1}^r f(p_i^{k_i})f(3^k)$ for $0\leq k\leq a-1$. This can be written as just $T\cdot f(3^k)$. Then, summing over $k$ for these signatures, we see that the total sum of $T$-values for types 2a, 2b and 2c signatures is
\begin{align}
T\cdot\left(\sum_{k=0}^a f(3^k)\right)+2T\cdot\left(\sum_{k=0}^{a-1} f(3^k)\right)&=T\cdot \left(3\left(\sum_{k=0}^{a-1} f(3^k)\right)+f(3^a)\right)\nonumber\\
&=T\cdot (3^a+3^{a-1})=T\cdot (4\cdot 3^{a-1})\label{3.8.3}.
\end{align}

Finally, take the sum over all $r$-tuples $k_1,\ldots,k_r$ giving admissible signatures $(n_1,n_2,n_3)$ for $C_n$. This will give us a sum over all possible admissible signatures $(n_1',n_2',n_3')$ of $C_{3^an}$. Care must be taken when considering the final case of \eqref{3.8.4}. Of the signatures for $3^an$ satisfying the final case, there are $2^r-2$ such possibilities (since each $k_i$ for $1\leq i\leq r$ can either be $a_i$ or 0, but not all are 0 and not all are $a_i$). Thus,
\begin{align*}
QC(3^an)&=\sum_{(n_1',n_2',n_3')}T(n_1',n_2',n_3')\\
&=4\cdot 3^{a-1}\cdot\sum_{(n_1,n_2,n_3)}T(n_1,n_2,n_3)+4\cdot 3^{a-1}\left(-\frac{1}{2}-\frac{1}{3}\cdot 2^r\right)\\
&\quad+1+2\cdot 3^{a-1}+\left(-\frac{1}{2}\cdot 4\cdot 3^{a-1}+1\right)\cdot\left(2^r-2\right)\\
&=4\cdot 3^{a-1}QC(n)-1+4\cdot 3^{a-1}+\left(1-10\cdot 3^{a-2}\right)\cdot 2^r.
\end{align*}

\end{proof}



Since we know the value of $QC(n)$ in this case, $QC(3^an)$ can be easily derived.

\begin{proof}[Proof of First Subcase of Theorem \ref{thm:qcodd}]
Assume $n=\prod_{i=1}^r p_i^{a_i}$ for $p_i\equiv 1\modd 6$ for each $i$. By Theorem \ref{thm:recursive3an}, for any $a\geq 2$,
\begin{equation*}
QC(3^an)=4\cdot 3^{a-1}QC(n)-1+4\cdot 3^{a-1}+\left(1-10\cdot 3^{a-2}\right)\cdot 2^r.
\end{equation*}
Because $QC(n)=\frac{1}{6}\prod_{i=1}^{r}p_i^{a_i-1}(p_i+1)-1+\frac{5}{3}\cdot 2^{r-1}$, it follows that
\begin{equation*}
QC(3^an)=\frac{1}{6}\prod_{i=1}^r 3^{a-1}4p_i^{a_i-1}(p_i+1)-1+2^r.
\end{equation*}
\end{proof}

\subsection{When $p_i\equiv 5\modd 6$ for some $1\leq i\leq r$}

We prove the explicit form of $QC(n)$ in the final remaining case when $n$ is divisible by a prime congruent to five mod six. The goal is to show, when $n=\prod_{i=1}^r p_i^{a_i}$ with $p_i\equiv 5\modd 6$ for some $i$ and $a_i$ positive integers,
\begin{equation*}
QC(n)=\frac{1}{6}\prod_{i=1}^rp_i^{a_i-1}(p_i+1)-1+2^{r-1}.
\end{equation*}
Consider first when $r=1$, and write $p_1=p, a_1=a$ so that $p\equiv 5\modd 6$. We derive $QC(n)$ in this case.

\begin{theorem}\label{thm:fivemodsixoneprime}
If $p\equiv 5\modd 6$ and $a$ is a positive integer, then 
\begin{equation*}
QC(p^a)=\frac{1}{6}p^{a-1}(p+1).
\end{equation*}
\end{theorem}

\begin{proof}
By Harvey's Theorem, the only admissible signatures for $p^a$ are of the form
\begin{equation*}
(p^k,p^a,p^a)
\end{equation*}
for $1\leq k\leq a$. The $T$-values of these signatures are
\begin{equation*}
T(p^k,p^a,p^a)=\left\lbrace\begin{array}{lr}
 \frac{1}{2}f(p^k) & 1\leq k\leq a-1\\
 \frac{1}{6}(3+2\tau_2(p^a)+f(p^a)) & k=a
\end{array}\right..
\end{equation*}
The fact that $p\equiv 5\modd 6$ implies $\tau_2(p^a)=0$, and so
\begin{equation*}
T(p^k,p^a,p^a)=\left\lbrace\begin{array}{lr}
 \frac{1}{2}f(p^k) & 1\leq k\leq a-1\\
 \frac{1}{6}f(p^a)+\frac{1}{2} & k=a
\end{array}\right..
\end{equation*}
Therefore,
\begin{align*}
QC(p^a)&=\sum_{k=1}^a T(p^k,p^a,p^a)\\
&=\frac{1}{2}\left(\sum_{k=1}^{a-1}f(p^k)\right)+\frac{1}{6}f(p^a)+\frac{1}{2}\\
&=\frac{1}{2}\left(p^{a-1}-1\right)+\frac{1}{6}p^{a-1}(p-2)+\frac{1}{2}\\
&=\frac{1}{6}p^{a-1}(p+1).
\end{align*}
\end{proof}

Then we prove a formula relating $QC(p_1^{a_1}p^a)$ to $QC(p^a)$, where $p\equiv 5\modd 6$ and $p_1$ is any odd prime.

\begin{theorem}
If $p\equiv 5\modd 6$ and $p_1$ is any odd prime, then for positive integers $a_1$ and $a$,
\begin{equation*}
QC(p_1^{a_1}p^a)=QC(p^a)\cdot p_1^{a_1-1}(p_1+1)+1.
\end{equation*}
\end{theorem}

\begin{proof}
We extend each signature for $C_{p^a}$ to a signature for $C_{p_1^{a_1}p^a}$ by multiplying each period of an admissible signature for $C_{p^a}$ by an appropriate power of $p_1^{k_1}$. The only admissible signatures for $C_{p^a}$ are of the form $(p^k,p^a,p^a)$ for $1\leq k\leq a$. We can extend this signature to two types 1a and 1b for $C_{p_1^{a_1}p^a}$:
\begin{align*}
\text{(Type 1a)} & \quad (p_1^{k_1}p^k,p_1^{a_1}p^a,p_1^{a_1}p^a),\quad 0\leq k_1\leq a_1, 0\leq k\leq a;\\
\text{(Type 1b)} & \quad (p_1^{a_1}p^k,p_1^{k_1}p^a,p_1^{a_1}p^a), \quad 0\leq k_1\leq a_1-1, 0\leq k\leq a-1.
\end{align*}
As in the proofs of the main theorem from before, we let $T$ denote the $T$-value of a signature prior to extension ($T$ depends on the ranging parameter $k$). Recall that this $T$ is in fact
\begin{equation*}
T:=T(p^k,p^a,p^a)=\left\lbrace\begin{array}{lr}
 \frac{1}{2}f(p^k) & 1\leq k\leq a-1\\
 \frac{1}{6}(3+2\tau_2(p^a)+f(p^a)) & k=a
\end{array}\right..
\end{equation*}
Throughout, $\tau_2(p_1^{a_1}p^a)=0$ because $p\equiv 5\modd 6$ (see Lemma \ref{lem:tau2}). 

Computing the $T$-values of the type 1a signatures for $C_{p_1^{a_1}p^a}$,
\begin{align*}
T(p_1^{k_1}p^k,p_1^{a_1}p^a,p_1^{a_1}p^a)&=\left\lbrace\begin{array}{lr}
 \frac{1}{2}f(p_1^{k_1})f(p^k) & \text{either } 1\leq k_1\leq a_1-1\\
 & \text{ or } 1\leq k\leq a-1\\
 &\\
  \frac{1}{2}f(p^a)+\frac{1}{2} & k_1=0, k=a\\
  &\\
  \frac{1}{6}\left(3+2\tau_2(p_1^{a_1}p^a)+f(p_1^{a_1})f(p^a)\right) & k_1=a_1, k=a\\
  &\\
  \frac{1}{2}f(p_1^{a_1})+\frac{1}{2} & k_1=a_1, k=0
\end{array}\right..
\end{align*}
Then comparing to the $T$-values of signatures for $C_{p^a}$,
\begin{equation}\label{3.10.1}
T(p_1^{k_1}p^k,p_1^{a_1}p^a,p_1^{a_1}p^a)=\left\lbrace\begin{array}{lr}
 \frac{1}{2}f(p_1^{k_1}) & 1\leq k_1\leq a_1-1, k=0\\
 &\\
  T\cdot f(p_1^{k_1}) & 0\leq k_1\leq a_1, 1\leq k\leq a-1\\
  &\\
  3\left(T-\frac{1}{2}\right)f(p_1^{k_1}) & 1\leq k_1\leq a_1-1, k=a\\
  &\\
  3\left(T-\frac{1}{2}\right)+\frac{1}{2} & k_1=0, k=a\\
  &\\
  \left(T-\frac{1}{2}\right)f(p_1^{a_1})+\frac{1}{2} & k_1=a_1, k=a \\
  &\\
  \frac{1}{2}f(p_1^{a_1})+\frac{1}{2} & k_1=a_1, k=0
\end{array}\right..
\end{equation}
The cases above having no $T$ occur when $k=0$, i.e., when the signature for $C_{p_1^{a_1}p^a}$ does not arise via extension of some signature of $C_{p^a}$.

Next, computing the $T$-values of the 1b signatures for $C_{p_1^{a_1}p^a}$ and comparing them to the $T$-values of $C_{p^a}$,
\begin{equation}\label{3.10.2}
T(p_1^{a_1}p^k,p_1^{k_1}p^a,p_1^{a_1}p^a)=f(p_1^{k_1})f(p^k)=\left\lbrace\begin{array}{lr}
 f(p_1^{k_1}) & k=0\\
 2Tf(p_1^{k_1}) & 1\leq k\leq a-1
\end{array}\right..
\end{equation}
The case when $k=a$ is not allowed, since otherwise, we would have a type 1a signature instead.

Now, for each fixed value of $k$ with $0\leq k\leq a$, we form the sum of $T$-values over $k_1$ found in \eqref{3.10.1} and \eqref{3.10.2}. The sums are given by
\begin{equation*}
\left\lbrace\begin{array}{lr}
 \frac{1}{2}\left(\sum_{k_1=1}^{a_1-1}f(p_1^{k_1})\right)+\frac{1}{2}f(p_1^{a_1})+\frac{1}{2}+\sum_{k_1=0}^{a_1-1}f(p_1^{k_1}) & k=0\\
 &\\
 T\cdot\left(\sum_{k_1=0}^{a_1}f(p_1^{k_1})\right)+2T\cdot\left(\sum_{k_1=0}^{a_1-1}f(p_1^{k_1})\right) & 1\leq k\leq a-1\\
 &\\
 3\left(T-\frac{1}{2}\right)\left(\sum_{k_1=1}^{a_1-1}f(p_1^{k_1})\right)+3\left(T-\frac{1}{2}\right)+\frac{1}{2}+\left(T-\frac{1}{2}\right)f(p_1^{a_1})+\frac{1}{2} & k=a
\end{array}\right..
\end{equation*}
This simplifies to
\begin{align*}
&\left\lbrace\begin{array}{lr}
 \frac{1}{2}(p_1^{a_1-1}-1+p_1^{a_1-1}(p_1-2))+\frac{1}{2}+p_1^{a_1-1} & k=0\\
 &\\
 T\cdot (p_1^{a_1-1}+p_1^{a_1-1}(p_1-2))+2T\cdot p_1^{a_1-1} & 1\leq k\leq a-1\\
 &\\
 \left(T-\frac{1}{2}\right)(3(p_1^{a_1-1}-1)+3+p_1^{a_1-1}(p_1-2))+1 & k=a
\end{array}\right.\\
&\\
&=\left\lbrace\begin{array}{lr}
 \frac{1}{2}p_1^{a_1-1}(p_1+1) & k=0\\
 &\\
 T\cdot p_1^{a_1-1}(p_1+1) & 1\leq k\leq a-1\\
 &\\
 \left(T-\frac{1}{2}\right)p_1^{a_1-1}(p_1+1)+1 & k=a
\end{array}\right..
\end{align*}
Summing over $k$ gives us a sum over all signatures for $C_{p^a}$, along with the degenerate case when $k=0$. Therefore,
\begin{align*}
QC(p_1^{a_1}p^a)&=\left(\sum_{k=1}^{a}T(p^k,p^a,p^a)\right)\cdot p_1^{a_1-1}(p_1+1)+1\\
&\quad +\frac{1}{2}p_1^{a_1-1}(p_1+1)-\frac{1}{2}p_1^{a_1-1}(p_1+1)\\
&=QC(p^a)\cdot p_1^{a_1-1}(p_1+1)+1.
\end{align*}
\end{proof}

Finally, we prove the general recursive formula needed to derive $QC(n)$ in this case.
\begin{theorem}\label{thm:recursivefivemod6}
For any $r\geq 1, p\equiv 5\modd 6$, $p_i$ any odd prime for $1\leq i\leq r$, and $a_i$ and $a$ positive integers,
\begin{align*}
QC(p_1^{a_1}\cdots p_{r+1}^{a_{r+1}}\cdot p^a)&=\left(QC(p_1^{a_1}\cdots p_r^{a_r}\cdot p^a)+1-2^r\right)\cdot p_{r+1}^{a_{r+1}-1}(p_{r+1}+1)\\
&\quad-1+2^{r+1}.
\end{align*}
\end{theorem}

\begin{proof}
We follow the same method and notation as in the proof of $QC(3n)$. Let $n=\prod_{i=1}^rp_i^{a_i}$. For the two types of signatures of $C_{np^a}$,
\begin{enumerate}
\item $(p_1^{k_1}\cdots p_r^{k_r}p^k,p_1^{a_1}\cdots p_r^{a_r}p^a,p_1^{a_1}\cdots p_r^{a_r}p^a),$
\item $(k_{r_1}\cdots k_{r_{m_1}},k_{s_1}\cdots k_{s_{m_2}},k_{t_1}\cdots k_{t_{m_3}}),$
\end{enumerate}
we extend to signatures of types 1a, 1b, 2a, 2b, and 2c of $C_{np_{r+1}^{a_{r+1}}p^a}$ as before:
\begin{align*}
&(\text{1a})\quad (p_1^{k_1}\cdots p_{r+1}^{k_{r+1}}p^k,p_1^{a_1}\cdots p_{r+1}^{a_{r+1}}p^a,p_1^{a_1}\cdots p_{r+1}^{a_{r+1}}p^a),\\&\quad\quad 0\leq k_i\leq a_i, 0\leq k\leq a,\\
& \quad\quad k_i\neq 0 \text{ and } k\neq 0 \text{ simultaneously for all } 1\leq i\leq r;\\
&(\text{1b})\quad (p_1^{k_1}\cdots p_r^{k_r}p_{r+1}^{a_{r+1}}p^k,p_1^{a_1}\cdots p_r^{k_r} p_{r+1}^{k_{r+1}}p^a,p_1^{a_1}\cdots p_{r+1}^{a_{r+1}}p^a),\\&\quad\quad 0\leq k_i\leq a_i,\\
&\quad\quad 0\leq k_{r+1}\leq a_{r+1}-1,\\
&\quad\quad 0\leq k\leq a-1,\\
&\quad\quad k_i\neq a_i \text{ and } k\neq a \text{ simultaneously for all } 1\leq i\leq r;\\
&(\text{2a})\quad (k_{r_1}\cdots k_{r_{m_1}}k_{r+1},k_{s_1}\cdots k_{s_{m_2}},k_{t_1}\cdots k_{t_{m_3}}),\\&\quad\quad 0\leq k_{r+1}\leq a_{r+1};\\
&(\text{2b})\quad (k_{r_1}\cdots k_{r_{m_1}},k_{s_1}\cdots k_{s_{m_2}}k_{r+1},k_{t_1}\cdots k_{t_{m_3}}),\\&\quad\quad 0\leq k_{r+1}\leq a_{r+1}-1;\\
&(\text{2c})\quad (k_{r_1}\cdots k_{r_{m_1}},k_{s_1}\cdots k_{s_{m_2}},k_{t_1}\cdots k_{t_{m_3}}k_{r+1}),\\&\quad\quad 0\leq k_{r+1}\leq a_{r+1}-1.
\end{align*}

Because $p\equiv 5\modd 6$, then $\tau_2(x\cdot p^a)=0$ for any integer $x$ (see Lemma \ref{lem:tau2}). Now, the $T$-value for a type (i) signature for $C_{np^a}$ is
\begin{equation*}
\left\lbrace\begin{array}{lr}
 \frac{1}{2}\prod_{i=1}^r f(p_i^{k_i})\cdot f(p^k) & \text{ either } 1\leq k_i\leq a_i-1 \text{ for some } i\\
 & \text{or } 1\leq k\leq a-1\\
 &\\
 \frac{1}{2}\prod_{i=1}^r f(p_i^{k_i})+\frac{1}{2} & k=0, k_i=a_i \text{ for all }i\\
 &\\
 \frac{1}{2}f(p^a)+\frac{1}{2} & k=a, k_i=0 \text{ for all }i\\
 &\\
 \frac{1}{6}\left(3+2\cdot \tau_2(n\cdot p_{r+1}^{a_{r+1}}\cdot p^a)+\prod_{i=1}^r f(p_i^{a_i})f(p^a)\right) & k=a, k_i=a_i\text{ for all }i\\
 &\\
 \frac{1}{2}\prod_{i=1}^mf(p_{x_i}^{k_{x_i}})+\frac{1}{2} & k_{x_i}=a_i, k=0\\
 &\\
 \frac{1}{2}\prod_{i=1}^mf(p_{x_i}^{k_{x_i}})f(p^a)+\frac{1}{2} & k_{x_i}=a_i, k=a
\end{array}\right..
\end{equation*}
The $T$-value of the type 1a signature for $C_{np_{r+1}^{a_{r+1}}p^a}$ is identical:
\begin{equation*}
\left\lbrace\begin{array}{lr}
 \frac{1}{2}\prod_{i=1}^{r+1} f(p_i^{k_i})\cdot f(p^k) & \text{ either } 1\leq k_i\leq a_i-1 \text{ for some } i\\
 & \text{or } 1\leq k\leq a-1\\
 &\\
 \frac{1}{2}\prod_{i=1}^{r+1} f(p_i^{k_i})+\frac{1}{2} & k=0, k_i=a_i \text{ for all }i\\
 &\\
 \frac{1}{2}f(p^a)+\frac{1}{2} & k=a, k_i=0 \text{ for all }i\\
 &\\
 \frac{1}{6}\left(3+2\cdot \tau_2(n\cdot p_{r+1}^{a_{r+1}}\cdot p^a)+\prod_{i=1}^{r+1} f(p_i^{a_i})f(p^a)\right) & k=a, k_i=a_i\text{ for all }i\\
 &\\
 \frac{1}{2}\prod_{i=1}^mf(p_{x_i}^{k_{x_i}})+\frac{1}{2} & k_{x_i}=a_i, k=0\\
 &\\
 \frac{1}{2}\prod_{i=1}^mf(p_{x_i}^{k_{x_i}})f(p^a)+\frac{1}{2} & k_{x_i}=a_i, k=a
\end{array}\right..
\end{equation*}
Now we need to compare the $T$-values of the type 1a signatures to the type (i) signatures. The $T$-values of the type 1a signatures, for various ranges of the parameters $k_1,\ldots,k_r$ and $k$, are given in Table \ref{fig:table1a} in terms of $T$, the $T$-value of the original signature prior to extending. Those values without a $T$ indicate a degenerate case, i.e., a $T$-value for a signature which did not arise via extension. The left column of Table \ref{fig:table1a} is $T(p_1^{k_1}\cdots p_{r+1}^{k_{r+1}}p^k,p_1^{a_1}\cdots p_{r+1}^{a_{r+1}}p^a,p_1^{a_1}\cdots p_{r+1}^{a_{r+1}}p^a)$, and the right column gives the ranges for which the $T$-value given in the left column is valid.
\begin{center}
\begin{figure}
\begin{tabular}{|l|}\hline
$T$-values for Type 1a Signatures for $C_{np_{r+1}^{a_{r+1}}p^a}$\\\hline
\end{tabular}
\begin{tabular}{|l|l|} \hline
$T$-value & Ranges \\ \hline \hline
 $T\cdot f(p_{r+1}^{k_{r+1}})$ & $1\leq k_{r+1}\leq a_{r+1}-1$ \\
 & and either $1\leq k_i\leq a_i-1$ for some $1\leq i\leq r$\\
 & or $1\leq k\leq a-1$ \\ \hline
 $\left(T-\frac{1}{2}\right)\cdot f(p_{r+1}^{k_{r+1}})$ & $1\leq k_{r+1}\leq a_{r+1}-1$\\
 & $k=0, k_i=a_i$ for all $1\leq i\leq r$ \\ \hline
 $\frac{1}{2}f(p_{r+1}^{k_{r+1}})$ & $1\leq k_{r+1}\leq a_{r+1}-1$\\
 & $k=0, k_i=0$ for all $1\leq i\leq r$ \\ \hline
 $\left(T-\frac{1}{2}\right)f(p_{r+1}^{k_{r+1}})$ & $1\leq k_{r+1}\leq a_{r+1}-1$\\
 & $k=a, k_i=0$ for all $1\leq i\leq r$ \\ \hline
 $T\cdot f(p_{r+1}^{k_{r+1}})$ & $k_{r+1}=0$ or $k_{r+1}=a_{r+1}$ \\ 
 & and either $1\leq k_i\leq a_i-1$ for some $1\leq i\leq r$\\
 & or $1\leq k\leq a-1$ \\ \hline
 $T$ & $k_{r+1}=0$\\
 & $k=0, k_i=a_i$ for all $1\leq i\leq r$\\ \hline
 $T$ & $k_{r+1}=0$\\
 & $k=0, k_{x_i}=a_i$\\ \hline
 $T$ & $k_{r+1}=0$\\
 & $k=a, k_i=0$ for all $1\leq i\leq r$\\ \hline
 $T$ & $k_{r+1}=0$\\
 & $k=a, k_{x_i}=a_i$\\ \hline
 $3\left(T-\frac{1}{2}\right)+\frac{1}{2}$ & $k_{r+1}=0$\\
 & $k=a, k_i=a_i$ for all $1\leq i\leq r$\\ \hline
 $\frac{1}{2}f(p_{r+1}^{a_{r+1}})+\frac{1}{2}$ & $k_{r+1}=a_{r+1}$\\
 & $k=0, k_i=0$ for all $1\leq i\leq r$\\ \hline
 $\left(T-\frac{1}{2}\right)f(p_{r+1}^{a_{r+1}})+\frac{1}{2}$ & $k_{r+1}=a_{r+1}$\\
 & $k=0, k_{x_i}=a_i$\\ \hline
 $\left(T-\frac{1}{2}\right)f(p_{r+1}^{a_{r+1}})+\frac{1}{2}$ & $k_{r+1}=a_{r+1}$\\
 & $k=0, k_i=a_i$ for all $1\leq i\leq r$\\ \hline
 $\left(T-\frac{1}{2}\right)f(p_{r+1}^{a_{r+1}})+\frac{1}{2}$ & $k_{r+1}=a_{r+1}$\\
 & $k=a, k_{x_i}=a_i$\\ \hline
 $\left(T-\frac{1}{2}\right)f(p_{r+1}^{a_{r+1}})+\frac{1}{2}$ & $k_{r+1}=a_{r+1}$\\
 & $k=a, k_i=0$ for all $1\leq i\leq r$\\ \hline
 $\left(T-\frac{1}{2}\right)f(p_{r+1}^{a_{r+1}})+\frac{1}{2}$ & $k_{r+1}=a_{r+1}$\\
 & $k=a, k_i=a_i$ for all $1\leq i\leq r$\\ \hline
\end{tabular}
\caption{The $T$-values for type 1a signatures for $C_{np_{r+1}^{a_{r+1}}p^a}$ with their respective ranges.}
\label{fig:table1a}
\end{figure}
\end{center}

Next, we compute the $T$-value of a type 1b signature. The type 1b signature are of the form $(p_1^{k_1}\cdots p_r^{k_r}p_{r+1}^{a_{r+1}}p^k,p_1^{a_1}\cdots p_r^{k_r} p_{r+1}^{k_{r+1}}p^a,p_1^{a_1}\cdots p_{r+1}^{a_{r+1}}p^a)$. The $T$-value of this type of signature is simply
\begin{equation*}
\left(\prod_{i=1}^r f(p_i^{k_i})\right)f(p_{r+1}^{k_{r+1}})f(p^k).
\end{equation*}
Now compare these values to their original $T$-value, which we denote by $T$, for various ranges of $k_i$ and $k$, for $1\leq i\leq r$. This is summarized in Table \ref{fig:table1b}. 

Then, for fixed values of $k_1,\ldots, k_r$ and $k$, we compute the total sum of the $T$-values for type 1a and 1b signatures over $k_{r+1}$. This is summarized in Table \ref{fig:table1a1b}. Significant simplifications yield the subsequent Table \ref{fig:table1a1bsimp}.
\begin{center}
\begin{figure}
\begin{tabular}{|l|}\hline
$T$-values for Type 1b Signatures of $C_{np_{r+1}^{a_{r+1}}p^a}$\\\hline
\end{tabular}
\begin{tabular}{|l|l|}\hline
$T$-values & Ranges \\ \hline
$2T\cdot f(p_{r+1}^{k_{r+1}})$ & $0\leq k_{r+1}\leq a_{r+1}-1$ \\
& and either $1\leq k_i\leq a_i-1$ for some $1\leq i\leq r$\\
 & or $1\leq k\leq a-1$ \\ \hline
 $2\left(T-\frac{1}{2}\right)f(p_{r+1}^{k_{r+1}})$ & $k=0, k_i=a_i$ for all $1\leq i\leq r$\\
 & $0\leq k_{r+1}\leq a_{r+1}-1$\\ \hline
 $f(p_{r+1}^{k_{r+1}})$ & $k=0, k_i=0$ for all $1\leq i\leq r$ \\
 & $1\leq k_{r+1}\leq a_{r+1}-1$\\ \hline
 $2\left(T-\frac{1}{2}\right)\cdot f(p_{r+1}^{k_{r+1}})$ & $k=0, k_i=0$ for all $1\leq i\leq r$\\
 & $0\leq k_{r+1}\leq a_{r+1}-1$\\ \hline
 $2\left(T-\frac{1}{2}\right)\cdot f(p_{r+1}^{k_{r+1}})$ & $k=0, k_{x_i}=a_i$\\
 & $0\leq k_{r+1}\leq a_{r+1}-1$\\ \hline
 $2\left(T-\frac{1}{2}\right)\cdot f(p_{r+1}^{k_{r+1}})$ & $k=a, k_{x_i}=a_i$\\
 & $0\leq k_{r+1}\leq a_{r+1}-1$\\ \hline
 1 & $k_{r+1}=0, k_i=0$ for all $1\leq i\leq r$, and $k=0$\\ \hline
\end{tabular}
\caption{The $T$-values for type 1b signatures for $C_{np_{r+1}^{a_{r+1}}p^a}$ with their respective ranges.}
\label{fig:table1b}
\end{figure}
\end{center}

\begin{center}
\begin{figure}
\begin{tabular}{|l|}\hline
Sum of 1a and 1b Signatures over $k_{r+1}$ \\ \hline
\end{tabular}
\begin{tabular}{|l|l|}\hline
$T$-values & Ranges for $k_1,\ldots,k_r$ and $k$ \\ \hline
$\displaystyle T\cdot \left(\sum_{k_{r+1}=1}^{a_{r+1}-1}f(p_{r+1}^{k_{r+1}})\right)$ & either $1\leq k_i\leq a_i-1$ for some $1\leq i\leq r$\\
$\displaystyle+2T\cdot\left(\sum_{k_{r+1}=0}^{a_{r+1}-1}f(p_{r+1}^{k_{r+1}})\right)$ & or $1\leq k\leq a-1$\\ 
 $+T+T\cdot f(p_{r+1}^{a_{r+1}})$ & \\ \hline
 $\displaystyle T+\left(T-\frac{1}{2}\right)\cdot\left(\sum_{k_{r+1}=1}^{a_{r+1}-1}f(p_{r+1}^{k_{r+1}})\right)$ & $k=0$ and $k_i=a_i$ for all $1\leq i\leq r$\\
 $+\left(T-\frac{1}{2}\right)f(p_{r+1}^{a_{r+1}})+\frac{1}{2}$ & \\
 $\displaystyle+2\left(T-\frac{1}{2}\right)\cdot \left(\sum_{k_{r+1}=0}^{a_{r+1}-1}f(p_{r+1}^{k_{r+1}})\right)$ & \\ \hline
 $\displaystyle\frac{1}{2}\cdot \left(\sum_{k_{r+1}=1}^{a_{r+1}-1}f(p_{r+1}^{k_{r+1}})\right)+\frac{1}{2}f(p_{r+1}^{a_{r+1}})$ & $k=0$ and $k_i=0$ for all $1\leq i\leq r$ \\ 
 $\displaystyle+\frac{1}{2}\left(\sum_{k_{r+1}=1}^{a_{r+1}-1}f(p_{r+1}^{k_{r+1}})\right)+1$ & \\ \hline
 $\displaystyle\left(T-\frac{1}{2}\right)\cdot \left(\sum_{k_{r+1}=1}^{a_{r+1}-1}f(p_{r+1}^{k_{r+1}})\right)$ & $k=a$ and $k_i=0$ for all $1\leq i\leq r$ \\
 $+T+\left(T-\frac{1}{2}\right)f(p_{r+1}^{a_{r+1}})$ & \\
 $\displaystyle+2\left(T-\frac{1}{2}\right)\cdot \left(\sum_{k_{r+1}=1}^{a_{r+1}-1}f(p_{r+1}^{k_{r+1}})\right)$ & \\ \hline
 $\displaystyle 3\left(T-\frac{1}{2}\right)\cdot \left(\sum_{k_{r+1}=1}^{a_{r+1}-1}f(p_{r+1}^{k_{r+1}})\right)$ & $k=a$ and $k_i=a_i$ for all $1\leq i\leq r$ \\
 $+3\left(T-\frac{1}{2}\right)+\frac{1}{2}$ & \\
 $+\left(T-\frac{1}{2}\right)f(p_{r+1}^{a_{r+1}})+\frac{1}{2}$ & \\ \hline
 $\displaystyle\left(T-\frac{1}{2}\right)\cdot \left(\sum_{k_{r+1}=1}^{a_{r+1}-1}f(p_{r+1}^{k_{r+1}})\right)$ & $k=0$ and $k_{x_i}=a_i$ \\
 $+T+\left(T-\frac{1}{2}\right)f(p_{r+1}^{a_{r+1}})+\frac{1}{2}$ & \\
 $\displaystyle+2\left(T-\frac{1}{2}\right)\cdot \left(\sum_{k_{r+1}=1}^{a_{r+1}-1}f(p_{r+1}^{k_{r+1}})\right)$ & \\ \hline
 $\displaystyle\left(T-\frac{1}{2}\right)\cdot \left(\sum_{k_{r+1}=1}^{a_{r+1}-1}f(p_{r+1}^{k_{r+1}})\right)$ & $k=a$ and $k_{x_i}=a_i$ \\
 $+T+\left(T-\frac{1}{2}\right)f(p_{r+1}^{a_{r+1}})+\frac{1}{2}$ & \\
 $\displaystyle+2\left(T-\frac{1}{2}\right)\cdot \left(\sum_{k_{r+1}=1}^{a_{r+1}-1}f(p_{r+1}^{k_{r+1}})\right)$ & \\ \hline
\end{tabular}
\caption{The $T$-values of type 1a and 1b signatures added together, with the sum running over $k_{r+1}$.}
\label{fig:table1a1b}
\end{figure}
\end{center}

\begin{center}
\begin{figure}
\begin{tabular}{|l|}\hline
Sum of 1a and 1b Signatures over $k_{r+1}$ \\ \hline
\end{tabular}
\begin{tabular}{|l|l|}\hline
$T$-values & Ranges for $k_1,\ldots,k_r$ and $k$ \\
 \hline
$T\cdot p_{r+1}^{a_{r+1}-1}(p_{r+1}+1)$ & either $1\leq k_i\leq a_i-1$ for some $1\leq i\leq r$ \\
& \\ \hline
 $\left(T-\frac{1}{2}\right)\cdot p_{r+1}^{a_{r+1}-1}(p_{r+1}+1)+1$ & $k=0$ and $k_i=a_i$ for all $1\leq i\leq r$\\
  & \\ \hline
 $\frac{1}{2}p_{r+1}^{a_{r+1}-1}(p_{r+1}+1)$ & $k=0$ and $k_i=0$ for all $1\leq i\leq r$ \\
  & \\ \hline
 $\left(T-\frac{1}{2}\right)\cdot p_{r+1}^{a_{r+1}-1}(p_{r+1}+1)+1$ & $k=a$ and $k_i=0$ for all $1\leq i\leq r$ \\
 & \\ \hline
 $\left(T-\frac{1}{2}\right)\cdot p_{r+1}^{a_{r+1}-1}(p_{r+1}+1)+1$ & $k=a$ and $k_i=a_i$ for all $1\leq i\leq r$ \\
  & \\ \hline
 $\left(T-\frac{1}{2}\right)\cdot p_{r+1}^{a_{r+1}-1}(p_{r+1}+1)+1$ & $k=0$ and $k_{x_i}=a_i$ \\
  & \\ \hline
 $\left(T-\frac{1}{2}\right)\cdot p_{r+1}^{a_{r+1}-1}(p_{r+1}+1)+1$ & $k=a$ and $k_{x_i}=a_i$ \\
 & \\ \hline
\end{tabular}
\caption{Simplified $T$-values taken from Table \ref{fig:table1a1b}.}
\label{fig:table1a1bsimp}
\end{figure}
\end{center}

Adding up all the values in Table \ref{fig:table1a1bsimp} gives the total sum over all $T$-values of types 1a and 1b signatures in terms of $T$, which are the $T$-values of type (i) signatures. This sum is
\begin{align}
&\left(\sum_{\text{type }(i)}T\right)\cdot p_{r+1}^{a_{r+1}-1}(p_{r+1}+1)-\frac{1}{2}p_{r+1}^{a_{r+1}-1}(p_{r+1}+1)+\frac{1}{2}p_{r+1}^{a_{r+1}-1}(p_{r+1}+1)\nonumber\\
&\quad -\frac{1}{2}p_{r+1}^{a_{r+1}-1}(p_{r+1}+1)+1-\frac{1}{2}p_{r+1}^{a_{r+1}-1}(p_{r+1}+1)+1\nonumber\\
&=\left(\left(\sum_{\text{type }(i)}T\right)+1-2^r\right)p_{r+1}^{a_{r+1}-1}(p_{r+1}+1)-1+2^{r+1}.\label{sum1a1b}
\end{align}
This completes the sums over the $T$-values of types 1a and 1b signatures, written in terms of type (i) $T$-values.

Next, consider the type (ii) signatures of the form
\begin{equation*}
(k_{r_1}\cdots k_{r_{m_1}},k_{s_1}\cdots k_{s_{m_2}},k_{t_1}\cdots k_{t_{m_3}}).
\end{equation*}
Recall that this type refers to signatures whose ranging parameters of exponents on primes comes from a partition of $k_1,\ldots,k_{r},k$ into at two or three parts. The $T$-value of the above signature is simply $T:=\prod_{i=1}^r f(p_i^{k_i})\cdot f(p^k)$. We calculate the $T$-values of the extended signatures 2a, 2b, and 2c in terms of the $T$-value for signatures of type (ii), which we denote by $T$.

The $T$-value of a type 2a signature for $C_{np_{r+1}^{a_{r+1}}p^a}$, which is of the form
\begin{equation*}
(k_{r_1}\cdots k_{r_{m_1}}k_{r+1},k_{s_1}\cdots k_{s_{m_2}},k_{t_1}\cdots k_{t_{m_3}})
\end{equation*}
for $0\leq k_{r+1}\leq a_{r+1}$, has $T$-value given by
\begin{equation*}
\prod_{i=1}^r f(p_i^{k_i})f(p^k)\cdot f(p_{r+1}^{k_{r+1}})=T\cdot f(p_{r+1}^{k_{r+1}}). 
\end{equation*}
The $T$-value of type 2b and 2c signatures for $C_{np_{r+1}^{a_{r+1}}p^a}$, which are of the forms
\begin{equation*}
(k_{r_1}\cdots k_{r_{m_1}},k_{s_1}\cdots k_{s_{m_2}}k_{r+1},k_{t_1}\cdots k_{t_{m_3}})
\end{equation*}
and
\begin{equation*}
(k_{r_1}\cdots k_{r_{m_1}},k_{s_1}\cdots k_{s_{m_2}},k_{t_1}\cdots k_{t_{m_3}}k_{r+1}),
\end{equation*}
respectively, for $0\leq k_{r+1}\leq a_{r+1}-1$, both have the same $T$-value given by
\begin{equation*}
\prod_{i=1}^r f(p_i^{k_i})f(p^k)\cdot f(p_{r+1}^{k_{r+1}})=T\cdot f(p_{r+1}^{k_{r+1}}). 
\end{equation*}
Thus, the total sum over $k_{r+1}$ of the type 2a, 2b, and 2c signatures for $C_{np_{r+1}^{a_{r+1}}p^a}$ in terms of $T$ is given by
\begin{align*}
&T\cdot\left(\left(\sum_{k_{r+1}=0}^{a_{r+1}}f(p_{r+1}^{k_{r+1}})\right)+2\cdot \left(\sum_{k_{r+1}=0}^{a_{r+1}-1}f(p_{r+1}^{k_{r+1}})\right)\right)\nonumber\\
&=T\cdot\left(p_{r+1}^{a_{r+1}-1}+p_{r+1}^{a_{r+1}-1}(p_{r+1}-2)+2p_{r+1}^{a_{r+1}-1}\right)\nonumber\\
&=T\cdot p_{r+1}^{a_{r+1}-1}(p_{r+1}+1).
\end{align*}
Taking the sum over all $k_1\ldots,k_r, k$ yields all the type (ii) $T$-values as well:
\begin{equation}
\left(\sum_{\text{type }(ii)}T\right)\cdot p_{r+1}^{a_{r+1}-1}(p_{r+1}+1).\label{sum2a2b2c}
\end{equation}

Finally, combining \eqref{sum1a1b} and \eqref{sum2a2b2c},
\begin{align*}
&QC(p_1^{a_1}\cdots p_{r+1}^{a_{r+1}}\cdot p^a)\\
&=\left(\left(\sum_{\text{type }(i)}T\right)+\left(\sum_{\text{type }(ii)}T\right)+1-2^r\right)p_{r+1}^{a_{r+1}-1}(p_{r+1}+1)-1+2^{r+1}\\
&=\left(QC(p_1^{a_1}\cdots p_r^{a_r}\cdot p^a)+1-2^r\right)\cdot p_{r+1}^{a_{r+1}}(p_{r+1}+1)-1+2^{r+1}.
\end{align*}
\end{proof}

We can now conclude the proof of Theorem \ref{thm:qcodd} by using the above theorems.

\begin{proof}[Proof of Second Subcase of Theorem \ref{thm:qcodd}]
We aim to show
\begin{equation*}
QC(p_1^{a_1}\cdots p_r^{a_r})=\frac{1}{6}\prod_{i=1}^r p_i^{a_i-1}(p_i+1)-1+2^{r-1}
\end{equation*}
when some $p_i\equiv 5\modd 6$ and all primes $p_i$ are odd. We proceed by induction on the number of primes $r$. This has been shown true for $r=1$ above in Theorem \ref{thm:fivemodsixoneprime}.

Now assume that the formula holds for some $r\geq 1$. We show the formula holds for $p_1^{a_1}\cdots p_{r+1}^{a_{r+1}}$, where some $p_i\equiv 5\modd 6$ for $1\leq i\leq r+1$ and all primes $p_i$ are odd. Since there exists some $p_i\equiv 5\modd 6$ for $1\leq i\leq r+1$, take one such $p_i^{a_i}$ and rename it $p^a$. Relabel all other primes to be written in increasing order in their indices. Then by Theorem \ref{thm:recursivefivemod6} and using the induction hypothesis,
\begin{align*}
&QC(p_1^{a_1}\cdots p_r^{a_r}p^a)=\left(QC(p_1^{a_1}\cdots p_{r-1}^{a_{r-1}}\cdot p^a)+1-2^{r-1}\right)p_r^{a_r-1}(p_r+1)-1+2^r\\
&=\left(\frac{1}{6}\prod_{i=1}^{r-1}p_i^{a_i-1}(p_i+1)\cdot p^{a-1}(p+1)-1+2^{r-1}+1-2^{r-1}\right)\cdot p_r^{a_r-1}(p_r+1)\\
&\quad -1+2^r\\
&=\frac{1}{6}\prod_{i=1}^{r}p_i^{a_i-1}(p_i+1)\cdot p^{a-1}(p+1)-1+2^r.
\end{align*}
Relabel the indices of the primes as they were originally, yielding
\begin{equation*}
QC(p_1^{a_1}\cdots p_{r+1}^{a_{r+1}})=\frac{1}{6}\prod_{i=1}^{r+1}p_i^{a_i-1}(p_i+1)-1+2^r.
\end{equation*}
Therefore, the formula holds for $p_1^{a_1}\cdots p_{r+1}^{a_{r+1}}$. We conclude
\begin{equation*}
QC(n)=\frac{1}{6}\prod_{i=1}^{r}p_i^{a_i-1}(p_i+1)\cdot p^{a-1}(p+1)-1+2^{r-1}
\end{equation*}
when $n=\prod_{i=1}^r p_i^{a_i}$ with each $p_i$ odd and some $p_i\equiv 5\modd 6$.

\end{proof}


\section{Examples and Further Exploration}\label{sec:conseq}



\begin{example}
We describe some interesting results on compact Riemann surfaces that admit an order $p$ autormophism, called a \emph{$p$-gonal surface}. Let $p\geq 5$ be a prime. A result by G. Gonz\'alez-Diez \cite{gonz95} proves the following: given any two order $p$ automorphisms, $\tau_1$ and $\tau_2$, on a compact Riemann surface $X$, the groups generated by $\tau_1$ and $\tau_2$ are conjugate in $\aut(X)$. That is, $\langle\tau_1\rangle$ and $\langle\tau_2\rangle$ are topologically equivalent. A related result due to Riera and Rodr{\'i}guez (Theorem 1, \cite{riera}) demonstrates a one-to-one correspondence between equivalence classes of $(p,p,p)$-generating vectors of $\Z/p\Z$ (the integers modulo $p$) and the set of conformal classes of Lefschetz curves having affine equation $y^p=x^u(x-1)$ for $u\in\lbrace 1,\ldots,p-2\rbrace$. This means that the number $QC(p)$ of topologically distinct actions of $C_p$ on quasiplatonic surfaces also distinguishes these surfaces up to conformal equivalence. In other words, for $p\geq 5$, the formula
\begin{equation*}
QC(p)=\left\lbrace\begin{array}{lr}
 \frac{1}{6}(p+1) & p\equiv 5\modd 6\\
 & \\
 \frac{1}{6}(p+1)+\frac{2}{3} & p\equiv 1\modd 6.
\end{array}\right.
\end{equation*}
enumerates the quasiplatonic cyclic $p$-gonal surfaces (see Corollary \ref{cor:primepower} in this paper). For quasiplatonic surfaces, we need $p\geq 5$, since $p=2$ or $p=3$ correspond to genus zero and genus one surfaces, respectively. The Lefschetz curves contain the regular dessins having a $C_p$ group of automorphisms(e.g., Example 4.1, \cite{jones1}).
\end{example}

\begin{example}
Consider quasiplatonic topological $C_7$-actions on surfaces. As discussed in Examples \ref{ex:twoactions} and \ref{ex:twoactionverified}, there are only two such actions on surfaces. Harvey's Theorem gives only one admissible signature $(7,7,7)$ of such surfaces, implying that these surfaces have genus given by the Riemann-Hurwitz formula as
\begin{equation*}
1+\frac{7}{2}\left(1-\frac{1}{7}-\frac{1}{7}-\frac{1}{7}\right)=3.
\end{equation*} 
On the other hand, there are $r(C_7)=8$ regular dessins with a $C_7$ automorphism group. Since three of these dessins lie on the Riemann sphere, there are only five regular dessins on genus three surfaces having a $C_7$ group of automorphisms. Moreover, there are only two quasiplatonic surfaces of genus three admitting a $C_7$ group of automorphisms (see e.g., Table 5.3 in \cite{jones} for the classification of quasiplatonic surfaces up to genus four). Refer to Figure \ref{fig:dessins} for these regular dessins. The arrows indicate the edges to identified of the fundamental domain of the Fuchsian group uniformizing the surface with the embedded dessin. Such a Fuchsian group is a normal subgroup of $\Delta(7,7,7)$ of index seven.
\begin{center}
\begin{figure}[ht]
\includegraphics[page=12,scale=.6]{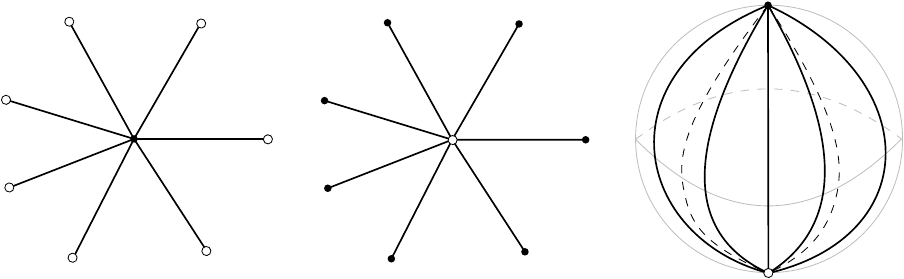}

\includegraphics[page=14,scale=.4]{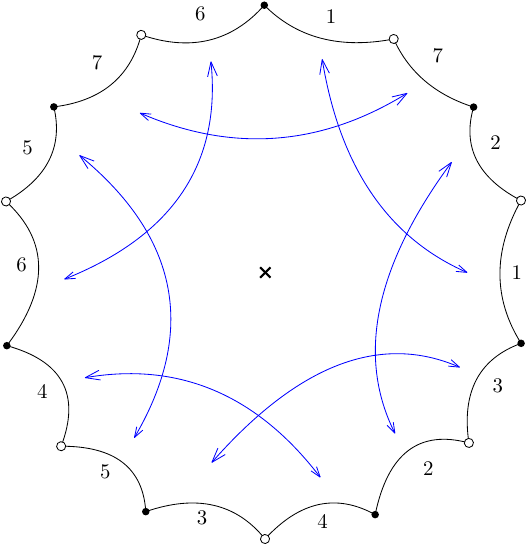}\quad \includegraphics[page=15,scale=.4]{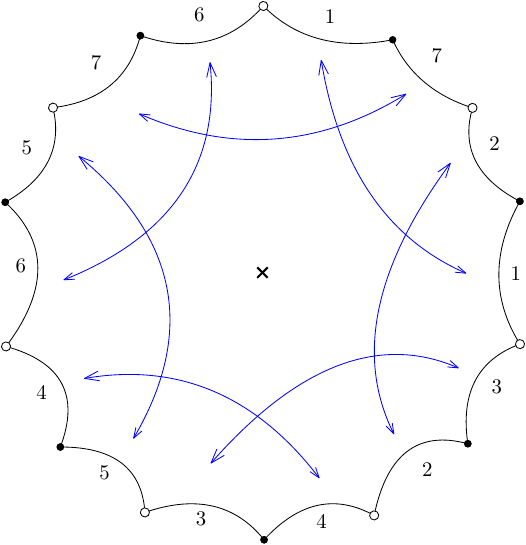}\quad \includegraphics[page=16,scale=.4]{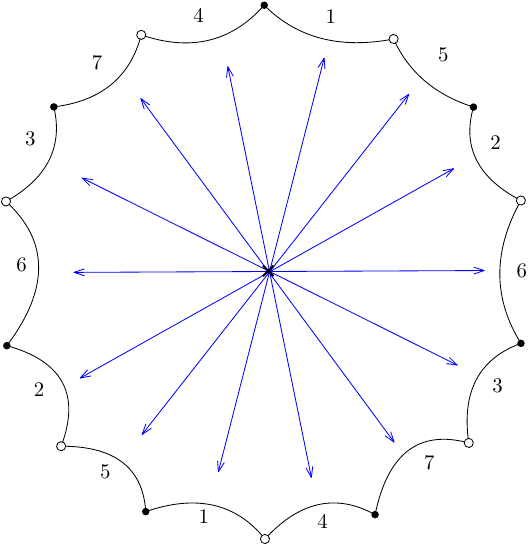}

\includegraphics[page=17,scale=.4]{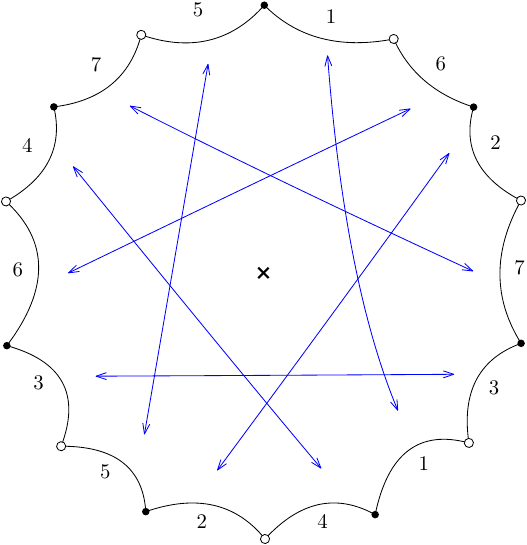}\quad \includegraphics[page=18,scale=.4]{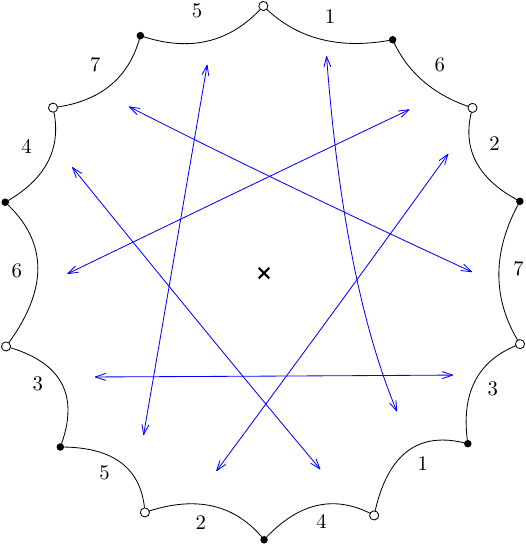}

\caption{The eight regular dessins with $C_7$ group of automorphisms. The three on the first row lie on $\widehat{\C}$, the three on the second row lie on the affine curve $y^2=x^8-x$, and the two on the third row lie on the Klein quartic curve $x^3y+y^3z+z^3x=0$ given in projective coordinates $[x:y:z]\in\mathbb{P}^2(\C)$.}\label{fig:dessins}

\end{figure}
\end{center}
To summarize:
\begin{itemize}
\item $QC(7)=2$ means that the two nonequivalent quasiplatonic topological $C_7$-actions given by the generating vectors $(3,3,1)$ and $(4,2,1)$ are the only ones.
\item $r(C_7)=8$ means there are eight regular dessins with a $C_7$ symmetry group. Five of these dessins correspond to the generating vectors
\begin{equation*}
(3,3,1), \ (3,1,3), \ (1,3,3), \ (4,2,1), \ (2,4,1).
\end{equation*}
Those five dessins lie on two distinct quasiplatonic surfaces, both of type $(7,7,7)$ and of genus three.
\end{itemize}

\end{example}

\begin{example}
Consider quasiplatonic topological $C_8$-actions on surfaces. Harvey's Theorem gives two admissible signatures for such $C_8$ actions: $(2,8,8)$ and $(4,8,8)$. By the Riemann-Hurwitz formula, the genus of the surface of type $(2,8,8)$ is
\begin{equation*}
1+\frac{8}{2}\left(1-\frac{1}{2}-\frac{1}{8}-\frac{1}{8}\right)=2.
\end{equation*}
The affine equation of this surface as a complex algebraic curve is given by $y^2=x^5-x$ (see Table 5.3 in \cite{jones}). The genus of the surfaces of type $(4,8,8)$ is
\begin{equation*}
1+\frac{8}{2}\left(1-\frac{1}{4}-\frac{1}{8}-\frac{1}{8}\right)=3.
\end{equation*}
It turns out that these surfaces of type $(4,8,8)$ are two distinct quasiplatonic surfaces of genus three with equations $y^2=x^8-1$ and $y^4+x^4=1$ (see again Table 5.3 in \cite{jones}). By Theorem \ref{thm:ben}, the number of quasiplatonic topological $C_8$-actions of type $(2,8,8)$ is
\begin{equation*}
T(2,8,8)=\frac{1}{2}\left(\tau_1(8,2)+\phi(2)\right)=\frac{1}{2}(1+1)=1,
\end{equation*}
whereas the number of actions of type $(4,8,8)$ is
\begin{equation*}
T(4,8,8)=\frac{1}{2}\left(\tau_1(8,4)+\phi(4)\right)=\frac{1}{2}(2+2)=2.
\end{equation*}
This verifies our main result of Theorem \ref{thm:qceven}:
\begin{equation*}
3=T(2,8,8)+T(4,8,8)=QC(8)=\frac{1}{6}\cdot 8\left(1+\frac{1}{2}\right)-1+2.
\end{equation*}
Now, there are $r(C_8)=8\left(1+\frac{1}{2}\right)=12$ regular dessins with a $C_8$ group of automorphisms. Disregarding the three on the Riemann sphere, nine of these dessins lie on the genus two or three surfaces mentioned above. See Figure \ref{fig:dessinsc8} for these nine regular dessins, obtained after the edge identifications designated by arrows. For each dessin in the figure, once vertices and faces are interchanged, one obtains two other inequivalent dessins of the same type with equivalent generating vectors.
\begin{center}
\begin{figure}[ht]
\includegraphics[page=19,scale=.43]{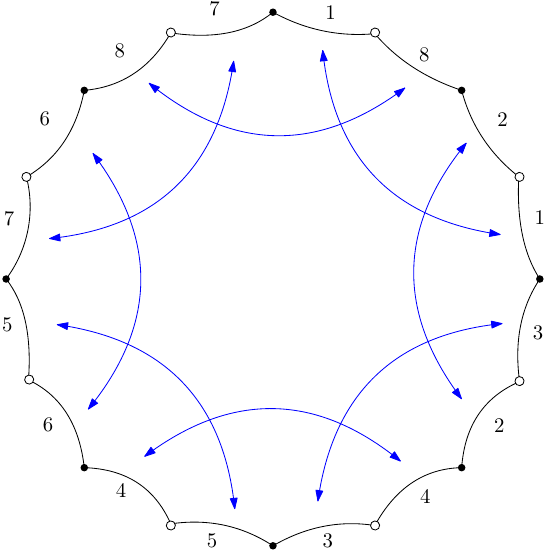}\quad \includegraphics[page=20,scale=.43]{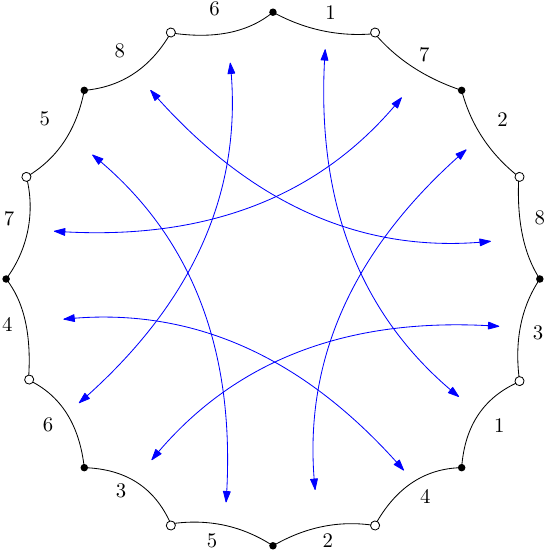}\quad  \includegraphics[page=21,scale=.43]{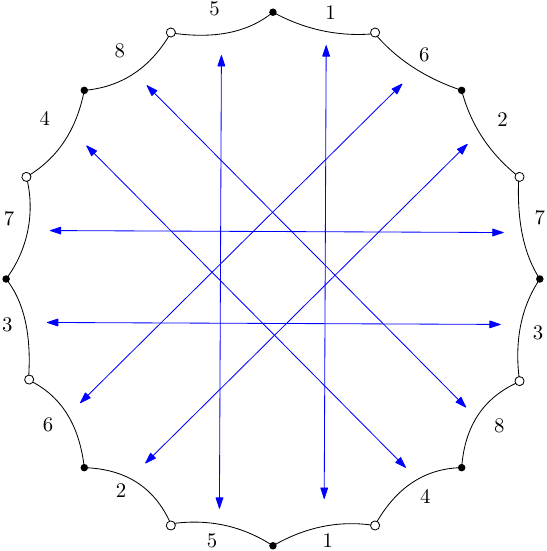}
\caption{By interchanging white and black vertices, white vertices and faces, or black vertices and faces, one can obtain all nine regular dessins with a $C_8$ automorphism group on higher genus surfaces.}\label{fig:dessinsc8}
\end{figure}
\end{center}
The associated generating vectors for the dessins in Figure \ref{fig:dessinsc8}, in order from left to right, are
\begin{align*}
(1,6,1), \ &(1,1,6) \ (6,1,1), \\
(2,5,1), \ &(1,2,5) \ (5,1,2), \\
(3,4,1), \ &(1,3,4) \ (4,1,3).
\end{align*}
Each row indeed defines the three equivalence classes of generating vectors, as predicted by $QC(8)=3$.
\end{example}

We include a table of the first 40 values of $QC(n)$ in Figure \ref{fig:table}. We have also included a graph of the first 2,000 values of $QC(n)$ in Figure \ref{fig:qcgraph}.

\begin{figure}[ht]
\begin{multicols}{3}

\begin{tabular}{|l|l|} \hline
$n$ & $QC(n)$  \\ \hline \hline
$1$ & $0$  \\ \hline
$2$ & $0$  \\ \hline
$3$ & $0$  \\ \hline
$4$ & $0$  \\ \hline
$5$ & $1$  \\ \hline
$6$ & $1$  \\ \hline
$7$ & $2$  \\ \hline
$8$ & $3$  \\ \hline
$9$ & $2$  \\ \hline
$10$ & $3$  \\ \hline
$11$ & $2$  \\ \hline
$12$ & $5$  \\ \hline
$13$ & $3$  \\ \hline
$14$ & $4$  \\ \hline
$15$ & $5$  \\ \hline
\end{tabular}

\begin{tabular}{|l|l|} \hline
$n$ & $QC(n)$  \\ \hline \hline
$16$ & $5$  \\ \hline
$17$ & $3$  \\ \hline
$18$ & $6$  \\ \hline
$19$ & $4$  \\ \hline
$20$ & $7$  \\ \hline
$21$ & $7$  \\ \hline
$22$ & $6$  \\ \hline
$23$ & $4$  \\ \hline
$24$ & $11$  \\ \hline
$25$ & $5$  \\ \hline
$26$ & $7$  \\ \hline
$27$ & $6$  \\ \hline
$28$ & $9$  \\ \hline
$29$ & $5$  \\ \hline
$30$ & $13$  \\ \hline
\end{tabular}

\begin{tabular}{|l|l|}\hline
$n$ & $QC(n)$  \\ \hline \hline
$31$ & $6$  \\ \hline
$32$ & $9$  \\ \hline
$33$ & $9$  \\ \hline
$34$ & $9$  \\ \hline
$35$ & $9$  \\ \hline
$36$ & $13$  \\ \hline
$37$ & $7$ \\ \hline
$38$ & $10$  \\ \hline
$39$ & $11$  \\ \hline
$40$ & $15$ \\ \hline
\end{tabular}

\end{multicols}
\caption{The first 40 values of $QC(n)$.}\label{fig:table}
\end{figure}

\begin{figure}[ht]
\includegraphics[scale=.7]{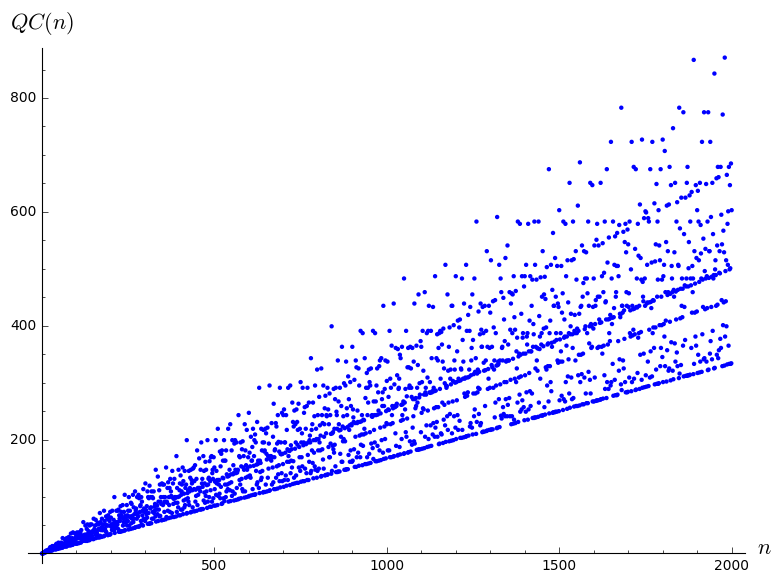}
\caption{The first 2,000 values of $QC(n)$. Sage code used in coding $QC(n)$ and making this figure is available upon request.}\label{fig:qcgraph}
\end{figure}

\subsection{Connection to K. Lloyd's Result}\label{subsec:lloyd}

K. Lloyd \cite{lloyd1972} provided a generating function for the number of distinct topological actions of the cyclic group $C_p$ for an odd prime $p$ on compact Riemann surfaces of arbitrary number of periods. We follow Lloyd's notation. Let $N(\Gamma_\rho,\Z_p)$ denote the number of topological actions of $\Z_p$ on a Riemann surface uniformized by a surface group normally contained in a Fuchsian group with signature $(0;p,\ldots,p)$ having $\rho$-number of periods with $p$.

\begin{theorem}[Lloyd, \cite{lloyd1972}]
\begin{align}
&\sum_{\rho\geq 3}N(\Gamma_\rho,\Z_p)x^\rho\nonumber\\
&=\frac{1}{p-1}\left(\frac{1}{p}\left(\frac{1}{(1-x)^{p-1}}+(p-1)\frac{1-x}{1-x^p}\right)+\sum_{\substack{\ell\ell'=p-1\\ \ell\neq 1}}\phi(\ell)(1-x^\ell)^{-\ell'}\right)\label{eq:lloyd}
\end{align}
\end{theorem}

In this paper, we have been interested in $QC(n)$, the number of group actions of $C_n$ on quasiplatonic surfaces with orbit genus zero and having three periods (the so-called quasiplatonic actions). This is the case of $\rho=3$ in \eqref{eq:lloyd}. We can show that the coefficient of $x^3$ in Lloyd's result is $QC(p)$. 

There are no quasiplatonic $C_3$-actions on quasiplatonic surfaces, because the only admissible action would be of the form $(n_1,n_2,n_3)$ with $n_i$ a factor of three, but none of these signatures satisfy $1/n_1+1/n_2+1/n_3<1$. Thus, we consider the only two cases for odd primes $p$: when $p\equiv 5\modd 6$ or $p\equiv 1\modd 6$. In either case, the only contribution to the $x^3$ term comes only from $1/(1-x)^{p-1}$ and $\sum_{\substack{\ell\ell'=p-1\\ \ell\neq 1}}\phi(\ell)(1-x^\ell)^{-\ell'}$ when $\ell=3$. The latter sum contributes to the $x^3$ term if and only if $\ell=3$, which happens if and only if $p-1$ is divisible by three. This only occurs when $p\equiv 1\modd 6$.

First assume $p\equiv 5\modd 6$. As stated in the prior paragraph, we only have contribution to $x^3$ coming from $1/((p-1)p(1-x)^{p-1})$. Because
\begin{equation*}
\frac{1}{1-x}=\sum_{k=0}^\infty x^k,
\end{equation*}
which converges absolutely for $|x|<1$, we can take derivatives to obtain
\begin{equation*}
\frac{1}{(1-x)^{p-1}}=\sum_{k=p-2}^\infty\binom{k}{p-2}x^{k-(p-2)}.
\end{equation*}
We seek the coefficient of $x^{k-(p-2)}$ when $k-(p-2)=3$, i.e., when $k=p+1$. The coefficient of $x^3$ is then given as
\begin{equation*}
\binom{p+1}{p-2}=\frac{(p+1)p(p-1)}{6}.
\end{equation*}
Thus, the coefficient of $x^3$ in \eqref{eq:lloyd} when $p\equiv 5\modd 6$ is
\begin{equation*}
\frac{1}{(p-1)p}\cdot\frac{(p+1)p(p-1)}{6}=\frac{p+1}{6}=QC(p).
\end{equation*}

Next, assume $p\equiv 1\modd 6$. The same calculation holds as in the previous case when $p\equiv 5\modd 6$, so we have a factor of $(p+1)/6$ in the coefficient of $x^3$. However, we must also consider the additional contribution from $\sum_{\substack{\ell\ell'=p-1\\ \ell\neq 1}}\phi(\ell)(1-x^\ell)^{-\ell'}$, which occurs when $\ell=3$. This means $\ell'=(p-1)/3$. The additional $x^3$ contribution thus arises in the term
\begin{equation*}
\frac{2}{p-1}\cdot\frac{1}{(1-x^3)^{(p-1)/3}}.
\end{equation*}
Starting from
\begin{equation*}
\frac{1}{1-x^3}=\sum_{k=0}^\infty x^{3k},
\end{equation*}
take derivatives to obtain the general formula
\begin{equation*}
\frac{1}{(1-x^3)^n}=\sum_{k=n-1}^\infty\frac{3k(3k-3)(3k-6)\cdots(3k-3(n-2))}{(n-1)!\cdot 3^{n-1}}x^{3k-3(n-1)}.
\end{equation*}
When $n=(p-1)/3$, the above equation becomes
\begin{equation*}
\frac{1}{(1-x^3)^{(p-1)/3}}=\sum_{k=(p-4)/3}^\infty\frac{3k(3k-3)\cdots (3k-p+7)}{\left(\frac{p-4}{3}\right)!\cdot 3^{(p-4)/3}}x^{3k-p+4}.
\end{equation*}
We are interested in the $x^3$ coefficient, that is, when $3k-p+4=3$, or when $k=(p-1)/3$. The coefficient is then
\begin{equation*}
\frac{(p-1)(p-4)(p-7)\cdots (p-(p-9))(p-(p-6))}{\left(\frac{p-1}{3}-1\right)\left(\frac{p-1}{3}-2\right)\cdots\left(\frac{p-1}{3}-\frac{p-7}{3}\right)\cdot 3^{(p-4)/3}}
\end{equation*}
Because $n-1=(p-4)/3$ and there are $n-1$ terms in the numerator, the above expression simplifies to
\begin{equation*}
\frac{(p-1)\cdot 3^{n-2}}{3^{n-1}}=\frac{p-1}{3}.
\end{equation*}
Recall the factor of $2/(p-1)$ in front of $1/(1-x^3)^{(p-1)/3}$. The additional $x^3$ contribution is thus
\begin{equation*}
\frac{2}{p-1}\cdot\frac{p-1}{3}=\frac{2}{3}.
\end{equation*}
Summing together the original calculation from the previous case, we see that the coefficient of $x^3$ in \eqref{eq:lloyd} when $p\equiv 1\modd 6$ is
\begin{equation*}
\frac{p+1}{6}+\frac{2}{3}=QC(p).
\end{equation*}

\section{Conclusions and Future Directions}\label{sec:conc}

We have computed the total number of distinct topological actions of $C_n$ on quasiplatonic surfaces. There are some natural extensions of this idea:
\begin{enumerate}
\item Suppose $G$ acts topologically on $X$ of genus $\sigma\geq 2$. Call $G$ a \emph{quasiplatonic group} if the orbit space $X/G$ is a compact surface of genus zero with three branch points. How many distinct topological $G$-actions are there for $G$ an arbitrary quasiplatonic group? Call this count $Q(G)$.
\item How does $Q(G)$ relate to $r(G)$, the number of regular dessins, up to isomorphism, with $\aut(D)\cong G$?
\item What is the growth rate of $QC(n)$? See the graph of $QC(n)$ in Figure \ref{fig:qcgraph}.
\item How does regular dessin equivalence translate to their generating vectors?
\end{enumerate}

\section{Acknowledgments}

The author would like to thank his advisor, Professor Ren Guo, for his guidance and support in this research, as well as Professor Aaron Wootton, Professor Allen Broughton, and Professor Steve Wilson for their insight and conversation. The author would also like to thank the faculty, staff, and students of the Department of Mathematics at Oregon State University for their sustained encouragement in furthering this research.

\end{document}